\documentclass[11pt]{article}

\usepackage[utf8]{inputenc}
\usepackage{authblk}
\usepackage[margin=1in]{geometry}

\usepackage{natbib}

\usepackage{mathrsfs,amsthm,xcolor,verbatim,bbm,amsmath,amsfonts,
amssymb,nicefrac,enumitem}
\usepackage{hyperref,bm,mathtools,xparse,etoolbox}
\hypersetup{hidelinks}
\usepackage[capitalise,sort]{cleveref} 

\usepackage{textcomp}
\usepackage{sidecap}
\usepackage{graphicx}
\usepackage{subfigure}
\usepackage{booktabs}
\usepackage{makecell}
\usepackage{multirow}

\crefname{enumi}{item}{items}

\crefname{figure}{Figure}{Figures}
\crefname{equation}{}{}
\crefname{subsection}{Subsection}{Subsections}

\theoremstyle{plain}
\newtheorem{theorem}{Theorem}[section]
\newtheorem{lemma}[theorem]{Lemma}
\newtheorem{proposition}[theorem]{Proposition}

\theoremstyle{definition}

\crefname{case}{Case}{Cases}
\crefname{cor}{Corollary}{Corollaries}
\usepackage{algorithm}
\usepackage{algpseudocode}

\usepackage{tikz}
\usetikzlibrary{arrows.meta,calc,shadows.blur}
\usetikzlibrary{matrix,chains,positioning,decorations.pathreplacing,arrows}
\usetikzlibrary{shapes,arrows}
\tikzset{
	font={\fontsize{9pt}{12}\selectfont}}

\pgfmathsetseed{7}

\newcommand{\drawblob}[7]{%
  \begin{scope}[rotate around={#5:(#1,#2)}]
    \fill[#6!12, opacity=0.55]
      (#1,#2) ellipse [x radius=1.05*#3, y radius=1.05*#4];
    \foreach \i in {1,...,#7}{
      \pgfmathsetmacro{\gx}{(rnd+rnd+rnd+rnd-2)/1.6}
      \pgfmathsetmacro{\gy}{(rnd+rnd+rnd+rnd-2)/1.6}
      \pgfmathsetmacro{\px}{#1 + \gx*#3}
      \pgfmathsetmacro{\py}{#2 + \gy*#4}
      \node[
        circle,
        fill=#6!70!black,
        draw=#6!85!black,
        very thin,
        inner sep=0pt,
        minimum size=3pt
      ] at (\px,\py) {};
    }
  \end{scope}
}

\ExplSyntaxOn

\bool_new:N \g_noteobserve

\NewDocumentCommand{\setnote}{}{
  \bool_gset_true:N \g_noteobserve
}

\NewDocumentCommand{\setobserve}{}{
  \bool_gset_false:N \g_noteobserve
}

\NewDocumentCommand{\nobs}{ o }{
  \IfValueT{#1}{
    \str_if_eq:noTF {note} {#1} {
      \bool_gset_true:N \g_noteobserve
    } {
      \str_if_eq:noTF {Note} {#1} {
        \bool_gset_true:N \g_noteobserve
      } {
        \bool_gset_false:N \g_noteobserve
      }
    }
  }
  \bool_if:nTF { \g_noteobserve } {
    \bool_gset_false:N \g_noteobserve
    note
  } {
    \bool_gset_true:N \g_noteobserve
    observe
  }
  \IfValueF{#1}{~}
}

\NewDocumentCommand{\Nobs}{ o }{
  \IfValueT{#1}{
    \str_if_eq:noTF {note} {#1} {
      \bool_gset_true:N \g_noteobserve
    } {
      \str_if_eq:noTF {Note} {#1} {
        \bool_gset_true:N \g_noteobserve
      } {
        \bool_gset_false:N \g_noteobserve
      }
    }
  }
  \bool_if:nTF { \g_noteobserve } {
    \bool_gset_false:N \g_noteobserve
    Note
  } {
    \bool_gset_true:N \g_noteobserve
    Observe
  }
  \IfValueF{#1}{~}
}

\int_new:N \g_furthermore

\NewDocumentCommand{\Moreover}{ o o }{
  \IfValueT{#1}{
    \str_case:nn {#1} {
      {Furthermore} {\int_set:Nn {\g_furthermore} {0}}
      {Moreover} {\int_set:Nn {\g_furthermore} {1}}
      {In~addition} {\int_set:Nn {\g_furthermore} {2}}
      {note} {\bool_gset_true:N \g_noteobserve}
      {observe} {\bool_gset_false:N \g_noteobserve}
    }
    \IfValueT{#2}{
      \str_case:nn {#2} {
        {Furthermore} {\int_set:Nn {\g_furthermore} {0}}
        {Moreover} {\int_set:Nn {\g_furthermore} {1}}
        {In~addition} {\int_set:Nn {\g_furthermore} {2}}
        {note} {\bool_gset_true:N \g_noteobserve}
        {observe} {\bool_gset_false:N \g_noteobserve}
      }
    }
  }
  \int_case:nn { \int_mod:nn {\g_furthermore} {3} } {
    { 0 } { Furthermore,~\nobs that}
    { 1 } { Moreover,~\nobs that}
    { 2 } { In~addition,~\nobs that}
  }
  \int_incr:N \g_furthermore
  \IfValueF{#1}{~}
}

\bool_new:N \g_hencetherefore

\NewDocumentCommand{\hence}{}{
  \bool_if:nTF { \g_hencetherefore } {
    \bool_gset_false:N \g_hencetherefore
    hence~
  } {
    \bool_gset_true:N \g_hencetherefore
    therefore~
  }
}

\NewDocumentCommand{\Hence}{}{
  \bool_if:nTF { \g_hencetherefore } {
    \bool_gset_false:N \g_hencetherefore
    Hence,~we~obtain~
  } {
    \bool_gset_true:N \g_hencetherefore
    Therefore,~we~obtain~
  }
}

\seq_new:N \g_cflist_loaded
\seq_new:N \g_cflist_pending

\NewDocumentCommand{\cfadd}{ m }
{
	\seq_if_in:NnF \g_cflist_loaded { #1 } {
		\seq_if_in:NnF \g_cflist_pending { #1 } {
			\seq_gput_right:Nn \g_cflist_pending { #1 }
		}
	}
}

\NewDocumentCommand{\cfconsiderloaded}{ m }{
	\seq_gput_right:Nn \g_cflist_loaded {#1}
}

\NewDocumentCommand{\cfremove}{ m }
{
	\seq_gremove_all:Nn \g_cflist_pending { #1 }
}

\NewDocumentCommand{\cfload}{ o }
{
	\seq_if_empty:NTF \g_cflist_pending {\unskip} {
		(cf.\ \cref{\seq_use:Nn \g_cflist_pending {,}})\IfValueTF{#1}{#1~}{\unskip}
		\seq_gconcat:NNN \g_cflist_loaded \g_cflist_loaded \g_cflist_pending
		\seq_gclear:N \g_cflist_pending
	}
}

\NewDocumentCommand{\cfclear} {} {
	\seq_gclear:N \g_cflist_loaded
	\seq_gclear:N \g_cflist_pending
}

\NewDocumentCommand{\cfout}{ o }
{
	\seq_if_empty:NTF \g_cflist_pending {\unskip} {
		(cf.\ \cref{\seq_use:Nn \g_cflist_pending {,}})\IfValueTF{#1}{#1~}{\unskip}
		\seq_gclear:N \g_cflist_pending
	}
}

\NewDocumentCommand{\ifnocf} { m } {
	\seq_if_empty:NT \g_cflist_pending { #1 }
}

\ExplSyntaxOff

\NewDocumentEnvironment{cproof}{m}
{\begin{proof}[Proof of \cref{#1}]}%
	{\noindent The proof of \cref{#1} is thus complete.
\end{proof}}

\NewDocumentEnvironment{cproof2}{m}
{\begin{proof}[Proof of \cref{#1}]}%
	{\noindent This completes the proof of \cref{#1}.
\end{proof}}

\title{CyclOT: Learning Quadratic Optimal Transport Maps via Synchronized Forward--Backward Interpolants}

\author[1]{
Shizhou Xu
	\thanks{Corresponding author: \texttt{shzxu@stanford.edu}.
This work was performed while the author was affiliated with
the Department of Mathematics, University of California, Davis.}
}
\author[2]{
Jiachen Liu
}

\author[3]{
Shih-Hsin Wang
	\thanks{
		{This work was performed while the author was affiliated with
the Department of Mathematics, University of Utah.}
	}
}

\author[4]{
Stefan Broecker
}

\author[5]{
Yuhao Huang
}

\author[5]{
Bao Wang
}

\author[2]{
Thomas Strohmer
}

\affil[1]{SLAC National Accelerator Laboratory, Stanford University\\
Menlo Park, CA 94025, USA}
\affil[2]{Department of Mathematics\\
University of California, Davis\\
Davis, CA 95616, USA}
\affil[3]{Department of Computer Science and Information Engineering\\
National Taiwan University\\
Taipei 10617, Taiwan}
\affil[4]{Department of Computer Science\\
University of California, Davis\\
Davis, CA 95616, USA}
\affil[5]{Department of Mathematics and
Scientific Computing and Imaging Institute\\
University of Utah\\
Salt Lake City, UT 84112, USA}

\date{}

\begin{document}
\maketitle

\begin{abstract}
We study the recovery of forward and reverse quadratic optimal-transport
maps from unpaired samples in high dimensions. We introduce a bidirectional
neural framework in which the learned maps induce forward and backward
displacement interpolants, while the training objective combines
bidirectional quadratic action, discriminator-restricted Jensen-Shannon
endpoint objectives, and two-sided cycle consistency. The construction
requires neither precomputed sample pairings nor an explicit
convex-potential parameterization. For absolutely continuous probability
measures supported on a compact convex set, and under the stated
generator-approximation, discriminator-richness, and minimizer-attainment
conditions, we prove a population recovery theorem: for every prescribed
accuracy, the sum of the corresponding \(L^2\) errors between any global
minimizer and the forward and reverse quadratic Brenier maps is below that
accuracy, provided the discriminator level is sufficiently large and the annealing
action weight becomes sufficiently small. Moreover, the cycle loss is bounded above by \(\lambda W_2^2(\mu_0,\mu_1)\). Complementary results quantify approximate invertibility and show that exact endpoint Jensen-Shannon divergence and cycle consistency control missing target mass and many-to-one map collapse, respectively. Experiments on Swiss roll, MNIST, CelebA, single-cell perturbation data, and chest X-ray images evaluate endpoint fidelity, transport cost, inverse consistency, and the geometry of the induced interpolations.
\end{abstract}



\newpage

\tableofcontents

\newpage

\section{Introduction}
\label{sec:intro}

Learning meaningful transformations between unpaired high-dimensional distributions is a fundamental challenge in modern machine learning. Such problems arise in image-to-image translation~\citep{zhu2017unpaired}, domain adaptation~\citep{courty2017optimal}, single-cell perturbation modeling~\citep{bunne2023learning}, and learning straight flow paths for flow-based generative models \cite{liu2022rectified}, where ground-truth correspondences are inherently unobserved and matching endpoint marginals alone is insufficient. The goal is instead to recover a structurally meaningful map that reflects the underlying geometry or dynamics. Optimal transport (OT) provides a principled framework for this task by selecting, among all maps matching the source and target distributions, one that minimizes the expected transport cost~\citep{villani2009optimal,santambrogio2015optimal,peyre2019computational}: under quadratic cost, the Brenier map is the gradient of a convex potential~\citep{brenier1991polar}, and the induced McCann displacement interpolation defines a Wasserstein geodesic between the two distributions~\citep{mccann1997convexity}. Equivalently, the Benamou--Brenier dynamic formulation characterizes this path as the minimum-kinetic-energy solution of a continuity equation~\citep{benamou2000computational}.

Despite this appealing characterization, learning OT maps remains difficult in high dimensions. Classical computational OT methods, including entropic Sinkhorn iterations~\citep{cuturi2013sinkhorn,peyre2019computational}, are primarily discrete, can be expensive at large sample sizes, and do not directly provide parametric maps for out-of-sample generalization. Recent flow-matching methods improve scalability by using prescribed couplings between source and target samples~\citep{liu2022rectified}. In particular, minibatch OT couplings are used to straighten flow trajectories and improve few-step generation~\citep{pooladian2023multisample,tong2023improving}. However, these couplings are batch-dependent approximations of the population OT plan and may introduce both computational overhead and statistical bias~\citep{fatras2019learning}.

Neural OT methods instead seek to learn transport maps or potentials directly from samples. Dual and minimax formulations, including Wasserstein adversarial objectives and neural OT methods~\citep{pmlr-v70-arjovsky17a,korotin2023neural}, are principled but can be sensitive to optimization and provide limited direct supervision of the intermediate transport path. Brenier-inspired approaches ~\citep{taghvaei20192, makkuva2020optimal,korotin2019wasserstein} represent the quadratic OT map as the gradient of a convex potential, often parameterized by input convex neural networks (ICNNs)~\citep{brenier1991polar,pmlr-v70-amos17b}. While this preserves the convex-potential structure required by Brenier's theorem, ICNN-based parameterizations can be restrictive and difficult to optimize compared with standard neural architectures. Optimal flow matching~\citep{kornilov2024ofm} is closely related in its goal of learning straight OT trajectories, but still belongs to the broader class of methods whose practical performance depends on surrogate coupling or potential-learning mechanisms. As a result, existing methods often trade scalability, architectural flexibility, optimization stability, and fidelity for the population OT solution.

Very recent work further highlights the importance of coupling design in flow-based generative modeling. Semi-discrete OT coupling methods replace minibatch OT with dataset-level semi-discrete formulations to improve scalability and trajectory straightness~\citep{mousavi2026semidiscrete,kong2026alignflow}, while hierarchical or dependent-coupling approaches modify the coupling structure to improve few-step or multistage generation~\citep{zhang2026hierarchical,ma2026stochastic}. These methods reinforce the importance of coupling structure, but they still construct couplings or interpolants through external pairing mechanisms. In contrast, our method, CyclOT, couples the induced intermediate interpolants through jointly learned forward and backward maps, with cycle consistency promoting approximate invertibility and coherence between the two transport directions, without requiring precomputed pairings or minibatch/semi-discrete OT subproblems during training.

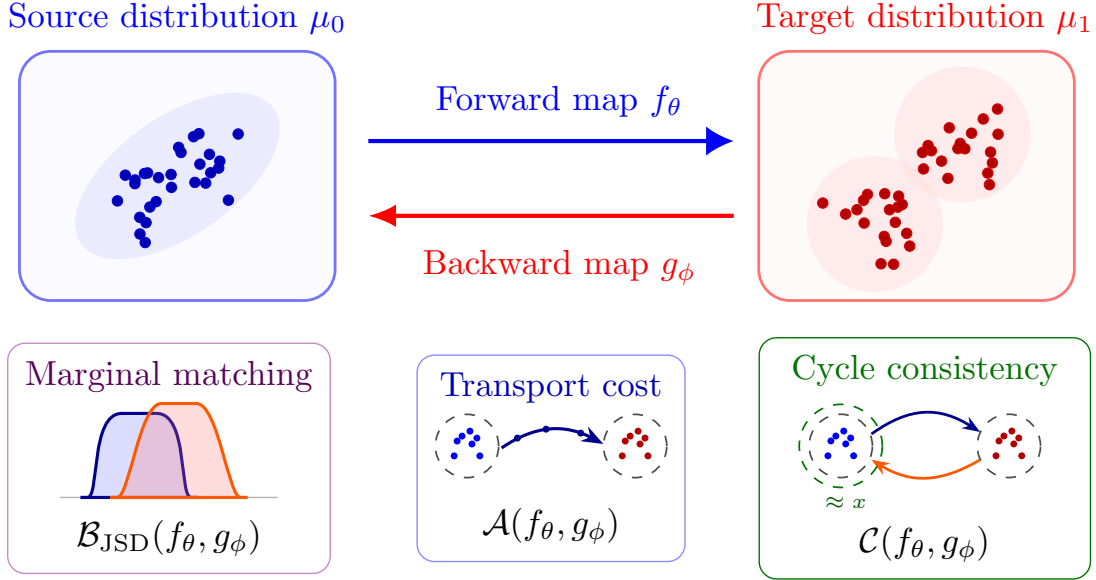
\begin{figure}[!ht]
\centering
\resizebox{0.9\textwidth}{!}{%
\begin{tikzpicture}[
    x=0.84cm,
    y=0.84cm,
    >=Latex,
    every node/.style={font=\sffamily},
    boxlabel/.style={font=\sffamily\bfseries},
    infobox/.style={
        rounded corners=5pt,
        inner sep=5pt,
        align=center,
        draw=black!25,
        thin,
        fill=white,
        fill opacity=0.95,
        text opacity=1
    }
]

\draw[
    rounded corners=9pt,
    draw=blue!55,
    line width=0.8pt,
    fill=blue!2
]
(-2.7,-1.5) rectangle (1.1,1.5);

\node[blue,font=\small] at (-0.8,1.9)
{Source distribution $\mu_0$};

\drawblob{-0.8}{0}{1.35}{0.68}{35}{blue}{28}


\draw[
    rounded corners=9pt,
    draw=red!55,
    line width=0.8pt,
    fill=red!2
]
(6.2,-1.5) rectangle (10,1.5);

\node[red,font=\small] at (8.2,1.9)
{Target distribution $\mu_1$};

\drawblob{7.6}{-0.58}{0.78}{0.78}{0}{red}{18}
\drawblob{8.65}{0.5}{0.78}{0.78}{0}{red}{18}


\draw[
    blue,
    very thick,
    -{Latex[length=2.8mm,width=3mm]}
]
(1.5,0.42) -- (5.9,0.42);

\node[blue,font=\small] at (3.8,0.9)
{Forward map $f_\theta$};

\draw[
    red,
    very thick,
    -{Latex[length=2.8mm,width=2.3mm]}
]
(5.9,-0.48) -- (1.5,-0.48);

\node[red,font=\small] at (3.8,-1.1)
{Backward map $g_\phi$};


\node[
    infobox,
    draw=violet!45,
    font=\small,
    minimum width=3.15cm
] at (-0.9,-3.4)
{
\textcolor{violet!70!black}{Marginal matching}\\[2pt]
\begin{tikzpicture}[baseline, x=1cm, y=1cm]
  \draw[gray!60, line width=0.4pt] (0.2,0) -- (2.4,0);
  \draw[draw=blue!55!black, fill=blue!55, fill opacity=0.25, line width=0.9pt]
    (0.3,0) .. controls (0.6,0) and (0.5,0.85) .. (1.0,0.85)
             .. controls (1.5,0.85) and (1.4,0) .. (1.7,0) -- cycle;
  \draw[draw=orange!70!red, fill=red!55, fill opacity=0.25, line width=0.9pt]
    (0.6,0) .. controls (0.9,0) and (1.0,0.95) .. (1.4,0.95)
             .. controls (1.8,0.95) and (1.9,0) .. (2.2,0) -- cycle;
\end{tikzpicture}\\[2pt]
$\mathcal{B}_{\mathrm{JSD}}(f_\theta,g_\phi)$
};


\node[
    infobox,
    draw=blue!45,
    font=\small,
    minimum width=2.45cm
] at (3.7,-3.4)
{
\textcolor{blue!70!black}{Transport cost}\\[2pt]
\begin{tikzpicture}[baseline, x=1cm, y=1cm]
  \draw[dashed, black!70, line width=0.5pt] (0.35,0.40) circle (0.32);
  \foreach \p in {(0.22,0.30),(0.30,0.50),(0.40,0.42),(0.48,0.30),
                  (0.25,0.45),(0.38,0.55),(0.45,0.48)}
    \fill[blue] \p circle (0.035);
 
  \draw[dashed, black!70, line width=0.5pt] (2.05,0.40) circle (0.32);
  \foreach \p in {(1.92,0.30),(2.00,0.50),(2.10,0.42),(2.18,0.30),
                  (1.95,0.45),(2.08,0.55),(2.15,0.48)}
    \fill[red!70!black] \p circle (0.035);
 
  \draw[blue!55!black, line width=0.9pt, -{Stealth[length=2.2mm]}]
    (0.70,0.38) .. controls (1.05,0.62) and (1.35,0.62) .. (1.73,0.42);
  \fill[blue!55!black] (0.86,0.48) circle (0.035);
  \fill[blue!55!black] (1.15,0.57) circle (0.035);
  \fill[blue!55!black] (1.50,0.53) circle (0.035);
\end{tikzpicture}\\[2pt]
$\mathcal{A}(f_\theta,g_\phi)$
};

\node[
    infobox,
    draw=green!45!black,
    font=\small,
    minimum width=3.35cm
] at (8.2,-3.4)
{
\textcolor{green!45!black}{Cycle consistency}\\[2pt]
\begin{tikzpicture}[baseline, x=1cm, y=1cm]
  \path[use as bounding box] (0.03,-0.30) rectangle (2.37,0.90);
 
  \draw[dashed, black!70, line width=0.5pt] (0.35,0.40) circle (0.32);
  \foreach \p in {(0.22,0.30),(0.30,0.50),(0.40,0.42),(0.48,0.30),
                  (0.25,0.45),(0.38,0.55),(0.45,0.48)}
    \fill[blue] \p circle (0.035);
 
  \draw[dashed, black!70, line width=0.5pt] (2.05,0.40) circle (0.32);
  \foreach \p in {(1.92,0.30),(2.00,0.50),(2.10,0.42),(2.18,0.30),
                  (1.95,0.45),(2.08,0.55),(2.15,0.48)}
    \fill[red!70!black] \p circle (0.035);
 
  \draw[blue!55!black, line width=0.8pt, -{Stealth[length=1.8mm]}]
    (0.66,0.52) .. controls (1.05,0.80) and (1.35,0.80) .. (1.76,0.54);
 
  \draw[orange!70!red, line width=0.8pt, -{Stealth[length=1.8mm]}]
    (1.76,0.26) .. controls (1.35,0.02) and (1.05,0.02) .. (0.66,0.26);
 
  \draw[green!45!black, dashed, line width=0.5pt] (0.35,0.40) circle (0.42);
  \node[green!35!black, font=\tiny] at (0.4,-0.18) {$\approx x$};
\end{tikzpicture}\\[2pt]
$\mathcal{C}(f_\theta,g_\phi)$
};

\end{tikzpicture}
}
\caption{Overview of our proposed \textbf{target-free OT learning via synchronized forward--backward interpolants.}
The forward map $f_\theta$ transports $\mu_0$ from left to right, while the backward map $g_\phi$ transports $\mu_1$ from right to left; Both $f_\theta$ and $g_\phi$ are neural network-approximated maps. The forward and backward induced marginals are matched at the two endpoints, providing marginal supervision for the transport maps. 
}
\label{fig:method_overview_minimal}
\end{figure}

\subsection{Our contributions}

We propose a new framework, \textbf{CyclOT}, for learning the quadratic OT map directly from unpaired samples, without iterative Reflow-style retraining~\citep{liu2022rectified}, minibatch or semi-discrete OT coupling solvers~\citep{pooladian2023multisample,tong2023improving,mousavi2026semidiscrete,kong2026alignflow}, or restrictive convex-potential parameterizations~\citep{taghvaei20192, makkuva2020optimal,korotin2019wasserstein, pmlr-v70-amos17b}. Our method learns two maps simultaneously: a forward transport map from source to target and a backward map from target to source. These maps induce forward and backward interpolants evolving from opposite endpoints, while cycle consistency couples the two transport directions and promotes coherent intermediate trajectories. The resulting objective can be approximated by two neural networks and optimized using stochastic optimization. At the theoretical level, our objective is motivated by an $L^2$ projection view of Wasserstein barycenters under independence constraints~\citep{xu2023fair}. This perspective suggests that enforcing independence-like intermediate agreement between map-induced interpolants can identify the quadratic OT structure without explicitly solving for a coupling.

Our main contributions are:
\begin{enumerate}[leftmargin=*]

    \item \textbf{A target-free learning framework for OT maps.} 
    We introduce synchronized forward--backward interpolants that impose structural consistency along the transport path without requiring paired samples, precomputed couplings, or externally specified trajectories. See Section~\ref{sec:proposed_framework} for details.

    \item \textbf{Scalable training with standard neural architectures.}  
    Our method, CyclOT, avoids iterative Reflow-style retraining~\citep{liu2022rectified}, minibatch OT approximations~\citep{pooladian2023multisample,tong2023improving}, and ICNNs~\citep{pmlr-v70-amos17b}, enabling efficient optimization with flexible expressive architectures such as ResNets or Transformers. See Section~\ref{subsec:population_objective} for details.

    \item \textbf{Provable recovery of the quadratic OT map.}  
    Under suitable assumptions, we show that minimizing our objective recovers the population quadratic OT map and its inverse. See Section~\ref{subsec:ot_recovery} for details.

    \item \textbf{Empirical validation across generative and scientific benchmarks.}  
    Across synthetic, image, and scientific datasets, our approach yields stable training, accurate one-step generation, and meaningful interpolations competitive with representative OT-based and flow-based generative methods. See Section~\ref{sec:experiments} for details.

\end{enumerate}

\subsection{Related work}
\label{sec:related}

\paragraph{Neural Optimal Transport and Direct Map Learning.}
Classical OT provides a rigorous variational formulation for matching probability measures~\citep{villani2009optimal,santambrogio2015optimal,peyre2019computational}, and the quadratic-cost case admits the Brenier-map structure~\citep{brenier1991polar}. However, computing OT maps from samples in high dimension remains challenging. Entropic solvers such as Sinkhorn iterations~\citep{cuturi2013sinkhorn} are effective for discrete OT but do not directly yield parametric out-of-sample maps. Neural OT methods address this by parameterizing transport maps, costs, or Kantorovich potentials with neural networks~\citep{korotin2023neural}. A prominent Brenier-inspired line represents maps as gradients of convex potentials, often implemented using ICNNs~\citep{taghvaei20192, makkuva2020optimal,korotin2019wasserstein, pmlr-v70-amos17b}. While theoretically appealing, such convex-potential parameterizations can restrict architectural flexibility and complicate optimization. CyclOT instead learns forward and backward maps using standard neural network architectures and enforces OT structure through endpoint marginal matching, transport-cost minimization, and cycle consistency.

\paragraph{Flow Matching and Coupling Design.}
Flow matching trains continuous normalizing flows by regressing velocity fields along prescribed probability paths~\citep{lipman2022flow}. The geometry of these paths depends critically on the coupling between source and target samples. Rectified flow and Reflow aim to learn straighter trajectories, but often require repeated simulation or retraining stages~\citep{liu2022rectified}. Minibatch OT and multisample couplings improve trajectory straightness by matching source and target samples within each batch~\citep{pooladian2023multisample,tong2023improving}, but introduce batch-dependent approximations and additional OT computation. Recent semi-discrete and structured coupling methods further improve scalability or few-step generation by replacing batch OT with semi-discrete or hierarchical coupling mechanisms~\citep{mousavi2026semidiscrete,kong2026alignflow,zhang2026hierarchical,ma2026stochastic}. In contrast, CyclOT does not require precomputed couplings, minibatch OT, or semi-discrete pairing.

\paragraph{Single-Step Generative Models.}
Accelerating generative sampling to one or a few steps is a major goal in diffusion and flow-based modeling. Consistency models and consistency trajectory models learn self-consistent maps along probability-flow trajectories~\citep{song2023consistency,kim2023consistency}, while progressive distillation and shortcut models reduce the number of sampling steps by distilling or shortcutting iterative samplers~\citep{salimans2022progressive,frans2024one}. Recent mean-flow and flow-map-matching approaches further seek direct map-based generation~\citep{geng2025mean,boffi2025flow}. These approaches primarily focus on sampling acceleration, often through distillation, consistency constraints, or specialized schedules. Our goal is different: we seek to recover the quadratic OT map and its associated displacement interpolation from unpaired samples, thereby obtaining one-step generation as a consequence of OT map learning.

\paragraph{Unpaired Translation and Cycle Consistency.}
Cycle-consistent adversarial methods learn mappings between unpaired domains by combining adversarial distribution matching with inverse-consistency constraints~\citep{zhu2017unpaired}. These methods are highly influential in image translation, but they do not in general identify the quadratic OT map or the Wasserstein geodesic between distributions. Our use of cycle consistency is instead tied to OT map recovery: the forward and backward maps define complementary displacement interpolants, while cycle consistency serves as a soft regularization aims to promote an approximately inverse relationship between the two maps and thereby imposes structural coherence on the induced transport paths.

\subsection{Organization}
Our paper is organized as follows.
Section~\ref{sec:problem_setup} is concerned with a review of the quadratic optimal transport formulation and the associated McCann displacement interpolation, which provide the foundation for our forward–backward transport framework.
In Section~\ref{sec:proposed_framework}, we propose our bidirectional map-based optimal transport framework, CyclOT, and introduce the corresponding objective function in the theoretical framework and its neural estimations that we use in the practical setting.
In Section~\ref{sec:global_recovery}, we further prove that, under suitable uniform approximation assumptions on the generator and adversarial network classes, minimizers of the proposed objective converge to the quadratic optimal transport maps.
We provide a local analysis of the proposed framework and further investigate the role of each term in the objective in Section~\ref{sec:theory}.  In particular, we study how cycle consistency restricts the feasible class of transport maps, how the action and cycle terms interact near inverse-compatible solutions, and how endpoint matching and cycle consistency suppress distinct forms of mode collapse.
 Section~\ref{sec:experiments} is devoted to evaluating the CyclOT on five datasets: \emph{Synthetic Swiss Roll}, \emph{MNIST}, \emph{CelebA}, \emph{single-cell data}, and \emph{chest X-ray images}. We compare CyclOT against five representative baselines: entropic OT (Sinkhorn)~\citep{cuturi2013sinkhorn}, Neural OT~\citep{korotin2023neural}, ICNN-based OT~\citep{makkuva2020optimal}, rectified flow~\citep{liu2022rectified}, and mini-batch OT-guided flow matching~\citep{tong2023improving}. Our experiments demonstrate that the proposed framework enables scalable and stable learning of optimal transport maps.

\section{Problem Setup \& Notation}
\label{sec:problem_setup}

\subsection{Quadratic optimal transport}
\label{subsec:quadratic_ot}

For a detailed introduction to optimal transport we recommend~\cite{villani2009optimal}.
Let \(\mathcal P_2(\mathbb R^d)\) denote the set of Borel probability
measures on \(\mathbb R^d\) with finite second moment. For
\(\mu,\nu\in\mathcal P_2(\mathbb R^d)\), let
\(\Pi(\mu,\nu)\) denote the set of their couplings. The quadratic
Wasserstein distance is defined by
\[
W_2^2(\mu,\nu)
:=
\inf_{\pi\in\Pi(\mu,\nu)}
\int_{\mathbb R^d\times\mathbb R^d}
\|x-y\|^2\,d\pi(x,y).
\]
Throughout the paper, \(\|\cdot\|\) denotes the Euclidean norm and
\(\operatorname{Id}\) denotes the identity map. For a measurable map \(T:\mathbb R^d\to\mathbb R^d\), the pushforward
\(T_\#\mu\) is defined by
\[
(T_\#\mu)(A)
=
\mu\bigl(T^{-1}(A)\bigr)
\]
for every Borel set \(A\subset\mathbb R^d\). Equivalently,
\(T_\#\mu=\nu\) means that \(T(X)\sim\nu\) whenever \(X\sim\mu\).

Given \(\mu_0,\mu_1\in\mathcal P_2(\mathbb R^d)\), the Monge problem
for the quadratic cost seeks a measurable map transporting \(\mu_0\)
to \(\mu_1\) while minimizing
\[
C_{\mu_0}(T)
:=
\frac12
\int_{\mathbb R^d}
\|T(x)-x\|^2\,d\mu_0(x).
\]
Thus, the quadratic Monge problem is
\begin{equation}
\inf_{\substack{
T:\mathbb R^d\to\mathbb R^d\ \mathrm{measurable}\\
T_\#\mu_0=\mu_1
}}
C_{\mu_0}(T).
\label{eq:quadratic_monge_problem}
\end{equation}
Note that the factor \(1/2\) is included in the Monge cost but not in the
definition of \(W_2^2\). If \(\mu_0\) is absolutely continuous with respect to Lebesgue measure, Brenier's theorem gives a unique minimizer \(T^\star\), up to \(\mu_0\)-almost-everywhere equality, of the form
\[
T^\star=\nabla\phi
\qquad
\mu_0\text{-a.e.},
\]
where the potential function \(\phi:\mathbb R^d\to(-\infty,+\infty]\) is convex
\citep{brenier1991polar,villani2009optimal}. In particular,
\[
T^\star_\#\mu_0=\mu_1,
\qquad
C_{\mu_0}(T^\star)
=
\frac12W_2^2(\mu_0,\mu_1).
\]
When \(\mu_1\) is also absolutely continuous, the reverse quadratic
Monge problem admits a unique Brenier map \(S^\star\), up to
\(\mu_1\)-almost-everywhere equality, satisfying
\[
S^\star_\#\mu_1=\mu_0.
\]
Suitable representatives of the two maps satisfy
\[
S^\star\circ T^\star
=
\operatorname{Id}
\quad\mu_0\text{-a.e.},
\qquad
T^\star\circ S^\star
=
\operatorname{Id}
\quad\mu_1\text{-a.e.},
\]
where $\operatorname{Id}$ denotes the identity map on the corresponding space.

Moreover,
\[
C_{\mu_1}(S^\star)
=
\frac12W_2^2(\mu_0,\mu_1),
\]
and hence
\begin{equation}
C_{\mu_0}(T^\star)+C_{\mu_1}(S^\star)
=
W_2^2(\mu_0,\mu_1).
\label{eq:bidirectional_brenier_action}
\end{equation}
We refer to \((T^\star,S^\star)\) as the forward--reverse Brenier pair.

For the population recovery results, we work under the following standing
setting. Let \(\mathcal X\subset\mathbb R^d\) be a nonempty compact convex
set with nonempty interior, and assume that
\[
\mu_0,\mu_1\in\mathcal P_{2,\mathrm{ac}}(\mathcal X),
\]
where
\[
\mathcal P_{2,\mathrm{ac}}(\mathcal X)
:=
\left\{
\mu\in\mathcal P_2(\mathbb R^d):
\mu(\mathcal X)=1,\ 
\mu\ll\mathcal L^d
\right\}.
\]
Here \(\mathcal L^d\) denotes the \(d\)-dimensional Lebesgue measure.

\subsection{Displacement interpolation}
\label{subsec:displacement_interpolation}

The forward Brenier map induces the McCann displacement interpolation
\citep{mccann1997convexity}
\begin{equation}
F_t^\star(x)
:=
(1-t)x+tT^\star(x),
\qquad
\mu_t^\star
:=
(F_t^\star)_\#\mu_0,
\qquad
t\in[0,1].
\label{eq:maccann_interp}
\end{equation}
Convexity of \(\mathcal X\) ensures that
\(F_t^\star(x)\in\mathcal X\) whenever
\(x,T^\star(x)\in\mathcal X\). The curve
\((\mu_t^\star)_{t\in[0,1]}\) is the constant-speed quadratic
Wasserstein geodesic from \(\mu_0\) to \(\mu_1\) satisfying
\[
W_2(\mu_s^\star,\mu_t^\star)
=
|t-s|W_2(\mu_0,\mu_1),
\qquad
s,t\in[0,1].
\]

The same interpolation can be represented using the reverse Brenier map.
Define
\[
G_t^\star(y)
:=
(1-t)S^\star(y)+ty,
\qquad
t\in[0,1].
\]
The almost-everywhere inverse relations imply
\begin{equation}
(G_t^\star)_\#\mu_1
=
(F_t^\star)_\#\mu_0
=
\mu_t^\star,
\qquad
t\in[0,1].
\label{eq:forward_reverse_brenier_interpolation}
\end{equation}
Indeed, if \(y=T^\star(x)\), then
\[
G_t^\star(y)
=
(1-t)S^\star(T^\star(x))+tT^\star(x)
=
(1-t)x+tT^\star(x)
=
F_t^\star(x)
\]
for \(\mu_0\)-almost every \(x\). Thus, although \(G_t^\star\) is
constructed from the reverse map, it parameterizes the same path in the
time direction from \(\mu_0\) to \(\mu_1\).

Each particle in \eqref{eq:maccann_interp} travels along a straight line
with constant velocity \(T^\star(x)-x\). Consequently,
\[
\int_0^1
\mathbb E_{X\sim\mu_0}
\left[
\frac12
\left|
\partial_tF_t^\star(X)
\right|^2
\right]dt
=
\frac12
\mathbb E_{X\sim\mu_0}
\left[
|T^\star(X)-X|^2
\right]\
=
\frac12W_2^2(\mu_0,\mu_1).
\]
This identity motivates the quadratic-action term in the proposed
objective, while \eqref{eq:forward_reverse_brenier_interpolation}
provides the ideal synchronization property that the learned
forward--backward construction seeks to recover.

\subsection{Bidirectional learning problem}
\label{subsec:bidirectional_learning_problem}

We observe independent, unpaired samples from \(\mu_0\) and \(\mu_1\).
Our goal is to learn a forward map \(f\) and a backward map \(g\) that
approximate \(T^\star\) and \(S^\star\), respectively. In particular, we
seek endpoint marginal agreement,
\[
f_\#\mu_0\approx\mu_1,
\qquad
g_\#\mu_1\approx\mu_0,
\]
together with approximate inverse consistency,
\[
g\circ f\approx\operatorname{Id}
\quad\text{on }\mu_0,
\qquad
f\circ g\approx\operatorname{Id}
\quad\text{on }\mu_1.
\]
Marginal matching and inverse consistency alone do not identify the
Brenier pair: many invertible maps can transport one distribution to the
other. The quadratic transport cost supplies the additional variational
criterion that selects the minimum-cost pair. CyclOT therefore combines endpoint matching, two-sided cycle consistency, and
bidirectional quadratic action.

For a measurable vector-valued function \(h\), we use the notation
\[
\|h\|_{L^2(\mu)}
:=
\left(
\int_{\mathbb R^d}\|h(x)\|^2\,d\mu(x)
\right)^{1/2}.
\]
Accordingly, population recovery of the Brenier pair will be measured by
\[
\|f-T^\star\|_{L^2(\mu_0)}
+
\|g-S^\star\|_{L^2(\mu_1)}.
\]
Section~\ref{sec:proposed_framework} turns the three structural requirements
above into a neural population objective.

Additional background on flow matching, rectified flow, and
independence-constrained Wasserstein barycenters is collected in
Appendix~\ref{app:background}.

\section{Proposed Framework}
\label{sec:proposed_framework}

We introduce a bidirectional map-based framework for learning the
quadratic optimal transport between \(\mu_0\) and \(\mu_1\) from
independent endpoint samples. The framework consists of three neural approximation: a neural forward map, a neural reverse map, and neural endpoint-matching penalties, which together formulate the three training objectives: action, marginal matching, and cycle consistency. We first define the maps and their induced interpolants, then introduce the population
objective and its practical minibatch implementation.

\subsection{Forward-backward map framework}
\label{subsec:motivation_setting}

Let
\[
\mathcal F=\{f_\theta:\theta\in\Theta\},
\qquad
\mathcal G=\{g_\phi:\phi\in\Phi\}
\]
be admissible classes of maps from \(\mathcal X\) to \(\mathcal X\).
We interpret \(f_\theta\) as a forward map from \(\mu_0\) toward
\(\mu_1\), and \(g_\phi\) as a reverse map from \(\mu_1\) toward
\(\mu_0\).

For \(f\in\mathcal F\) and \(g\in\mathcal G\), define the induced
displacement interpolants
\[
F_t^f(x)
:=
(1-t)x+t f(x),
\qquad
\mu_t^f
:=
(F_t^f)_\#\mu_0,
\]
and
\[
G_t^g(y)
:=
(1-t)g(y)+t y,
\qquad
\nu_t^g
:=
(G_t^g)_\#\mu_1,
\qquad
t\in[0,1].
\]
Their endpoint laws are
\[
\mu_0^f=\mu_0,
\qquad
\mu_1^f=f_\#\mu_0,
\]
and
\[
\nu_0^g=g_\#\mu_1,
\qquad
\nu_1^g=\mu_1.
\]
Both paths are therefore oriented from \(\mu_0\) to \(\mu_1\), although
they are generated from opposite endpoints. Convexity of
\(\mathcal X\) ensures that the interpolated points remain in
\(\mathcal X\) whenever \(f\) and \(g\) are \(\mathcal X\)-valued.

The two intermediate laws need not agree for arbitrary \(f\) and \(g\).
However, if
\[
f_\#\mu_0=\mu_1,
\qquad
g\circ f=\operatorname{Id}
\quad\mu_0\text{-a.e.},
\]
then
\[
G_t^g(f(x))
=
(1-t)g(f(x))+tf(x)
=
(1-t)x+tf(x)
=
F_t^f(x)
\]
for \(\mu_0\)-almost every \(x\). Consequently,
\[
\mu_t^f=\nu_t^g
\qquad
\text{for every }t\in[0,1].
\]
Thus, the framework does not directly penalize discrepancies between
intermediate marginals. Their synchronization follows from endpoint
matching and forward-backward inverse consistency.

\subsection{Neural population objective}
\label{subsec:population_objective}

We enforce the endpoint constraints using neural-approximation Jensen Shannon Divergence (neural JSD) objectives. Let \(P,Q\in\mathcal P(\mathcal X)\), where \(P\) is a generated distribution and \(Q\) is its target distribution. For a target-positive discriminator
\(D:\mathcal X\to(0,1)\), define
\[
\Phi_{P,Q}(D)
:=
\log2
+
\frac12
\left[
\int_{\mathcal X}\log(1-D)\,dP
+
\int_{\mathcal X}\log D\,dQ
\right].
\]
The unrestricted supremum of \(\Phi_{P,Q}\) is the
Jensen-Shannon divergence:
\[
\operatorname{JSD}(P,Q)
=
\sup_{\substack{
D:\mathcal X\to(0,1)\\
D\ {\rm measurable}
}}
\Phi_{P,Q}(D).
\]
Let
$\mathcal D_m\subset C(\mathcal X;(0,1))$
be a neural network class of capacity level \(m\), and assume
that \(D\equiv1/2\in\mathcal D_m\). Define the neural JSD with capacity level $m$:
\begin{equation}
\widehat{\operatorname{JSD}}_m(P,Q)
:=
\sup_{D\in\mathcal D_m}\Phi_{P,Q}(D).
\label{eq:neural_jsd_def}
\end{equation}
Note that, because $\Phi_{P,Q}\left(\frac12\right)=0$, we have
\[
0
\leq
\widehat{\operatorname{JSD}}_m(P,Q)
\leq
\operatorname{JSD}(P,Q).
\]

The bidirectional neural endpoint penalty is
\begin{equation}
\widehat{\mathcal B}_m(f,g)
:=
\widehat{\operatorname{JSD}}_m(f_\#\mu_0,\mu_1)
+
\widehat{\operatorname{JSD}}_m(g_\#\mu_1,\mu_0).
\label{eq:bidirectional_neural_jsd}
\end{equation}
The two terms in \eqref{eq:bidirectional_neural_jsd} are implemented
using separate forward and reverse neural JSD with the same capacity level $m$.

Furthermore, we define the bidirectional quadratic action by
\begin{equation}
\begin{aligned}
\mathcal A(f,g)
:=
\frac12
\int_{\mathcal X}
\|f(x)-x\|^2\,d\mu_0(x)
+
\frac12
\int_{\mathcal X}
\|g(y)-y\|^2\,d\mu_1(y),
\end{aligned}
\label{eq:bidirectional_action}
\end{equation}
and the two-sided cycle loss by
\begin{equation}
\begin{aligned}
\mathcal C(f,g)
:=
\int_{\mathcal X}
\|g(f(x))-x\|^2\,d\mu_0(x)
+
\int_{\mathcal X}
\|f(g(y))-y\|^2\,d\mu_1(y).
\end{aligned}
\label{eq:bidirectional_cycle}
\end{equation}

For an action coefficient \(\lambda>0\), the neural population objective is
\begin{equation}
\mathcal E_{\lambda,m}(f,g)
:=
\lambda\mathcal A(f,g)
+
\widehat{\mathcal B}_m(f,g)
+
\mathcal C(f,g).
\label{eq:population_objective}
\end{equation}
Here, the three terms have complementary roles:
\begin{itemize}[leftmargin=*]

\item
The endpoint term \(\widehat{\mathcal B}_m\) promotes the forward and
reverse marginal constraints via the neural JSD family with capacity level $m$
\[
f_\#\mu_0\approx\mu_1,
\qquad
g_\#\mu_1\approx\mu_0.
\]

\item
The cycle term \(\mathcal C\) explicitly promotes approximate inverse consistency in both directions for more stable and efficient learning by avoid mode collapses.

\item
The action term \(\mathcal A\) favors maps with small quadratic
displacement and supplies the optimal-transport selection criterion
among approximately endpoint-feasible inverse pairs.

\end{itemize}

The action coefficient \(\lambda\) controls the action term relative to endpoint matching and cycle consistency. Small positive values allow the
feasibility terms to dominate while retaining the quadratic-cost
criterion.

\subsection{Practical optimization algorithm}
\label{subsec:training_procedure}

In practice, we choose \(f_\theta\) and \(g_\phi\) to be represented by neural
networks. We use two target-positive discriminators:
\(D^1_{\omega_1}\) distinguishes samples from \(\mu_1\) from forward
generated samples, while \(D^0_{\omega_0}\) distinguishes samples from
\(\mu_0\) from reverse generated samples.

At iteration \(k\), draw independent minibatches
\[
\{x_0^i\}_{i=1}^B
\overset{\mathrm{i.i.d.}}{\sim}
\mu_0,
\qquad
\{x_1^i\}_{i=1}^B
\overset{\mathrm{i.i.d.}}{\sim}
\mu_1.
\]
The empirical forward and reverse discriminator scores are
\begin{equation}
\begin{aligned}
\widehat{\mathcal V}^{\,1}_B(\theta,\omega_1)
&:=
\frac{1}{2B}
\sum_{i=1}^B
\left[
\log D^1_{\omega_1}(x_1^i)
+
\log\left(
1-D^1_{\omega_1}(f_\theta(x_0^i))
\right)
\right],\\
\widehat{\mathcal V}^{\,0}_B(\phi,\omega_0)
&:=
\frac{1}{2B}
\sum_{i=1}^B
\left[
\log D^0_{\omega_0}(x_0^i)
+
\log\left(
1-D^0_{\omega_0}(g_\phi(x_1^i))
\right)
\right].
\end{aligned}
\label{eq:empirical_discriminator_scores}
\end{equation}
The discriminator parameters are updated by maximizing these scores,
or equivalently by minimizing $\widehat{\mathcal L}_{D^1} = -\widehat{\mathcal V}^{\,1}_B$ and $\widehat{\mathcal L}_{D^0} = -\widehat{\mathcal V}^{\,0}_B$. 

For the map update, we use the non-saturating adversarial loss
\begin{equation}
\widehat{\mathcal L}_{\mathrm{adv}}^{\mathrm{NS}}
:=
-\frac{1}{2B}
\sum_{i=1}^B
\left[
\log D^1_{\omega_1}(f_\theta(x_0^i))
+
\log D^0_{\omega_0}(g_\phi(x_1^i))
\right].
\label{eq:non_saturating_adversarial_loss}
\end{equation}
This is a practical surrogate for the minimax neural-JSD map update and
has the same intended marginal-matching equilibrium in the idealized
expressive-discriminator setting.

The empirical action and cycle losses are
\begin{equation}
\begin{aligned}
\widehat{\mathcal A}_B(\theta,\phi)
:=
\frac{1}{2B}
\sum_{i=1}^B
\Bigl[
\|f_\theta(x_0^i)-x_0^i\|^2
+
\|g_\phi(x_1^i)-x_1^i\|^2
\Bigr]
\end{aligned}
\label{eq:empirical_bidirectional_action}
\end{equation}
and
\begin{equation}
\begin{aligned}
\widehat{\mathcal C}_B(\theta,\phi)
:=
\frac{1}{B}
\sum_{i=1}^B
\Bigl[
\|g_\phi(f_\theta(x_0^i))-x_0^i\|^2
+
\|f_\theta(g_\phi(x_1^i))-x_1^i\|^2
\Bigr].
\end{aligned}
\label{eq:empirical_cycle_loss}
\end{equation}
The factor \(1/2\) appears in
\(\widehat{\mathcal A}_B\), consistently with
\eqref{eq:bidirectional_action}, but not in
\(\widehat{\mathcal C}_B\), because
\eqref{eq:bidirectional_cycle} is the sum of the two directional cycle
errors.

We use the action-weight schedule
\begin{equation}
\lambda_k
=
\lambda_{\mathrm{init}}\beta^k,
\qquad
\lambda_{\mathrm{init}}>0,
\qquad
0<\beta<1,
\qquad
k=0,\ldots,K-1.
\label{eq:transport_weight_schedule}
\end{equation}
The map parameters are updated by minimizing
\begin{equation}
\widehat{\mathcal L}_{\mathrm{map},k}
=
\widehat{\mathcal L}_{\mathrm{adv}}^{\mathrm{NS}}
+
\widehat{\mathcal C}_B
+
\lambda_k\widehat{\mathcal A}_B.
\label{eq:practical_map_loss}
\end{equation}

The complete procedure, which alternates between discriminator ascent and
map descent, is summarized in Algorithm~\ref{alg:sfb_ot}.

\begin{algorithm}[t]
\caption{Forward-Backward Neural Optimal Transport}
\label{alg:sfb_ot}
\begin{algorithmic}[1]

\Require
Endpoint distributions \(\mu_0,\mu_1\);
generators \(f_\theta,g_\phi\);
discriminators \(D^1_{\omega_1},D^0_{\omega_0}\);
batch size \(B\);
number of iterations \(K\);
initial action weight \(\lambda_{\mathrm{init}}>0\);
annealing factor \(\beta\in(0,1)\).

\For{\(k=0,\ldots,K-1\)}

    \State Draw independent minibatches
    \[
    \{x_0^i\}_{i=1}^B\sim\mu_0,
    \qquad
    \{x_1^i\}_{i=1}^B\sim\mu_1.
    \]

    \State Compute
    \(\widehat{\mathcal V}^{\,1}_B\) and
    \(\widehat{\mathcal V}^{\,0}_B\) using
    \eqref{eq:empirical_discriminator_scores}.

    \State Update \(\omega_1,\omega_0\) by discriminator ascent:
    \[
    \omega_1
    \leftarrow
    \omega_1
    +
    \eta_D\nabla_{\omega_1}
    \widehat{\mathcal V}^{\,1}_B,
    \qquad
    \omega_0
    \leftarrow
    \omega_0
    +
    \eta_D\nabla_{\omega_0}
    \widehat{\mathcal V}^{\,0}_B.
    \]

    \State Set
    \[
    \lambda_k
    =
    \lambda_{\mathrm{init}}\beta^k.
    \]

    \State Compute
    \(\widehat{\mathcal L}_{\mathrm{adv}}^{\mathrm{NS}}\),
    \(\widehat{\mathcal A}_B\), and
    \(\widehat{\mathcal C}_B\).

    \State Form
    \[
    \widehat{\mathcal L}_{\mathrm{map},k}
    =
    \widehat{\mathcal L}_{\mathrm{adv}}^{\mathrm{NS}}
    +
    \widehat{\mathcal C}_B
    +
    \lambda_k\widehat{\mathcal A}_B.
    \]

    \State Update \(\theta,\phi\) by map descent:
    \[
    \theta
    \leftarrow
    \theta
    -
    \eta_G\nabla_\theta
    \widehat{\mathcal L}_{\mathrm{map},k},
    \qquad
    \phi
    \leftarrow
    \phi
    -
    \eta_G\nabla_\phi
    \widehat{\mathcal L}_{\mathrm{map},k}.
    \]

\EndFor

\State \Return \(f_\theta,g_\phi\).

\end{algorithmic}
\end{algorithm}

For direct consistency with the compact-domain analysis, the generator
outputs should remain in \(\mathcal X\). This can be enforced through a
range-preserving output parameterization or, when convenient, by composing
raw neural networks with the metric projection onto the compact convex
set \(\mathcal X\).

Theorem~\ref{thm:neural_jsd_ot_recovery} concerns global minimizers of
the fixed population minimax objective
\(\mathcal E_{\lambda,m}\). It does not establish convergence of the
alternating stochastic algorithm, finite-sample discriminator training,
the non-saturating surrogate
\eqref{eq:non_saturating_adversarial_loss}, or the changing-weight
schedule \eqref{eq:transport_weight_schedule}. These practical choices
are motivated by the population objective but require a separate
optimization analysis.

\section{Global Recovery of the Quadratic Brenier Maps}
\label{sec:global_recovery}

Section~\ref{sec:proposed_framework} introduced the neural population
objective
\[
\mathcal E_{\lambda,m}(f,g)
=
\lambda\mathcal A(f,g)
+
\widehat{\mathcal B}_m(f,g)
+
\mathcal C(f,g).
\]
We now show that its global minimizers recover the forward and reverse
quadratic Brenier maps when the discriminator level becomes sufficiently
large and the action weight becomes sufficiently small but strictly positive.

The proof has three components. First, neural approximation provides a
Brenier-approximating comparison sequence whose action converges to the
optimal value and whose cycle loss vanishes. Second, an increasing
discriminator family converts small neural JSD into uniform Wasserstein
control of the generated marginals. Third, stability of the quadratic
Monge problem converts approximate marginal feasibility and near-optimal
action into \(L^2\)-closeness to the Brenier maps.

\subsection{Neural approximation ingredients}
\label{subsec:neural_approximation_ingredients}

\subsubsection{Generator approximation \& range preservation}
Classical universal-approximation results imply that suitable scalar neural
families are dense in \(C(K)\) under the uniform norm for every compact
\(K\subset\mathbb R^d\), cf.~\cite{bandeira2026mathematics}. For example, Hornik~\cite{hornik1991approximation}
establishes uniform approximation on compact sets for continuous, bounded,
nonconstant activations, while Leshno et al.~\cite{leshno1993multilayer} characterize universal approximation, under suitable regularity assumptions, by nonpolynomiality of the activation. Applying these results coordinatewise gives local uniform approximation of continuous maps \(K\to\mathbb R^d\).

Uniform approximation also implies \(L^p\)-approximation for every finite
Borel measure \(\mu\) on \(K\) and \(1\leq p<\infty\). Indeed, if
\(h_c\in C(K)\), then
\[
\|N-h_c\|_{L^p(\mu)}
\leq
\mu(K)^{1/p}\|N-h_c\|_\infty.
\]
Together with the fact that \(C(K)\) is dense in \(L^p(\mu)\) and triangle inequality, we have the neural family is dense in \(L^p(\mu)\). The vector-valued conclusion follows coordinatewise.

We must additionally ensure that admissible generators take values in
\(\mathcal X\). Let
\(\mathcal N_{\mathrm f},\mathcal N_{\mathrm g}
\subset C(\mathbb R^d;\mathbb R^d)\) be locally uniformly dense raw neural
classes, and define their range-preserving subclasses by
\[
\mathcal F
:=
\{N|_{\mathcal X}:N\in\mathcal N_{\mathrm f},\
  N(\mathcal X)\subseteq\mathcal X\},
\qquad
\mathcal G
:=
\{N|_{\mathcal X}:N\in\mathcal N_{\mathrm g},\
  N(\mathcal X)\subseteq\mathcal X\}.
\]
These subclasses remain uniformly dense in
\(C(\mathcal X;\mathcal X)\). To see this, fix
\(x_0\in\operatorname{int}(\mathcal X)\) and choose \(\rho>0\) such that $B(x_0,\rho)\subset\mathcal X$. For \(h\in C(\mathcal X;\mathcal X)\) and \(0<\delta<1\), set
\[
h_\delta:=(1-\delta)h+\delta x_0.
\]
Convexity gives
\[
B(h_\delta(x),\delta\rho)\subset\mathcal X
\qquad\text{for every }x\in\mathcal X.
\]
Indeed, if \(\|z-h_\delta(x)\|<\delta\rho\), then
\[
z
=
(1-\delta)h(x)
+
\delta\left(x_0+\frac{z-h_\delta(x)}{\delta}\right)
\in\mathcal X.
\]
Choose a raw neural map \(N\) with
\[
\|N-h_\delta\|_\infty<\frac{\delta\rho}{2}.
\]
Then \(N(\mathcal X)\subseteq\mathcal X\), while
\[
\|N-h\|_\infty
\leq
\frac{\delta\rho}{2}
+
\delta\operatorname{diam}(\mathcal X)
\longrightarrow0
\qquad\text{as }\delta\downarrow0.
\]
Thus, \(\mathcal F\) and \(\mathcal G\) are uniformly dense in
\(C(\mathcal X;\mathcal X)\). Moreover, convexity ensures that the
straight-line interpolation between \(x\in\mathcal X\) and
\(f(x)\in\mathcal X\) remains in \(\mathcal X\).

\subsubsection{Joint map \& cycle approximation}
To approximate an almost-everywhere inverse pair \((T,S)\), separate
\(L^2\)-approximations are not sufficient. Indeed, even if $T_j\to T \ \text{in }L^2(\mu_0)$ and $S_j\to S \ \text{in }L^2(\mu_1)$, it does not generally follow that $S_j\circ T_j\to S\circ T$ or $T_j\circ S_j\to T\circ S$ in the corresponding \(L^2\) spaces. Consequently, the cycle loss need not
converge to zero. The following lemma constructs a single approximating
sequence that controls both map-approximation errors and the cycle loss
simultaneously.

\begin{lemma}[Joint approximation of an almost-everywhere inverse pair]
\label{lem:joint_cycle_approximation}
Let \(\mathcal X\subset\mathbb R^d\) be a nonempty compact convex set, let
\(\mu_0,\mu_1\in\mathcal P(\mathcal X)\), and let
\(T,S:\mathcal X\to\mathcal X\) be measurable maps satisfying
\[
T_\#\mu_0=\mu_1,
\qquad
S\circ T=\operatorname{Id}
\quad\mu_0\text{-a.e.},
\qquad
T\circ S=\operatorname{Id}
\quad\mu_1\text{-a.e.}
\]
Assume that $\mathcal F,\mathcal G\subset C(\mathcal X;\mathcal X)$ are uniformly dense in \(C(\mathcal X;\mathcal X)\). Then there exists a sequence $(T_j,S_j)\in\mathcal F\times\mathcal G$ such that
\[
\|T_j-T\|_{L^2(\mu_0)}
+
\|S_j-S\|_{L^2(\mu_1)}
\longrightarrow0
\]
and
\[
\mathcal C(T_j,S_j)\longrightarrow0.
\]
\end{lemma}

See Appendix~\ref{a:proof_joint_cycle_approximation} for a proof of the lemma.

\subsubsection{Discriminator uniform approximation}
Universal approximation also yields discriminator classes that are dense in
\(C(\mathcal X;(0,1))\). For continuous nonpolynomial hidden-layer
activations, including ReLU, leaky ReLU, softplus, and ELU, this follows
from Leshno et al.~\cite{leshno1993multilayer}. For continuous sigmoidal
hidden-layer activations, such as the logistic sigmoid and hyperbolic
tangent, see Cybenko~\cite{cybenko1989approximation} and
Hornik~\cite{hornik1991approximation}.

Regardless of the hidden-layer activation, we use a logistic-sigmoid
output. Thus every discriminator has the form
\[
D=\sigma\circ\ell,
\qquad
\sigma(u):=\frac{1}{1+e^{-u}},
\]
where \(\ell:\mathcal X\to\mathbb R\) is a neural logit. In an
implementation based on a binary-cross-entropy-with-logits loss, this
sigmoid transformation may be applied implicitly for numerical stability.

For each \(r\in\mathbb N\), let \(\mathcal L_r\) be the union of the
logit realization classes of finitely many neural architectures with
compact parameter sets, and define
\[
\mathcal R_r
:=
\{\sigma\circ\ell:\ell\in\mathcal L_r\}.
\]
Define the discriminator levels cumulatively by
\[
\mathcal D_m
:=
\{D\equiv1/2\}
\cup
\bigcup_{r=1}^m\mathcal R_r.
\]
Then
\[
\mathcal D_1\subseteq\mathcal D_2\subseteq\cdots,
\qquad
D\equiv1/2\in\mathcal D_m,
\]
and every \(D\in\mathcal D_m\) takes values in \((0,1)\).

Moreover, because \(\mathcal X\) and the parameter sets at each fixed
level are compact, and only finitely many architectures occur in
\(\mathcal D_m\), continuity of the sigmoid-output networks implies that
there exists \(c_m\in(0,1/2)\) such that
\[
c_m\leq D(x)\leq1-c_m
\qquad
\text{for every }D\in\mathcal D_m
\text{ and }x\in\mathcal X.
\]
This fixed-level range bound is established explicitly in
Lemma~\ref{lem:fixed_discriminator_compactness} below.

Suppose that the allowable architectures and parameter bounds increase
with \(r\), so that $\bigcup_{r\geq1}\mathcal L_r$ is uniformly dense in \(C(\mathcal X)\). Let \(H\in C(\mathcal X;(0,1))\). Since \(\mathcal X\) is compact, \(H\) is bounded away from zero and one, and hence
\[
\operatorname{logit}(H)
:=
\log\frac{H}{1-H}
\in C(\mathcal X).
\]
Given \(\varepsilon>0\), choose
\(\ell\in\bigcup_{r\geq1}\mathcal L_r\) such that
\[
\|\ell-\operatorname{logit}(H)\|_\infty<4\varepsilon.
\]
Since the logistic sigmoid is \(1/4\)-Lipschitz,
\[
\|\sigma\circ\ell-H\|_\infty
=
\|\sigma\circ\ell
-\sigma\circ\operatorname{logit}(H)\|_\infty
\leq
\frac14
\|\ell-\operatorname{logit}(H)\|_\infty
<\varepsilon.
\]
The network \(\ell\) belongs to some finite level, so
\(\sigma\circ\ell\in\mathcal D_m\) for all sufficiently large \(m\).
Consequently,
\[
\overline{\bigcup_{m\geq1}\mathcal D_m}
=
C(\mathcal X;(0,1))
\]
in the relative uniform-norm topology.

\subsubsection{Fixed-level discriminator compactness \& continuity}
At each fixed level, compact parameter sets and a uniform range bound yield compact discriminator and log-discriminator classes.

\begin{lemma}[Fixed-level compactness of log-discriminator classes]
\label{lem:fixed_discriminator_compactness}
Fix \(m\in\mathbb N\). Suppose that \(\mathcal D_m\) is the realization
class of finitely many neural architectures whose parameter sets are
compact. Assume that the activations are continuous and that every
realized discriminator is continuous and takes values in \((0,1)\) (i.e. because it has a logistic-sigmoid output). Then, \(\mathcal D_m\) is compact in
\[
\bigl(C(\mathcal X),\|\cdot\|_\infty\bigr),
\]
and the families
\[
\{\log D:D\in\mathcal D_m\},
\qquad
\{\log(1-D):D\in\mathcal D_m\}
\]
are compact in the same space.
\end{lemma}

\begin{proof}
Index the finitely many architectures allowed at level \(m\) by
\(a=1,\ldots,A_m\). For architecture \(a\), let
\(\Theta_{m,a}\) be its compact parameter set and define
\[
F_{m,a}:\Theta_{m,a}\times\mathcal X\to(0,1),
\qquad
F_{m,a}(\theta,x):=D_{a,\theta}(x).
\]
Because architecture \(a\) consists of finitely many affine layers
composed with continuous activation functions, its realization
\[
F_{m,a}(\theta,x):=D_{a,\theta}(x)
\]
is jointly continuous in the parameters \(\theta\) and the input \(x\). Furthermore, since $\Theta_{m,a}\times\mathcal X$ is compact, \(F_{m,a}\) is uniformly continuous. Consequently, the
realization map
\[
R_{m,a}:\Theta_{m,a}\to C(\mathcal X),
\qquad
R_{m,a}(\theta):=D_{a,\theta},
\]
is continuous with respect to the Euclidean norm on
\(\Theta_{m,a}\) and the uniform norm on \(C(\mathcal X)\). Indeed, if
\(\theta_n\to\theta\), uniform continuity gives
\[
\|D_{a,\theta_n}-D_{a,\theta}\|_\infty
=
\sup_{x\in\mathcal X}
|F_{m,a}(\theta_n,x)-F_{m,a}(\theta,x)|
\longrightarrow0.
\]
Now, since \(\Theta_{m,a}\) is compact, its image $R_{m,a}(\Theta_{m,a})$ is compact in \(C(\mathcal X)\). Therefore,
\[
\mathcal D_m
=
\bigcup_{a=1}^{A_m}R_{m,a}(\Theta_{m,a})
\]
is compact as a finite union of compact sets. That proves the compactness of $\mathcal D_m$.

It remains to prove compactness of the logarithmic classes. For each
architecture \(a\), consider the continuous function
\[
q_{m,a}(\theta,x)
:=
\min\bigl\{
F_{m,a}(\theta,x),
1-F_{m,a}(\theta,x)
\bigr\}.
\]
Because every discriminator takes values strictly in \((0,1)\), we have $q_{m,a}(\theta,x)>0$ on the compact set \(\Theta_{m,a}\times\mathcal X\). It then follows from the continuity of $q_{m,a}$ and compactness of $\Theta_{m,a}\times\mathcal X$ that
\[
\eta_{m,a}
:=
\min_{(\theta,x)\in\Theta_{m,a}\times\mathcal X}
q_{m,a}(\theta,x)
>0.
\]
Since only finitely many architectures are allowed, the number
\[
c_m
:=
\frac12
\min_{1\leq a\leq A_m}\eta_{m,a}
\]
is strictly positive and belongs to \((0,1/2)\). It follows that $c_m\leq D(x)\leq1-c_m$ for every $D\in\mathcal D_m$ and $x\in\mathcal X$. On \([c_m,1-c_m]\), the scalar functions
\[
u\longmapsto\log u,
\qquad
u\longmapsto\log(1-u)
\]
are \(1/c_m\)-Lipschitz. Thus, for \(D,E\in\mathcal D_m\), we have
\[
\|\log D-\log E\|_\infty
\leq
\frac1{c_m}\|D-E\|_\infty
\]
and
\[
\|\log(1-D)-\log(1-E)\|_\infty
\leq
\frac1{c_m}\|D-E\|_\infty.
\]
Therefore, the maps
\[
D\longmapsto\log D,
\qquad
D\longmapsto\log(1-D)
\]
are continuous from \(\mathcal D_m\) into \(C(\mathcal X)\). Their
images of the compact set \(\mathcal D_m\) are consequently compact. That completes the proof.
\end{proof}

For \(P,Q\in\mathcal P(\mathcal X)\) and
\(D\in C(\mathcal X;(0,1))\), define
\[
\Phi_{P,Q}(D)
:=
\log 2
+
\frac12
\left[
\int_{\mathcal X}\log D\,dP
+
\int_{\mathcal X}\log(1-D)\,dQ
\right].
\]
Then, by definition,
\[
\widehat{\operatorname{JSD}}_m(P,Q)
=
\sup_{D\in\mathcal D_m}\Phi_{P,Q}(D).
\]
We write \(P_j\rightharpoonup P\) for weak convergence of probability
measures. The following proposition shows that, at every fixed
discriminator level \(m\), the map
\[
(P,Q)
\longmapsto
\widehat{\operatorname{JSD}}_m(P,Q)
\]
is continuous on
\(\mathcal P(\mathcal X)\times\mathcal P(\mathcal X)\) endowed with the
product weak topology.

\begin{proposition}[Fixed-level weak continuity of neural JSD]
\label{prop:fixed_level_neural_jsd_continuity}
Fix \(m\in\mathbb N\), and suppose that
\[
\{\log D:D\in\mathcal D_m\},
\qquad
\{\log(1-D):D\in\mathcal D_m\}
\]
are compact in \(C(\mathcal X)\). If
\(P_j,Q_j,P,Q\in\mathcal P(\mathcal X)\) satisfy $P_j\rightharpoonup P$ and $Q_j\rightharpoonup Q$, then
\[
\widehat{\operatorname{JSD}}_m(P_j,Q_j)
\longrightarrow
\widehat{\operatorname{JSD}}_m(P,Q).
\]
If, in addition, \(D\equiv1/2\in\mathcal D_m\), then
\[
P_j\rightharpoonup Q
\quad\Longrightarrow\quad
\widehat{\operatorname{JSD}}_m(P_j,Q)\longrightarrow0.
\]
\end{proposition}

\begin{proof}
Let $\mathcal H_m^{(1)}:=\{\log D:D\in\mathcal D_m\}$ and $\mathcal H_m^{(0)}:=\{\log(1-D):D\in\mathcal D_m\}$. By Lemma~\ref{lem:fixed_discriminator_compactness},
\(\mathcal H_m^{(1)}\) is compact in
\((C(\mathcal X),\|\cdot\|_\infty)\). Therefore, there exits a finite \(\delta\)-net \(\{h_1,\ldots,h_N\}\) for \(\mathcal H_m^{(1)}\). It follows that
\[
\sup_{h\in\mathcal H_m^{(1)}}
\left|\int h\,d(P_j-P)\right|
\leq
\max_{1\leq i\leq N}
\left|\int h_i\,d(P_j-P)\right|
+2\delta.
\]
For each fixed \(\delta>0\), the functions
\(h_1,\ldots,h_N\) are fixed. Weak convergence and finiteness of the
net therefore imply
\[
\max_{1\leq i\leq N}
\left|\int h_i\,d(P_j-P)\right|
\longrightarrow0.
\]
Hence,
\[
\limsup_{j\to\infty}
\sup_{h\in\mathcal H_m^{(1)}}
\left|\int h\,d(P_j-P)\right|
\leq2\delta.
\]
Since \(\delta>0\) is arbitrary, letting \(\delta\downarrow0\) proves
the desired convergence. The same argument applied to \(\mathcal H_m^{(0)}\) gives
\[
\sup_{D\in\mathcal D_m}
\left|\int\log(1-D)\,d(Q_j-Q)\right|
\longrightarrow0.
\]
Consequently,
\[
\sup_{D\in\mathcal D_m}
|\Phi_{P_j,Q_j}(D)-\Phi_{P,Q}(D)|
\longrightarrow0.
\]
Using
\[
\left|\sup_D F_j(D)-\sup_D F(D)\right|
\leq
\sup_D|F_j(D)-F(D)|,
\]
we obtain the first conclusion. For every \(D\in\mathcal D_m\),
\[
\Phi_{Q,Q}(D)
=
\log 2
+
\frac12\int_{\mathcal X}\log(D(1-D))\,dQ
\leq0,
\]
because \(D(1-D)\leq1/4\). The discriminator \(D\equiv1/2\) attains
zero, and therefore
\[
\widehat{\operatorname{JSD}}_m(Q,Q)=0.
\]
The second conclusion follows from the first.
\end{proof}

We note here that the fixed-level qualification is essential. The proposition does not assert
uniform convergence as \(m\to\infty\).

\subsection{Recovery of the forward-backward Brenier pair}
\label{subsec:ot_recovery}

\paragraph{Neural approximation properties used below.}
The preceding results yield the following two properties.

\begin{enumerate}
\item[\textnormal{(G)}]
Since \(\mu_0\) and \(\mu_1\) are absolutely continuous, the forward
and reverse quadratic Brenier maps are almost-everywhere inverses. The range-preserving generator construction and
Lemma~\ref{lem:joint_cycle_approximation} therefore give a sequence
\[
(T_j,S_j)\in\mathcal F\times\mathcal G
\]
such that
\[
\|T_j-T^\star\|_{L^2(\mu_0)}
+
\|S_j-S^\star\|_{L^2(\mu_1)}
\longrightarrow0,
\qquad
\mathcal C(T_j,S_j)\longrightarrow0.
\]

\item[\textnormal{(D)}]
The discriminator sequence satisfies $\mathcal D_1\subseteq\mathcal D_2\subseteq\cdots$ and $D\equiv1/2\in\mathcal D_m$ for every fixed \(m\), both log-discriminator families are compact in
\(C(\mathcal X)\), and
\[
\overline{\bigcup_{m\geq1}\mathcal D_m}
=
C(\mathcal X;(0,1))
\]
in the relative uniform-norm topology.
\end{enumerate}
Property \textnormal{(G)} provides a Brenier map approximation whose action
converges to \(W^2\) and whose cycle loss vanishes. Property
\textnormal{(D)} provides fixed-level weak continuity and, through nestedness and density, uniform separation of distinct measures. No uniform approximation of true JSD and no simultaneous generator-discriminator diagonal are required.

For clarity, the neural population objective used below is
\[
\widehat{\mathcal B}_m(f,g)
:=
\widehat{\operatorname{JSD}}_m(f_\#\mu_0,\mu_1)
+
\widehat{\operatorname{JSD}}_m(g_\#\mu_1,\mu_0),
\]
\[
\mathcal E_{\lambda,m}(f,g)
:=
\lambda\mathcal A(f,g)
+
\widehat{\mathcal B}_m(f,g)
+
\mathcal C(f,g).
\]

In the following theorem, the phrase ``sufficiently expressive such that
universal approximation holds'' refers precisely to the concrete neural
approximation constructions above and hence suffices properties \textnormal{(G)} and
\textnormal{(D)}.

\begin{theorem}[OT Recovery under Neural JSD and Discriminator Annealing]
\label{thm:neural_jsd_ot_recovery}
Let \(\mathcal X\subset\mathbb R^d\) be a nonempty compact convex set with
nonempty interior, and let
\(\mu_0,\mu_1\in\mathcal P_{2,ac}(\mathcal X)\). Let
\(T^\star,S^\star:\mathcal X\to\mathcal X\) be the almost-everywhere unique
quadratic Brenier maps satisfying
\[
(T^\star)_\#\mu_0=\mu_1,
\qquad
(S^\star)_\#\mu_1=\mu_0,
\]
and set $W:=W_2(\mu_0,\mu_1)>0$. Assume that the generator classes \(\mathcal F,\mathcal G\) and the discriminator class $\mathcal{D} := \bigcup_{m\geq1}\mathcal D_m$ are sufficiently expressive such that property \textnormal{(G)} and \textnormal{(D)} are satisfied. For each \(m\in\mathbb N\) and \(\lambda>0\), assume that the minimum is attained, and let
\[
(f_{\lambda,m},g_{\lambda,m})
\in
\operatorname*{arg\,min}_{f\in\mathcal F,\;g\in\mathcal G}
\mathcal E_{\lambda,m}(f,g).
\]
Then, for every \(\varepsilon>0\), there exist
\(M_\varepsilon\in\mathbb N\) and \(\alpha_\varepsilon>0\) such that,
whenever
\[
m\geq M_\varepsilon
\qquad\text{and}\qquad
0<\lambda\leq\frac{\alpha_\varepsilon}{W^2},
\]
one has
\[
\|f_{\lambda,m}-T^\star\|_{L^2(\mu_0)}
+
\|g_{\lambda,m}-S^\star\|_{L^2(\mu_1)}
<\varepsilon,
\qquad
\mathcal C(f_{\lambda,m},g_{\lambda,m})
\leq\lambda W^2.
\]
\end{theorem}

\begin{proof}
We divide the proof into four steps.

\paragraph{Step 1: Brenier approximation benchmark.}
By property \textnormal{(G)} which follows from Lemma~\ref{lem:joint_cycle_approximation}, there exists
\((T_j,S_j)\in\mathcal F\times\mathcal G\) such that
\[
T_j\to T^\star\quad\text{in }L^2(\mu_0),
\qquad
S_j\to S^\star\quad\text{in }L^2(\mu_1),
\]
and \(\mathcal C(T_j,S_j)\to0\). The couplings
\((T_j,T^\star)_\#\mu_0\) and \((S_j,S^\star)_\#\mu_1\) give
\[
W_2((T_j)_\#\mu_0,\mu_1)
\leq
\|T_j-T^\star\|_{L^2(\mu_0)}
\longrightarrow0
\]
and
\[
W_2((S_j)_\#\mu_1,\mu_0)
\leq
\|S_j-S^\star\|_{L^2(\mu_1)}
\longrightarrow0.
\]
Therefore, for every fixed \(m\),
Proposition~\ref{prop:fixed_level_neural_jsd_continuity} gives
\begin{equation}
\widehat{\mathcal B}_m(T_j,S_j)\longrightarrow0.
\label{eq:fixed_m_generator_neural_loss}
\end{equation}
Here, \(m\) is held fixed while \(j\to\infty\).

The same \(L^2\)-convergence gives
\[
\mathcal A(T_j,S_j)\longrightarrow W^2.
\]
Indeed,
\[
\begin{aligned}
&\left|
\|T_j-\operatorname{Id}\|_{L^2(\mu_0)}^2
-
\|T^\star-\operatorname{Id}\|_{L^2(\mu_0)}^2
\right|\\
&\qquad\leq
\left(
\|T_j-\operatorname{Id}\|_{L^2(\mu_0)}
+
\|T^\star-\operatorname{Id}\|_{L^2(\mu_0)}
\right)
\|T_j-T^\star\|_{L^2(\mu_0)}
\longrightarrow0.
\end{aligned}
\]
and similarly for \(S_j\). Moreover, \(\mathcal A(T^\star,S^\star)=W^2\). Hence, we have $$\mathcal A(T_j,S_j)\longrightarrow \mathcal A(T^\star,S^\star) = W^2$$ by definition of $\mathcal A$.

Now, fix \(m\) and \(\lambda>0\). Global optimality of $(f_{\lambda,m},g_{\lambda,m})$ gives
\[
\mathcal E_{\lambda,m}(f_{\lambda,m},g_{\lambda,m})
\leq
\mathcal E_{\lambda,m}(T_j,S_j)
\]
for every \(j\). Letting \(j\to\infty\) yields
\begin{equation}
\mathcal E_{\lambda,m}(f_{\lambda,m},g_{\lambda,m})
\leq
\lambda W^2.
\label{eq:neural_population_benchmark}
\end{equation}
Because \(D\equiv1/2\in\mathcal D_m\), neural JSD is nonnegative. All
three terms in \(\mathcal E_{\lambda,m}\) are therefore nonnegative, and
\eqref{eq:neural_population_benchmark} implies
\begin{equation}
\mathcal A(f_{\lambda,m},g_{\lambda,m})\leq W^2,
\label{eq:neural_action_bound}
\end{equation}
\begin{equation}
\mathcal C(f_{\lambda,m},g_{\lambda,m})\leq\lambda W^2,
\label{eq:neural_cycle_bound}
\end{equation}
and
\begin{equation}
\widehat{\operatorname{JSD}}_m
((f_{\lambda,m})_\#\mu_0,\mu_1)
\leq\lambda W^2,
\label{eq:forward_neural_jsd_bound}
\end{equation}
\begin{equation}
\widehat{\operatorname{JSD}}_m
((g_{\lambda,m})_\#\mu_1,\mu_0)
\leq\lambda W^2.
\label{eq:reverse_neural_jsd_bound}
\end{equation}

\paragraph{Step 2: Identification and uniform separation.}
For \(P,Q\in\mathcal P(\mathcal X)\), we first show that
\begin{equation}
\left[
\widehat{\operatorname{JSD}}_r(P,Q)=0
\text{ for every }r\geq1
\right]
\quad\Longrightarrow\quad
P=Q.
\label{eq:neural_identification}
\end{equation}
Suppose that the quantities on the left vanish. Then
\[
\Phi_{P,Q}(D)\leq0
\qquad
\text{for every }D\in\bigcup_{r\geq1}\mathcal D_r.
\]
Fix \(\varphi\in C(\mathcal X)\). For all sufficiently small positive and
negative \(t\),
\[
D_t:=\frac12+t\varphi
\]
belongs to \(C(\mathcal X;(0,1))\). By density, choose
\(D_{t,k}\in\bigcup_r\mathcal D_r\) with
\(D_{t,k}\to D_t\) uniformly. Since \(D_t\) is bounded away from zero and
one,
\[
\log D_{t,k}\to\log D_t,
\qquad
\log(1-D_{t,k})\to\log(1-D_t)
\]
uniformly. Therefore
\[
\Phi_{P,Q}(D_t)\leq0.
\]
At \(t=0\), \(D_0\equiv1/2\) and \(\Phi_{P,Q}(D_0)=0\). Hence \(t=0\) is
a local maximum of \(t\mapsto\Phi_{P,Q}(D_t)\), and
\[
0
=
\left.\frac{d}{dt}\Phi_{P,Q}(D_t)\right|_{t=0}
=
\int_{\mathcal X}\varphi\,dP
-
\int_{\mathcal X}\varphi\,dQ.
\]
Since this holds for every \(\varphi\in C(\mathcal X)\), the Riesz
representation theorem gives \(P=Q\), proving
\eqref{eq:neural_identification}.

We next establish uniform separation. That is, for every
\(Q\in\mathcal P(\mathcal X)\) and \(a>0\), there exist
\(M_{Q,a}\in\mathbb N\) and \(\beta_{Q,a}>0\) such that
\begin{equation}
m\geq M_{Q,a},
\qquad
\widehat{\operatorname{JSD}}_m(P,Q)\leq\beta_{Q,a}
\quad\Longrightarrow\quad
W_2(P,Q)<a
\label{eq:neural_w2_separation}
\end{equation}
for every \(P\in\mathcal P(\mathcal X)\).

We prove by contradiction. Assume this is false, choosing the discriminator threshold \(1/n\) and the
minimum level \(n\) would give \(m_n\geq n\) and
\(P_n\in\mathcal P(\mathcal X)\) such that
\[
W_2(P_n,Q)\geq a,
\qquad
\widehat{\operatorname{JSD}}_{m_n}(P_n,Q)\leq\frac1n.
\]
Because \(\mathcal X\) is compact, the space
\(\mathcal P(\mathcal X)\) is compact under weak convergence. Hence,
after passing to a subsequence, which we do not relabel, there exists
\(P\in\mathcal P(\mathcal X)\) such that
\[
P_n\rightharpoonup P.
\]
On the compact space \(\mathcal X\), weak convergence implies
\(W_2\)-convergence. Therefore,
\[
W_2(P,Q)
=
\lim_{n\to\infty}W_2(P_n,Q)
\geq a.
\]

Now, fix \(r\in\mathbb N\). Eventually \(m_n\geq r\), so the set inclusion nestedness gives
\[
0
\leq
\widehat{\operatorname{JSD}}_r(P_n,Q)
\leq
\widehat{\operatorname{JSD}}_{m_n}(P_n,Q)
\leq
\frac1n.
\]
Also, for fixed \(r\), the map
\[
P'\mapsto\widehat{\operatorname{JSD}}_r(P',Q)
\]
is lower semicontinuous under weak convergence, because it is a supremum of weakly continuous functions of \(P'\). Hence,
\[
0
\leq
\widehat{\operatorname{JSD}}_r(P,Q)
\leq
\liminf_{n\to\infty}
\widehat{\operatorname{JSD}}_r(P_n,Q)
=0.
\]
Thus, the level-\(r\) value vanishes for every \(r\). By
\eqref{eq:neural_identification}, \(P=Q\), which contradicts
\(W_2(P,Q)\geq a\). This proves \eqref{eq:neural_w2_separation}.

\paragraph{Step 3: Stability of the quadratic Brenier map.}
Let $\mu\in\mathcal P_{2,ac}(\mathcal X)$, $\nu\in\mathcal P(\mathcal X)$, and \(T\) be the quadratic Brenier map from \(\mu\) to \(\nu\). Write $W_{\mu,\nu}:=W_2(\mu,\nu)$. We claim that, for every \(\tau>0\), there exist
\(\eta_\tau,\delta_\tau>0\) such that every measurable
\(h:\mathcal X\to\mathcal X\) satisfying
\[
W_2(h_\#\mu,\nu)\leq\eta_\tau
\]
and
\[
\int_{\mathcal X}\|h(x)-x\|^2\,d\mu(x)
\leq
W_{\mu,\nu}^2+\delta_\tau
\]
also satisfies
\[
\|h-T\|_{L^2(\mu)}<\tau.
\]

Suppose otherwise. Then, for some \(\tau_0>0\), there would exist measurable
maps \(h_n:\mathcal X\to\mathcal X\) such that
\[
W_2((h_n)_\#\mu,\nu)\leq\frac1n,
\]
\[
I_n
:=
\int_{\mathcal X}\|h_n(x)-x\|^2\,d\mu(x)
\leq
W_{\mu,\nu}^2+\frac1n,
\]
but
\[
\|h_n-T\|_{L^2(\mu)}\geq\tau_0.
\]
The graph coupling induced by \(h_n\) gives
\[
W_2(\mu,(h_n)_\#\mu)\leq I_n^{1/2}.
\]
It therefore follows from the triangle inequality that
\[
W_{\mu,\nu}\leq I_n^{1/2}+\frac1n.
\]
Together with the upper bound on \(I_n\), this implies
\[
I_n\longrightarrow W_{\mu,\nu}^2.
\]

Define
\[
\gamma_n:=(\operatorname{Id},h_n)_\#\mu.
\]
After passing to a subsequence, \(\gamma_n\rightharpoonup\gamma\) in
\(\mathcal P(\mathcal X\times\mathcal X)\). Its first marginal is \(\mu\),
and its second marginal is \(\nu\) because
\((h_n)_\#\mu\to\nu\) in \(W_2\). Continuity of the quadratic cost on the
compact product gives
\[
\int_{\mathcal X\times\mathcal X}\|x-y\|^2\,d\gamma(x,y)
=
\lim_{n\to\infty}I_n
=
W_{\mu,\nu}^2.
\]
Thus, \(\gamma\) is an optimal quadratic coupling between
\(\mu\) and \(\nu\). Since \(\mu\) is absolutely continuous, the
optimal quadratic coupling is unique and is induced by the Brenier map
\(T\), which is unique \(\mu\)-almost everywhere. Therefore,
\[
\gamma=(\operatorname{Id},T)_\#\mu.
\]

Fix \(s>0\) and choose
\(T_s\in C(\mathcal X;\mathbb R^d)\) with
\(\|T_s-T\|_{L^2(\mu)}<s\). Since
\((x,y)\mapsto\|y-T_s(x)\|^2\) is bounded and continuous, let $n \rightarrow \infty$, we have
\[
\int_{\mathcal X}\|h_n(x)-T_s(x)\|^2\,d\mu(x)
\longrightarrow
\int_{\mathcal X}\|T(x)-T_s(x)\|^2\,d\mu(x)
<s^2.
\]
Moreover,
\[
\|h_n-T\|_{L^2(\mu)}^2
\leq
2\|h_n-T_s\|_{L^2(\mu)}^2
+
2\|T_s-T\|_{L^2(\mu)}^2.
\]
Consequently,
\[
\limsup_{n\to\infty}\|h_n-T\|_{L^2(\mu)}^2
\leq4s^2.
\]
Letting \(s\downarrow0\) contradicts
\(\|h_n-T\|_{L^2(\mu)}\geq\tau_0\). This proves our stability claim.

\paragraph{Step 4: Forward and reverse recovery.}
Apply Step~3 with \(\tau=\varepsilon/2\) to
\[
(\mu,\nu,T)=(\mu_0,\mu_1,T^\star)
\]
and
\[
(\mu,\nu,T)=(\mu_1,\mu_0,S^\star).
\]
Let the resulting constants be
\[
\eta_\varepsilon^{\mathrm f},\delta_\varepsilon^{\mathrm f},
\qquad
\eta_\varepsilon^{\mathrm g},\delta_\varepsilon^{\mathrm g},
\]
and define
\[
a_\varepsilon
:=
\min\left\{
\eta_\varepsilon^{\mathrm f},
\eta_\varepsilon^{\mathrm g},
\frac{\delta_\varepsilon^{\mathrm f}}{2W},
\frac{\delta_\varepsilon^{\mathrm g}}{2W},
\frac W2
\right\}.
\]
Apply \eqref{eq:neural_w2_separation} with
\((Q,a)=(\mu_1,a_\varepsilon)\) and
\((Q,a)=(\mu_0,a_\varepsilon)\). Denote the resulting constants by
\[
M_\varepsilon^{\mathrm f},\beta_\varepsilon^{\mathrm f},
\qquad
M_\varepsilon^{\mathrm g},\beta_\varepsilon^{\mathrm g},
\]
and set
\[
M_\varepsilon
:=
\max\{M_\varepsilon^{\mathrm f},M_\varepsilon^{\mathrm g}\},
\qquad
\alpha_\varepsilon
:=
\min\{\beta_\varepsilon^{\mathrm f},
       \beta_\varepsilon^{\mathrm g}\}.
\]

Suppose
\[
m\geq M_\varepsilon,
\qquad
0<\lambda\leq\frac{\alpha_\varepsilon}{W^2}.
\]
The bounds \eqref{eq:forward_neural_jsd_bound} and
\eqref{eq:reverse_neural_jsd_bound}, followed by uniform separation, give
\[
W_2((f_{\lambda,m})_\#\mu_0,\mu_1)<a_\varepsilon,
\qquad
W_2((g_{\lambda,m})_\#\mu_1,\mu_0)<a_\varepsilon.
\]

Write
\[
I_f
:=
\int_{\mathcal X}\|f_{\lambda,m}(x)-x\|^2\,d\mu_0(x),
\qquad
I_g
:=
\int_{\mathcal X}\|g_{\lambda,m}(y)-y\|^2\,d\mu_1(y).
\]
By the action bound \eqref{eq:neural_action_bound} established in
Step~1 and the definition of the bidirectional action,
\[
\frac12(I_f+I_g)
=
\mathcal A(f_{\lambda,m},g_{\lambda,m})
\leq W^2.
\]
Therefore,
\[
I_f+I_g\leq2W^2.
\]
Now, let $\nu_g:=(g_{\lambda,m})_\#\mu_1$. The graph coupling
\((\operatorname{Id},g_{\lambda,m})_\#\mu_1\) gives
\[
W_2(\mu_1,\nu_g)\leq I_g^{1/2}.
\]
Moreover, the reverse triangle inequality for \(W_2\) gives
\[
W_2(\mu_1,\nu_g)
\geq
W_2(\mu_1,\mu_0)-W_2(\nu_g,\mu_0).
\]
Consequently,
\[
I_g^{1/2}
\geq
W-W_2((g_{\lambda,m})_\#\mu_1,\mu_0)
>
W-a_\varepsilon.
\]
Therefore,
\[
\begin{aligned}
I_f
&\leq2W^2-I_g <2W^2-(W-a_\varepsilon)^2 =W^2+2Wa_\varepsilon-a_\varepsilon^2 \leq W^2+\delta_\varepsilon^{\mathrm f}.
\end{aligned}
\]
The symmetric argument gives
\[
I_g<W^2+\delta_\varepsilon^{\mathrm g}.
\]
Since
\(a_\varepsilon\leq\eta_\varepsilon^{\mathrm f}\) and
\(a_\varepsilon\leq\eta_\varepsilon^{\mathrm g}\), it follows from Step~3 that:
\[
\|f_{\lambda,m}-T^\star\|_{L^2(\mu_0)}<\frac\varepsilon2,
\qquad
\|g_{\lambda,m}-S^\star\|_{L^2(\mu_1)}<\frac\varepsilon2.
\]
Adding these inequalities proves the map-recovery conclusion, while
\eqref{eq:neural_cycle_bound} gives
\[
\mathcal C(f_{\lambda,m},g_{\lambda,m})\leq\lambda W^2.
\]
\end{proof}

\subsection{Discriminator refinement \& action-weight annealing}
\label{subsec:annealing_interpretation}

The transport weight \(\lambda\) balances the transportation cost against marginal
matching and cycle consistency. The theoretical estimate gives
\[
\widehat{\mathcal B}_m(f_{\lambda,m},g_{\lambda,m})
+
\mathcal C(f_{\lambda,m},g_{\lambda,m})
\leq
\lambda W^2,
\qquad
\mathcal A(f_{\lambda,m},g_{\lambda,m})\leq W^2,
\]
with increasing discriminator class indexed by $m$. Thus, for consistent implementation, annealing \(\lambda\) toward zero is essential; together with a sufficiently expressive discriminator level, it forces marginal mismatch to vanish to mimic the increasing discriminator class. At the same time, \(\lambda\) must remain strictly positive. Even when its coefficient is small, the action term retains the optimal-transport selection mechanism among approximately feasible transports. Finally, at the end of annealing, the matching has already been forged during the previous training dynamics when the actions scale is still large. Therefore, although the final training epochs should have nearly zero action terms and focuses on distribution matching, the perturbation on the already forged matching remains relatively small.

The proof also relates the annealing scale to the required accuracy of
marginal enforcement. For a prescribed tolerance \(a>0\), the separation
property provides discriminator levels and thresholds
\[
M_{\mu_1,a},\ \beta_{\mu_1,a},
\qquad
M_{\mu_0,a},\ \beta_{\mu_0,a}.
\]
Consequently, if
\[
m\geq
\max\{M_{\mu_1,a},M_{\mu_0,a}\},
\qquad
\lambda W^2
\leq
\min\{\beta_{\mu_1,a},\beta_{\mu_0,a}\},
\]
then both generated marginals lie within \(W_2\)-distance \(a\) of their
targets. Hence \(\lambda W^2\) acts as a feasibility-error budget, while the
quantities \(\beta_{\mu_i,a}\) describe the resolution of the discriminator
family at the desired marginal scale. The theorem uses this relation with
\(a=a_\varepsilon\). This calibration is qualitative: an explicit annealing
rate would require quantitative bounds on the discriminator separation
modulus.

This suggests a continuation interpretation. A larger initial value of
\(\lambda\) imposes a stronger low-displacement geometric bias. Gradually
reducing \(\lambda\) then transfers emphasis to marginal matching and cycle consistency, while a residual positive action weight preserves the
optimal-transport selection principle. The theorem establishes this
small-positive-\(\lambda\) behavior for global minimizers; it does not assert that beginning with a large \(\lambda\), or any particular annealing schedule, is necessary for optimization, nor does it analyze local minima.

\section{Structural Roles of Endpoint Matching, Cycle Consistency, \& Action}
\label{sec:theory}

Section~\ref{sec:global_recovery} establishes global recovery of the
forward--reverse Brenier pair. We now examine the structural mechanisms
behind this result. The analysis addresses three questions: how cycle
consistency reduces the class of admissible transport maps, how the proposed
penalties control different forms of mode collapse, and how the local
competition between quadratic action and cycle consistency disfavors
suboptimal inverse-compatible solutions.

The discussion has three parts. First, endpoint matching alone leaves a large
class of admissible forward--backward map pairs. Cycle consistency narrows
this class by requiring the two maps to be approximately inverse to one
another. This reduction is substantial, although it does not by itself
identify the Brenier pair, because many inverse-compatible maps can transport
the prescribed endpoint marginals.

Second, we distinguish two forms of mode collapse. Distribution-level
collapse occurs when a generated endpoint distribution fails to cover a
region carrying nonnegligible target mass. After sufficient discriminator
refinement, the neural-JSD term controls this failure at any prescribed
positive spatial scale. Map-level collapse occurs when many distinct inputs
are sent to the same output. The cycle loss controls this second failure by
bounding the conditional dispersion of inputs associated with a common
generated output, as well as the corresponding dispersion of the induced
trajectories.

Third, we study the local competition between quadratic action and cycle
consistency near an inverse-compatible pair. Along a smooth
endpoint-feasible perturbation that decreases action, the action improvement
is generally first order in the perturbation size, whereas the cycle error
created by incompatibility of the forward and backward perturbations is
second order. Thus, the action term favors motion toward a lower-cost
transport, while the cycle term resists departures from inverse
compatibility. In particular, a suboptimal inverse-compatible pair admitting
such an action-decreasing feasible perturbation cannot be a local minimizer
of the population objective. At the Brenier pair, by contrast, no
endpoint-feasible perturbation can further reduce the quadratic action. This
provides a local structural interpretation of how the two terms jointly
favor the Brenier pair; it is not, by itself, a convergence guarantee for
gradient-based training.

Throughout this section, let
\(\mathcal X\subset\mathbb R^d\) be compact and convex, let
\(\mu_0,\mu_1\in\mathcal P_{2,ac}(\mathcal X)\), and let
\((T^\star,S^\star)\) denote the forward--reverse quadratic Brenier pair.
We use the population objective introduced in
Section~\ref{sec:proposed_framework}:
\[
\mathcal E_{\lambda,m}(f,g)
=
\lambda\mathcal A(f,g)
+
\widehat{\mathcal B}_m(f,g)
+
\mathcal C(f,g),
\]
where
\[
\widehat{\mathcal B}_m(f,g)
=
\widehat{\operatorname{JSD}}_m(f_\#\mu_0,\mu_1)
+
\widehat{\operatorname{JSD}}_m(g_\#\mu_1,\mu_0).
\]
The three terms therefore play complementary roles:
\begin{itemize}[leftmargin = *]
    \item Cycle consistency narrows the admissible map class by promoting
    approximate forward--backward invertibility and controlling many-to-one
    map collapse.
    \item Neural-JSD endpoint matching controls discrepancies between the
    generated and target endpoint distributions after sufficient
    discriminator refinement.
    \item Quadratic action favors lower-cost transports and, in competition
    with the cycle penalty, provides the minimum-action selection mechanism
    leading to the forward--reverse Brenier pair.
\end{itemize}

At a fixed discriminator level \(m\), vanishing neural JSD need not imply
exact equality of the corresponding measures. Consequently, exact endpoint
matching is used below only to describe the ideal limiting feasible class.
Its connection to the neural objective is provided by the identification and
uniform Wasserstein-separation properties established in
Section~\ref{sec:global_recovery}.

\subsection{Feasible-class reduction via cycle consistency}
\label{subsec:theory_q1}

Define the exact endpoint-matching class
\[
\mathcal M_{\mathrm{match}}
:=
\left\{
(f,g):
f_\#\mu_0=\mu_1,
\quad
g_\#\mu_1=\mu_0
\right\}.
\]
Under the discriminator assumptions of
Section~\ref{sec:global_recovery}, this limiting feasible class can also be
characterized through the full discriminator sequence:
\[
(f,g)\in\mathcal M_{\mathrm{match}}
\quad\Longleftrightarrow\quad
\widehat{\mathcal B}_r(f,g)=0
\quad\text{for every }r\geq1.
\]
In general, however, vanishing neural JSD at one fixed discriminator level
does not imply exact endpoint matching.

Endpoint matching alone leaves a large class of admissible map pairs. More
strongly, even exact cycle consistency does not uniquely determine the
Brenier pair. To see this, let
\(R:\mathcal X\to\mathcal X\) be a measurable
\(\mu_0\)-preserving bijection with a measurable inverse, where these
properties are understood up to \(\mu_0\)-null sets, and define
\[
f_R:=T^\star\circ R,
\qquad
g_R:=R^{-1}\circ S^\star.
\]
Then
\[
(f_R)_\#\mu_0=\mu_1,
\qquad
(g_R)_\#\mu_1=\mu_0.
\]
Moreover, because \(T^\star\) and \(S^\star\) are almost-everywhere
inverses,
\[
g_R\circ f_R=\operatorname{Id}
\quad\mu_0\text{-a.e.},
\qquad
f_R\circ g_R=\operatorname{Id}
\quad\mu_1\text{-a.e.}
\]
Thus, endpoint matching together with exact cycle consistency still leaves
a potentially large class of inverse-compatible transport pairs.

We next quantify the restriction imposed by a small cycle loss. For
\(\varepsilon>0\), define the two-sided reconstruction-failure probability
\[
p_\varepsilon(f,g)
:=
\mathbb P_{X\sim\mu_0}
\bigl(\|g(f(X))-X\|>\varepsilon\bigr)
+
\mathbb P_{Y\sim\mu_1}
\bigl(\|f(g(Y))-Y\|>\varepsilon\bigr).
\]

\begin{proposition}[Approximate invertibility at resolution \(\varepsilon\)]
\label{prop:cycle_resolution_invertibility}
For every \(\varepsilon>0\),
\[
p_\varepsilon(f,g)
\leq
\min\left\{
2,\frac{\mathcal C(f,g)}{\varepsilon^2}
\right\}.
\]
Consequently, if \(\mathcal C(f,g)\leq\delta\), then
\[
p_\varepsilon(f,g)
\leq
\min\left\{
2,\frac{\delta}{\varepsilon^2}
\right\}.
\]
\end{proposition}

See Appendix~\ref{a:proof_cycle_resolution_invertibility} for the proof of this statement. Proposition~\ref{prop:cycle_resolution_invertibility} shows that cycle consistency narrows the endpoint-matched candidate class by requiring the forward and backward maps to reconstruct one another with high probability
at every fixed positive resolution. Nevertheless, the construction \((f_R,g_R)\) shows that even exact cycle consistency leaves multiple inverse-compatible candidates. Cycle consistency is therefore a candidate-reduction mechanism rather than, by itself, an optimal-transport selection principle.

\subsection{Endpoint mode coverage and suppression of map-level collapse}
\label{subsec:cycle_mode_collapse}

Mode collapse can refer to two distinct failures. First, the generated
endpoint distribution may fail to cover a region carrying nonnegligible
target mass. Second, the learned map may send widely dispersed inputs to the
same output, making the input difficult to recover from the generated
sample. Neural-JSD endpoint matching and cycle consistency control these two
failures in complementary ways.

We first formulate the endpoint guarantee using the neural JSD appearing in
the proposed objective. For a Borel set \(A\subseteq\mathcal X\) and
\(r>0\), define its \(r\)-neighborhood by
\[
A^{(r)}
:=
\left\{
z\in\mathcal X:
\operatorname{dist}(z,A)<r
\right\}.
\]

\begin{proposition}[Neural-JSD mode coverage at a prescribed spatial scale]
\label{prop:neural_jsd_mode_coverage}
Assume the discriminator conditions of
Section~\ref{sec:global_recovery}. The, for every \(r>0\) and
\(\delta\in(0,1)\), there exist
\[
M_{r,\delta}\in\mathbb N,
\qquad
\beta_{r,\delta}>0,
\]
such that, whenever
\[
m\geq M_{r,\delta},
\qquad
\widehat{\mathcal B}_m(f,g)\leq\beta_{r,\delta},
\]
one has, for every Borel set \(A\subseteq\mathcal X\),
\[
(f_\#\mu_0)(A^{(r)})
\geq
\mu_1(A)-\delta
\]
and
\[
(g_\#\mu_1)(A^{(r)})
\geq
\mu_0(A)-\delta.
\]
In particular, if
\[
(f_\#\mu_0)(A^{(r)})=0,
\]
then \(\mu_1(A)\leq\delta\), and the analogous conclusion holds in the
reverse direction.
\end{proposition}

\begin{proof}
Set
\[
a:=r\sqrt{\delta}.
\]
Apply the neural-JSD separation property from
Section~\ref{sec:global_recovery} to the target measures
\(\mu_1\) and \(\mu_0\). Taking the larger of the resulting discriminator
levels and the smaller of the resulting neural-JSD thresholds gives
\(M_{r,\delta}\) and \(\beta_{r,\delta}\) such that
\[
m\geq M_{r,\delta},
\qquad
\widehat{\mathcal B}_m(f,g)\leq\beta_{r,\delta}
\]
imply
\[
W_2(f_\#\mu_0,\mu_1)<a
\]
and
\[
W_2(g_\#\mu_1,\mu_0)<a.
\]

Let \(P:=f_\#\mu_0\), and let
\(\pi\in\Pi(P,\mu_1)\) be an optimal quadratic coupling. If
\(y\in A\) and \(z\notin A^{(r)}\), then
\(\|z-y\|\geq r\). Therefore,
\[
\begin{aligned}
\mu_1(A)
&\leq
P(A^{(r)})
+
\pi\bigl(\{(z,y):\|z-y\|\geq r\}\bigr)\\
&\leq
P(A^{(r)})
+
\frac{1}{r^2}
\int_{\mathcal X\times\mathcal X}
\|z-y\|^2\,d\pi(z,y)\\
&=
P(A^{(r)})
+
\frac{W_2^2(P,\mu_1)}{r^2}\\
&<
P(A^{(r)})+\delta.
\end{aligned}
\]
This proves
\[
(f_\#\mu_0)(A^{(r)})
\geq
\mu_1(A)-\delta.
\]
The reverse inequality follows by applying the same argument to
\(g_\#\mu_1\) and \(\mu_0\).
\end{proof}

The enlargement from \(A\) to \(A^{(r)}\) is essential. Wasserstein
closeness does not control the probabilities of arbitrary sets without
additional assumptions on their boundaries. Likewise, a small neural JSD at
one fixed discriminator level does not imply the total-variation estimate
available for unrestricted population JSD. Proposition
\ref{prop:neural_jsd_mode_coverage} instead gives the appropriate
neural-JSD guarantee: after sufficient discriminator refinement, target mass
cannot be entirely absent from a prescribed positive neighborhood.

For the global minimizers considered in
Theorem~\ref{thm:neural_jsd_ot_recovery}, the benchmark inequality gives
\[
\widehat{\mathcal B}_m(f_{\lambda,m},g_{\lambda,m})
\leq
\lambda W_2^2(\mu_0,\mu_1).
\]
Consequently, the conclusion of
Proposition~\ref{prop:neural_jsd_mode_coverage} holds whenever
\[
m\geq M_{r,\delta},
\qquad
0<\lambda
\leq
\frac{\beta_{r,\delta}}
{W_2^2(\mu_0,\mu_1)}.
\]
This relates the spatial resolution \(r\), the tolerated missing mass
\(\delta\), the discriminator level \(m\), and the action weight
\(\lambda\).

We next quantify the many-to-one map collapse. Let \(X_0\sim\mu_0\) and
\(X_1\sim\mu_1\), and define
\[
\mathcal V_0(f)
:=
\mathbb E\!\left[
\operatorname{Var}(X_0\mid f(X_0))
\right],
\qquad
\mathcal V_1(g)
:=
\mathbb E\!\left[
\operatorname{Var}(X_1\mid g(X_1))
\right],
\]
where
\[
\operatorname{Var}(X\mid Z)
:=
\mathbb E\!\left[
\|X-\mathbb E[X\mid Z]\|^2
\,\middle|\,
Z
\right].
\]
These quantities measure the average dispersion of inputs associated with a
common generated output.

Define the directional cycle losses
\[
\mathcal C_0(f,g)
:=
\mathbb E\|g(f(X_0))-X_0\|^2,
\qquad
\mathcal C_1(f,g)
:=
\mathbb E\|f(g(X_1))-X_1\|^2.
\]

\begin{proposition}[Cycle consistency controls many-to-one collapse]
\label{prop:cycle_bounds_map_collapse}
For any measurable maps \(f,g\) for which the displayed quantities are
finite,
\[
\mathcal V_0(f)\leq\mathcal C_0(f,g),
\qquad
\mathcal V_1(g)\leq\mathcal C_1(f,g).
\]
Consequently,
\[
\mathcal V_0(f)+\mathcal V_1(g)
\leq
\mathcal C(f,g).
\]
\end{proposition}

\begin{proof}
The random variable \(g(f(X_0))\) is measurable with respect to
\(\sigma(f(X_0))\). Since
\(\mathbb E[X_0\mid f(X_0)]\) is the best
\(\sigma(f(X_0))\)-measurable \(L^2\)-predictor of \(X_0\),
\[
\mathbb E
\left\|
X_0-\mathbb E[X_0\mid f(X_0)]
\right\|^2
\leq
\mathbb E
\|X_0-g(f(X_0))\|^2.
\]
This proves the source-side inequality. Interchanging
\((X_0,f,g)\) with \((X_1,g,f)\) proves the target-side inequality.
Adding the two bounds gives the final claim.
\end{proof}

The same estimate controls conditional trajectory dispersion. Define
\[
X_t^F=(1-t)X_0+t f(X_0),
\qquad
X_t^B=(1-t)g(X_1)+tX_1.
\]
Conditionally on \(f(X_0)\), the term \(t f(X_0)\) is fixed, and
conditionally on \(g(X_1)\), the term \((1-t)g(X_1)\) is fixed.
Consequently,
\[
\mathbb E\!\left[
\operatorname{Var}(X_t^F\mid f(X_0))
\right]
=
(1-t)^2\mathcal V_0(f)
\leq
(1-t)^2\mathcal C_0(f,g)
\]
and
\[
\mathbb E\!\left[
\operatorname{Var}(X_t^B\mid g(X_1))
\right]
=
t^2\mathcal V_1(g)
\leq
t^2\mathcal C_1(f,g).
\]

Propositions~\ref{prop:neural_jsd_mode_coverage} and
\ref{prop:cycle_bounds_map_collapse} therefore describe complementary controls. After sufficient discriminator refinement, neural JSD prevents target mass from being absent at a prescribed positive spatial scale, whereas cycle consistency controls many-to-one preimage collapse and the resulting conditional dispersion of learned trajectories. Consistent with this analysis, the cycle ablation results in Appendix~\ref{app:cycle_ablation} show that removing cycle consistency leads to mode collapse. These are
population-level guarantees and do not, by themselves, eliminate errors caused by finite samples or incomplete optimization.

\subsection{Local interaction among action, neural JSD, \& cycle consistency}
\label{subsec:dynamic_stability_low_cycle}

We next examine the local interaction among the three terms in the population
objective defined in equation~\ref{eq:population_objective}. Unlike an exactly constrained formulation, CyclOT does not project generator updates onto the marginal-matching or inverse-compatible classes. We therefore analyze a general perturbation of the neural-generator parameters.

Fix a discriminator level \(m\), and write the generators as
\(f=f_\theta\) and \(g=g_\phi\). An admissible smooth perturbation is a pair
of one-sided \(C^2\) parameter curves
\[
\alpha\longmapsto\theta_\alpha,
\qquad
\alpha\longmapsto\phi_\alpha,
\qquad
\alpha\geq0,
\]
that remain in the admissible parameter sets and satisfy
\[
\theta_0=\theta,
\qquad
\phi_0=\phi.
\]
Any range restrictions imposed on the generators are also required to hold
along these curves.

For this subsection only, assume that the generator realization maps are
jointly \(C^2\) in their inputs and parameters on a neighborhood of the
relevant compact sets. This holds, for example, for finite neural networks
with \(C^2\) activations and smooth output maps along bounded parameter
curves. Writing
\[
f_\alpha:=f_{\theta_\alpha},
\qquad
g_\alpha:=g_{\phi_\alpha},
\]
we then have
\[
f_\alpha=f-\alpha U_f+O(\alpha^2),
\qquad
g_\alpha=g-\alpha U_g+O(\alpha^2),
\]
with uniform \(C^1\) remainders, where
\[
U_f:=
-\left.\frac{d}{d\alpha}f_{\theta_\alpha}\right|_{\alpha=0^+},
\qquad
U_g:=
-\left.\frac{d}{d\alpha}g_{\phi_\alpha}\right|_{\alpha=0^+}.
\]

Define
\[
\begin{aligned}
a(f,g;U_f,U_g)
:=
\int_{\mathcal X}
\langle f(x)-x,U_f(x)\rangle\,d\mu_0(x)
+
\int_{\mathcal X}
\langle g(y)-y,U_g(y)\rangle\,d\mu_1(y).
\end{aligned}
\]
The quadratic action satisfies
\[
\mathcal A(f_\alpha,g_\alpha)
=
\mathcal A(f,g)
-
\alpha a(f,g;U_f,U_g)
+
O(\alpha^2).
\]
Thus \(a>0\) means that the perturbation decreases the action to first
order.

Because the perturbation may change the endpoint marginals, the neural-JSD
term also contributes. Suppose that the fixed-level discriminator class has
a compact parameterization, its discriminators are continuously
differentiable on the relevant compact set, and their values are uniformly
bounded away from \(0\) and \(1\). Danskin's theorem then gives the
one-sided expansion
\[
\widehat{\mathcal B}_m(f_\alpha,g_\alpha)
=
\widehat{\mathcal B}_m(f,g)
+
\alpha b_m(f,g;U_f,U_g)
+
o(\alpha),
\]
where
\[
b_m(f,g;U_f,U_g)
:=
\left.
\frac{d}{d\alpha}
\widehat{\mathcal B}_m(f_\alpha,g_\alpha)
\right|_{\alpha=0^+}.
\]
If the maximizing discriminator is not unique, this derivative is obtained
by maximizing the corresponding directional derivative over the active
maximizers.

We next expand the cycle loss without assuming inverse compatibility.
Define the current residuals
\[
R_0(x):=g(f(x))-x,
\qquad
R_1(y):=f(g(y))-y,
\]
and their first-order changes
\[
W_0(x)
:=
Dg(f(x))U_f(x)+U_g(f(x)),
\]
\[
W_1(y)
:=
Df(g(y))U_g(y)+U_f(g(y)).
\]
Taylor expansion gives
\[
g_\alpha(f_\alpha(x))-x
=
R_0(x)-\alpha W_0(x)+O(\alpha^2)
\]
and
\[
f_\alpha(g_\alpha(y))-y
=
R_1(y)-\alpha W_1(y)+O(\alpha^2).
\]
Consequently,
\[
\mathcal C(f_\alpha,g_\alpha)
=
\mathcal C(f,g)
-
2\alpha\Gamma(f,g;U_f,U_g)
+
O(\alpha^2),
\]
where
\[
\Gamma(f,g;U_f,U_g)
:=
\langle R_0,W_0\rangle_{L^2(\mu_0)}
+
\langle R_1,W_1\rangle_{L^2(\mu_1)}.
\]
Thus \(\Gamma>0\) means that the perturbation decreases cycle loss to first
order, whereas \(\Gamma<0\) means that cycle consistency opposes the
perturbation.

For
\[
E_{\mathrm{mis}}(f,g;U_f,U_g)
:=
\|W_0\|_{L^2(\mu_0)}^2
+
\|W_1\|_{L^2(\mu_1)}^2,
\]
Cauchy--Schwarz gives
\[
|\Gamma(f,g;U_f,U_g)|
\leq
\sqrt{
\mathcal C(f,g)
E_{\mathrm{mis}}(f,g;U_f,U_g)
}.
\]
Hence, along any fixed perturbation, the first-order cycle contribution is
small near a low-cycle-loss pair. In the special case of exact inverse
compatibility, \(R_0=R_1=0\), so \(\Gamma=0\) and
\[
\mathcal C(f_\alpha,g_\alpha)
=
\alpha^2E_{\mathrm{mis}}(f,g;U_f,U_g)
+
O(\alpha^3).
\]
Exact inverse compatibility is therefore not assumed; it merely recovers the
previous quadratic special case.

Combining the three expansions for
\[
\mathcal E_{\lambda,m}
=
\lambda\mathcal A
+
\widehat{\mathcal B}_m
+
\mathcal C
\]
yields
\[
\begin{aligned}
\mathcal E_{\lambda,m}(f_\alpha,g_\alpha)
-
\mathcal E_{\lambda,m}(f,g)
&=
\alpha
\left[
b_m(f,g;U_f,U_g)
-
\lambda a(f,g;U_f,U_g)
-
2\Gamma(f,g;U_f,U_g)
\right]\\
&\quad+o(\alpha).
\end{aligned}
\]
Therefore, if
\[
\lambda a(f,g;U_f,U_g)
+
2\Gamma(f,g;U_f,U_g)
>
b_m(f,g;U_f,U_g),
\]
then the complete neural-JSD population objective strictly decreases for all
sufficiently small \(\alpha>0\). Any pair admitting such a direction cannot
be a local minimizer or a stationary point with respect to all admissible
directions.

This condition separates the three local forces. The term
\(\lambda a\) favors lower quadratic transport cost, \(b_m\) measures the
change in endpoint matching, and \(\Gamma\) measures whether the same update
improves or worsens cycle consistency. When \(\Gamma<0\), action and cycle
consistency compete; when \(\Gamma>0\), they cooperate. Decreasing
\(\lambda\) correspondingly shifts the local balance from transport cost
toward endpoint matching and cycle consistency.

This is a local population-level calculation with optimized discriminators.
It does not imply that every nonoptimal pair admits a decreasing direction,
nor does it establish convergence of finite-sample alternating or stochastic
training. Global recovery of the Brenier pair remains the content of
Theorem~\ref{thm:neural_jsd_ot_recovery}.

\section{Experiments}\label{sec:experiments}
We evaluate the proposed method, CyclOT, on five different datasets, spanning synthetic geometry, handwritten images, natural images, and scientific data: \emph{Swiss roll}, \emph{MNIST}, \emph{CelebA}, \emph{single-cell data}, and \emph{chest X-ray images}. The experiments are designed to test three aspects of the method that are central to our claims:
(i) \emph{marginal matching}, namely whether the learned pushforward matches the target distribution;
(ii) \emph{transport structure}, namely whether the learned map remains short, regular, and approximately invertible; and
(iii) \emph{intermediate geometry}, namely whether the induced interpolation path is more coherent than that of standard baselines.

\subsection{Experimental protocol \& interpretation of metrics}
\label{subsec:experimental_protocol}

We evaluate the three complementary roles represented in the population
objective
\[
\mathcal E_{\lambda,m}
=
\lambda\mathcal A
+
\widehat{\mathcal B}_m
+
\mathcal C:
\]
endpoint marginal matching, quadratic transport cost, and approximate inverse
consistency. Because the true Brenier maps are unavailable in the real-data
experiments and there exists a trade-off between the transportation loss and the marginal distribution matching loss in practice, no single reported metric measures OT-map recovery performance. We therefore use several complementary diagnostics and interpret them jointly.

\paragraph{Comparison methods.}
We compare CyclOT with entropic OT based on Sinkhorn iterations~\citep{cuturi2013sinkhorn},
Neural OT~\citep{korotin2023neural}, ICNN OT~\citep{makkuva2020optimal}, rectified flow with reflow~\citep{liu2022rectified}, and mini-batch OT-guided
conditional flow matching (OT-CFM)~\citep{tong2023improving}. Sinkhorn is treated as a finite-sample,
transductive coupling reference that does not define an out-of-sample map. The
remaining methods learn global neural maps or ODE transports.

\begin{itemize}[leftmargin=*]
    \item \textbf{CyclOT}:\ a bidirectional learned transport method with forward and backward maps, dynamic path matching, cycle consistency, and an annealed quadratic displacement penalty.
    \item \textbf{Entropic OT}: a Sinkhorn solver. For image experiments, we use the entropic coupling to obtain hard matches by selecting the target point with largest transport mass for each source point.
    \item \textbf{Neural OT}: a learned deterministic map trained through a neural optimal transport saddle objective.
    \item \textbf{ICNN OT}: an input-convex neural network baseline that learns convex Kantorovich potentials and defines the transport map as a potential gradient.
    \item \textbf{Rectified flow with reflow}: an iterative rectified-flow baseline that repeatedly retrains the velocity field using endpoint pairs induced by the previously learned flow, thereby straightening the learned transport trajectories.
    \item \textbf{Mini-batch OT-guided flow matching}: an OT-CFM-style baseline that first computes mini-batch Sinkhorn pairings and then trains a velocity field on the induced straight-line paths.
\end{itemize}

\paragraph{Evaluation Metrics}
To evaluate the numerical performance of comparison methods, we adopt the following metrics among the baseline datasets, when appropriate:

\begin{itemize}[leftmargin = *]
    \item \textbf{Empirical transport cost}: For image experiments, the reported cost is the average squared displacement between a source sample and its transported output. A lower value means that the learned map moves samples by a shorter distance. This quantity estimates the forward directional component of the quadratic action \(\mathcal A\). \textit{It is not, by itself, a measure of OT-map quality. In particular, a map that remains close to the source distribution can have low cost while failing to reach the target distribution. We therefore interpret transport cost together with endpoint-matching metrics.}
    \item \textbf{Target-domain rate and target probability score}: The target-domain rate (TDR) is the fraction of transported samples classified as belonging to the target domain, while the target probability score (TPS) is the average classifier probability assigned to the target domain. Higher values indicate stronger classifier-level target-domain transfer, which serves as a practical proxy of less mode collapse.
    \item \textbf{Feature-distribution discrepancy}: FID~\cite{heusel2017gans} compares transported and target samples through their empirical means and covariances in a pretrained feature space. A lower FID indicates closer agreement of these feature distributions. MMD~\cite{JMLR:v13:gretton12a} compares transported and target samples through kernel mean embeddings, with a lower empirical \(\operatorname{MMD}^2\) indicating closer distributional agreement in the chosen feature space and kernel family. Both quantities are finite-sample, representation-dependent proxies for marginal distribution matching performance.
    \item \textbf{Class coverage}: For MNIST, Class-TV measures the total variation distance between the predicted target-class proportions of transported samples and those of real target samples. A lower value indicates better coverage of digits \(5,\ldots,9\). Because Class-TV is computed conditionally on samples classified into the target domain, it must be interpreted jointly with TDR: a method may have balanced proportions among a small number of successful outputs while failing to transport most source samples.
    \item \textbf{Local mixing}: Enrich-\(k100\) compares local neighborhoods in the pooled transported and target sample sets. When the two sets have an equal size that is sufficiently large, a value near \(0.5\) indicates local mixing. We therefore report $\left|\operatorname{Enrich}_{100}-0.5\right|$, for which a smaller value is better. This measures local mixing in the chosen representation and complements global statistics such as FID and MMD.
    \item \textbf{Perturbation-signature correlation}: For the single-cell task, signature correlation measures agreement between the population-level feature shift produced by the learned map and the observed control-to-treatment feature shift. A higher value indicates better recovery of the average perturbation pattern. 
    \item \textbf{Cycle consistency and map collapse}: Forward and reverse cycle errors measure how close is $g_\phi(f_\theta(x))$ and $f_\theta(g_\phi(y))$ to $x$ and $y$ on average, respectively. In the experiments below, the displayed cycle reconstructions provide a qualitative version of this diagnostic. They do not establish global injectivity or exclude collapse on unobserved or low-probability regions.
\end{itemize}

\paragraph{Computational measures.}
Training time measures wall-clock optimization cost under the stated hardware
and software configuration. Image exposures count the number of source or
target training examples accessed during optimization and therefore measure
sample usage rather than transport quality. These quantities should be
interpreted separately from statistical performance.

\paragraph{Connection with the theory.}
The empirical quantities mirror the theoretical roles of the three objective terms numerically with a goal to assess whether the observed tradeoffs are consistent with the mechanism underlying Theorem~\ref{thm:neural_jsd_ot_recovery} and the structural results of Section~\ref{sec:theory}. Distributional metrics probe endpoint agreement, empirical squared displacement probes the action term, and held-out cycle error probes approximate inverse compatibility. Their joint use is essential: low displacement without endpoint agreement can favor a near-identity map, while low marginal agreement or cycle error is insufficient in selecting the Brenier map without effectively penalizing the displacement loss.

\subsection{Swiss roll: geometric fidelity in low dimensions}
\label{subsec:swissroll}

\paragraph{Setup.}
We first consider transport from a two-dimensional Gaussian source to a Swiss-roll target. This experiment isolates the geometric aspect of the problem: both endpoint distributions and the induced transport paths can be visualized directly in the data space. The goal is not only to match the target marginal, but also to learn a transport map whose pointwise displacements are coherent with the geometry of the target manifold.

\paragraph{Results.}
Figure~\ref{fig:swissroll_qualitative} provides a qualitative visualization of the learned bidirectional transport. Marginal matching shows that the forward map transports Gaussian samples to the Swiss-roll distribution, while the backward map returns target samples to the Gaussian source. Cycle reconstructions approximately preserve target samples, indicating approximate inverse consistency between the two maps. The interpolation paths from source samples $z$ to their endpoints $f_\theta(z)$ illustrate the displacement field induced by the learned forward map.

\begin{figure}[t]
    \centering

    \subfigure[Pointwise interpolation paths $z\to f_\theta(z)$.]{
        \includegraphics[width=0.47\linewidth]{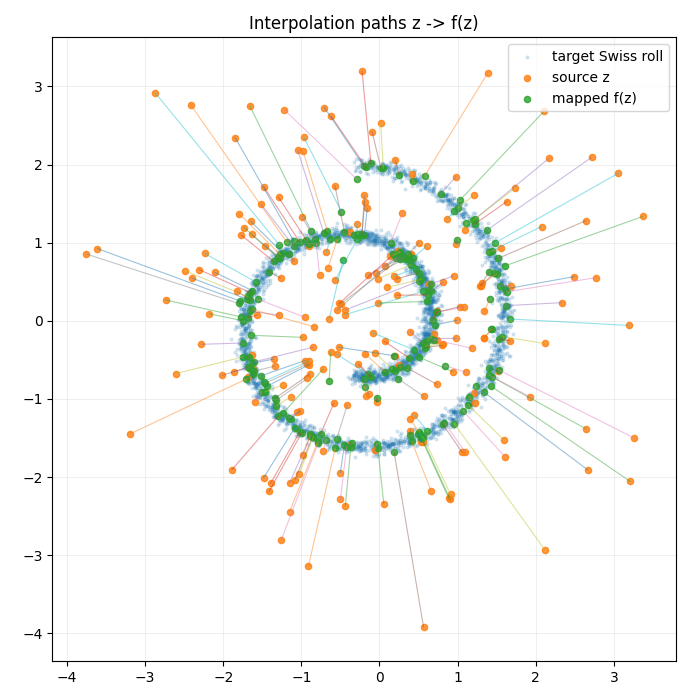}
        \label{fig:swissroll_paths_panel}
    }
    \hfill
    \subfigure[Marginal matching and cycle reconstruction.]{
        \includegraphics[width=0.47\linewidth]{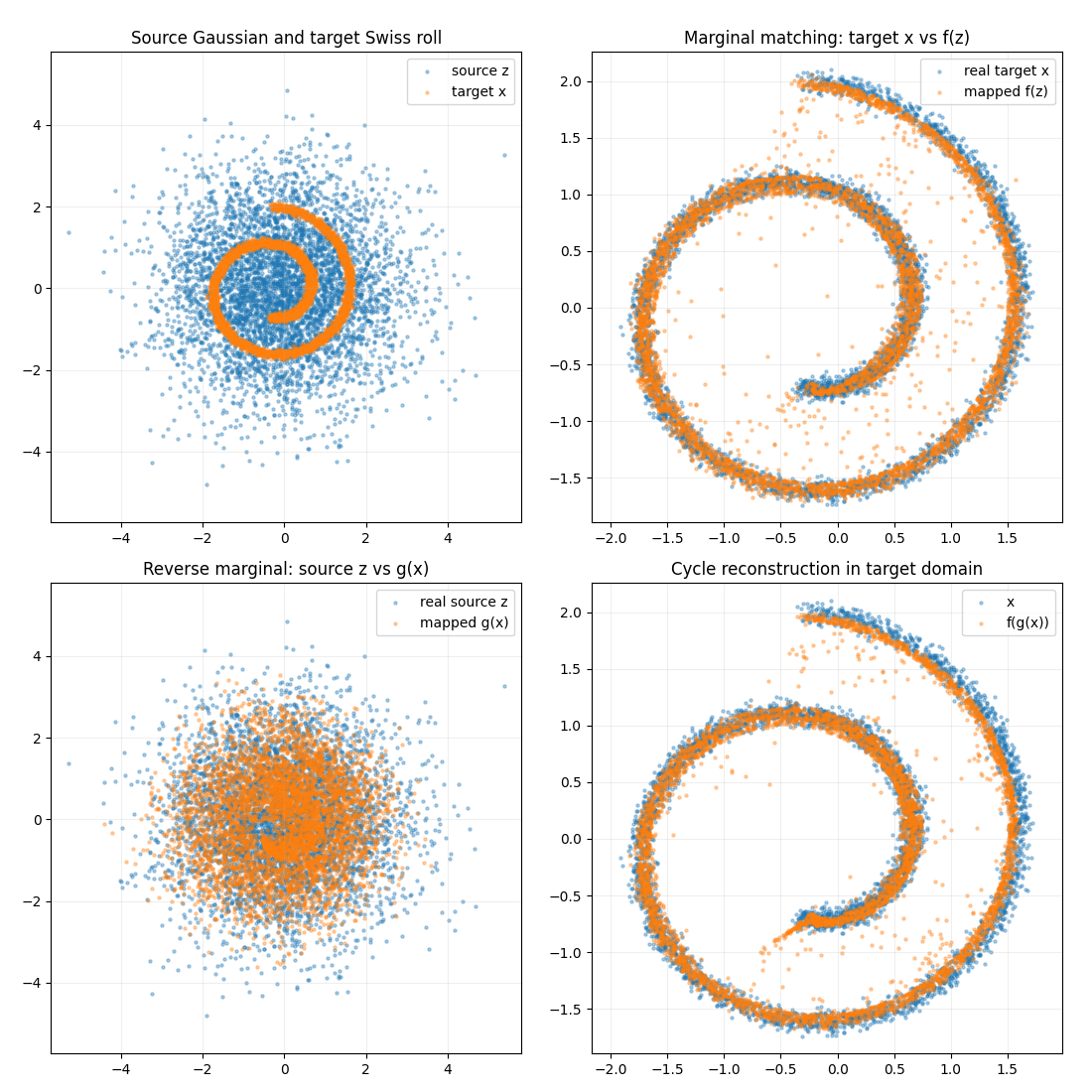}
        \label{fig:swissroll_panel}
    }

    \caption{
    Qualitative visualization of the Gaussian-to-Swiss-roll experiment. The left panel shows pointwise map-induced interpolation paths from
    source samples to transported endpoints. The right panel shows forward
    marginal matching, reverse marginal matching, and target-domain cycle
    reconstruction.
    }
    \label{fig:swissroll_qualitative}
\end{figure}

\subsection{MNIST: transport from digits \(0\text{--}4\) to digits \(5\text{--}9\)}
\label{subsec:mnist}

We next study unpaired transport between two handwritten digit domains. The source distribution consists of MNIST~\cite{lecun1998gradient} digits \(0,1,2,3,4\), and the target distribution consists of digits \(5,6,7,8,9\). This experiment tests whether a method can learn a global map that transports samples into the target digit domain while maintaining controlled displacement and avoiding trivial low-cost solutions that remain close to the source images.

\paragraph{Setup.}
All neural methods use MNIST images normalized to \([-1,1]\). The proposed method, CyclOT, learns a forward map \(f_\theta\) from digits \(0\text{--}4\) to digits \(5\text{--}9\) and a backward map \(g_\phi\) from digits \(5\text{--}9\) to digits \(0\text{--}4\), both implemented as residual CNN encoder--decoder maps.

For entropic OT, we solve a finite-sample Sinkhorn problem between sampled source images and sampled target images using squared Euclidean pixel cost. Given the entropic coupling \(P\), we form a hard matching by assigning each source image \(x_i\) to the target image \(y_{j^\star(i)}\), where
\[
    j^\star(i)=\arg\max_j P_{ij}.
\]
This baseline is transductive: it provides a strong finite-sample OT reference but does not learn a global map for unseen source images. For neural OT, ICNN OT, rectified flow, and mini-batch OT-guided flow matching, we train global maps or ODE flows and evaluate them on held-out source test images.

\paragraph{Evaluation.}
Across experiments, we report metrics that separately quantify transport
structure, target-population matching, and computational efficiency. For the MNIST transport task, these include the per-sample squared transport cost, target-domain rate (TDR), target probability score (TPS), MNIST-feature FID, class-distribution total variation distance (Class-TV), and training time. Detailed definitions and computation procedures are provided in Appendix~\ref{app:evaluation_metrics}.

\paragraph{Results.}
Table~\ref{tab:mnist_results} reports the quantitative results. Among learned global maps, CyclOT achieves the best overall tradeoff between target-domain generation quality and transport structure. It attains a target-domain rate of \(0.957\), indicating that most transported images are classified as digits \(5\text{--}9\), while also achieving the lowest MNIST-feature FID among learned methods. The class-distribution TV distance is also the smallest among learned methods, suggesting that the transported samples better cover the target digit classes rather than collapsing to a small subset of target digits.

Neural OT achieves a lower transport cost but worse FID and class-distribution discrepancy, suggesting that lower displacement alone does not necessarily imply better target-distribution matching. The Sinkhorn baseline achieves high target-domain validity by directly matching to observed target samples, but does not define a learned global transport map.

Figure~\ref{fig:mnist_results} presents qualitative results on MNIST. The first row shows the source digits (0-4), while the second row shows their corresponding mapped digits in the target classes (5-9). Additional visualization results are provided in Appendix~\ref{app:mnist_extra}. The visual results are consistent with the quantitative metrics: CyclOT produces samples that are visually closer to target digits \(5\text{--}9\), while the lower-cost flow baselines tend to preserve too much of the original \(0\text{--}4\) digit structure.

\begin{table}[t]
\centering
\small
\setlength{\tabcolsep}{4pt}
\begin{tabular}{lccccccc}
\toprule
Method
& Global map?
& Cost $\downarrow$
& TDR $\uparrow$
& TPS $\uparrow$
& FID $\downarrow$
& Class-TV $\downarrow$ 
& Train Time (min) $\downarrow$ \\
\midrule
Sinkhorn
& No
& 245.201
& 0.992
& 0.984
& 93.773
& 0.072 
& -- \\
CyclOT
& Yes
& 212.43
& \textbf{0.957}
& \textbf{0.954}
& \textbf{92.269}
& \textbf{0.039} 
& \textbf{7} \\
Neural OT
& Yes
& 133.68
& 0.956
& 0.952
& 120.45
& 0.130 
& 198 \\
ICNN OT
& Yes
& 210.82
& 0.805
& 0.749
& 1565.65
& 0.162 
& 43 \\
Rectified flow
& Yes
& 246.18
& 0.565
& 0.567
& 1282.97
& 0.171
& 29 \\
OT-CFM
& Yes
& \textbf{100.87}
& 0.390
& 0.393
& 841.36
& 0.241 
& 32 \\
\bottomrule
\end{tabular}
\caption{MNIST \(0\text{--}4\to5\text{--}9\) transport. We report transport cost, target-domain rate (TDR), target probability score (TPS), MNIST-feature FID, and total variation distance between predicted and real target-class distributions over digits \(5,\ldots,9\). Sinkhorn hard matching is a finite-sample transductive baseline and does not define a learned global map. Among learned global maps, CyclOT achieves the best MNIST-feature FID and target class-distribution matching, while maintaining high target-domain validity.}
\label{tab:mnist_results}
\end{table}

\begin{figure}[t]
    \centering
    \includegraphics[width=.95\linewidth]{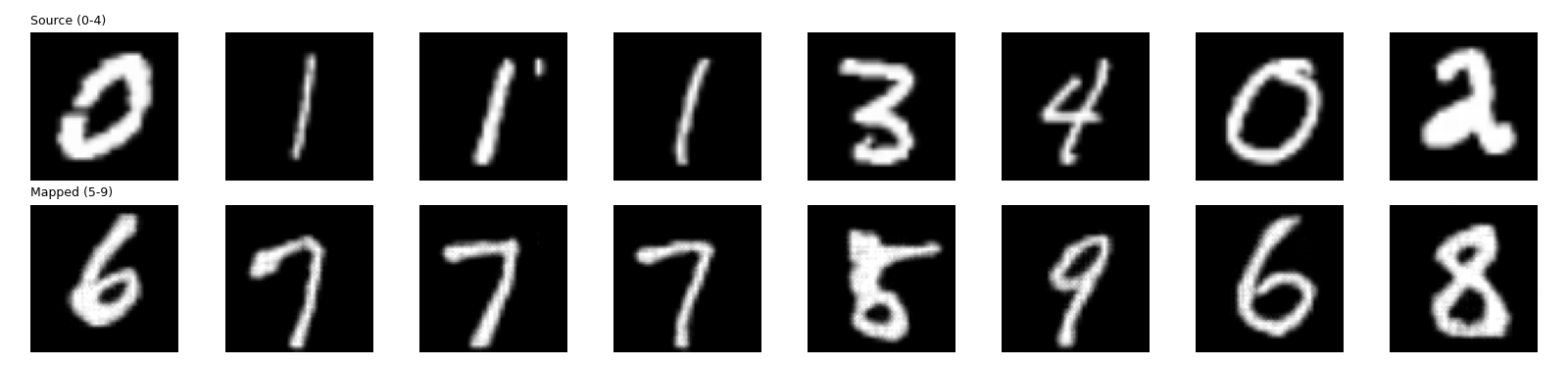}
    \caption{MNIST \(0\text{--}4\to5\text{--}9\) transport.}
    \label{fig:mnist_results}
\end{figure}

\subsection{CelebA: bidirectional transport on natural images}
\label{subsec:celeba}

We then consider a transport task on CelebA~\cite{liu2015deep}, mapping between female and male face distributions. Compared with MNIST, this setting is substantially more challenging: the data lie on a richer natural-image manifold, and successful transport must preserve shared facial structure while modifying domain-level attributes.

\paragraph{Setup.}
We partition CelebA using the \texttt{Male} attribute, taking female faces as the source domain and male faces as the target domain. The forward map $f_\theta$ learns to transport female faces to the male domain, while the backward map $g_\phi$ learns the reverse mapping from male faces to the female domain. Images are center-cropped, resized to \(64\times64\), and normalized to \([-1,1]\). The forward and backward maps are implemented as residual encoder-decoder CNNs with skip connections.

\paragraph{Evaluation.}
Across all experiments, we present qualitative results and quantitatively compare CyclOT with Neural OT and ICNN OT using seven evaluation metrics: transport cost, Fréchet Inception Distance (FID), target domain rate(TDR), Target Probability Score(TPS), Enrich-k100 and Female Image Exposures(FIE), and Male Image Exposures(MIE). The results are reported in Table~\ref{tab:celeba_quantitative_comparison}. The TDR measures the proportion of transported source-female images that are classified as male, while the TPS measures the average predicted probability assigned to the male class. Both metrics are computed using the FairFace classifier~\citep{DBLP:journals/corr/abs-1908-04913}. As a real-data reference, we also report FairFace results on real CelebA female and male images in Table \ref{tab:fairface_baseline}. Female/Male Image Exposures denote the total number of times female/male images from the CelebA training set are accessed during training. Detailed definitions and computation procedures of each metric are provided in Appendix \ref{app:evaluation_metrics}.

\paragraph{Results.}
Figure~\ref{fig:face_cycle} qualitatively demonstrates the learned
bidirectional transport on CelebA. The forward and backward maps modify
gender-related visual attributes while largely preserving pose, facial layout, and background. The generated images do not exhibit obvious visual artifacts, suggesting that the learned maps remain stable under cross-domain translation. The cycle reconstructions $g_\phi(f_\theta(x))$ and $f_\theta(g_\phi(y))$ are visually close to the original samples, indicating that the cycle-consistency term helps retain sample-level information and discourages many-to-one collapse. Although some fine details are blurred or slightly smoothed, the reconstructions preserve the main semantic and structural content of the inputs. Additional experimental results are presented in Appendix~\ref{app:celeba_extra}.

Table~\ref{tab:celeba_quantitative_comparison} shows that Neural OT achieves a lower transport cost, whereas CyclOT achieves higher TDR and TPS. Moreover, our results are closer to the FairFace baseline reported in Table~\ref{tab:fairface_baseline}: The TDR is 92.0\% versus baseline 92.6\%, and the TPS is 0.8602 $\pm$ 0.20 versus baseline 0.8893 $\pm$ 0.21. In addition, Figure~\ref{fig:TDR_time} further shows that CyclOT reaches a higher TDR more rapidly during training.

Figure~\ref{fig:face_comparison} shows that, under the same training budget, CyclOT produces stronger and more consistent female-to-male translations than Neural OT while preserving pose, background, and coarse facial structure. In contrast, Neural OT often retains more source-domain appearance, despite achieving a lower transport cost. This comparison highlights that transport cost alone is not sufficient for learning an effective cross-domain map: when the target marginal is not enforced strongly enough, minimizing displacement can favor mappings that remain close to the source distribution. CyclOT instead achieves stronger target-domain alignment while preserving the overall structure of the input samples.

Figure~\ref{fig:training_curve_NOT} shows that Neural OT becomes unstable as training progresses and deteriorates after reaching an intermediate optimum. In contrast, Figure~\ref{fig:training_curve_Ours} shows more stable training dynamics. Under our BCE-type GAN objective, the balanced discriminator equilibrium corresponds to a generator loss of $\log2$ and a discriminator loss of $2\log2$ under the summed-BCE convention. In addition, the transport cost curve shows that CyclOT converges smoothly during training.

\begin{figure}[t]
    \centering
    \subfigure[Training dynamics of Neural OT on the CelebA female-to-male translation task. Here, $T$ denotes the learned transport map, and $f$ denotes the learned Kantorovich potential.]{
        \includegraphics[width=0.47\textwidth]{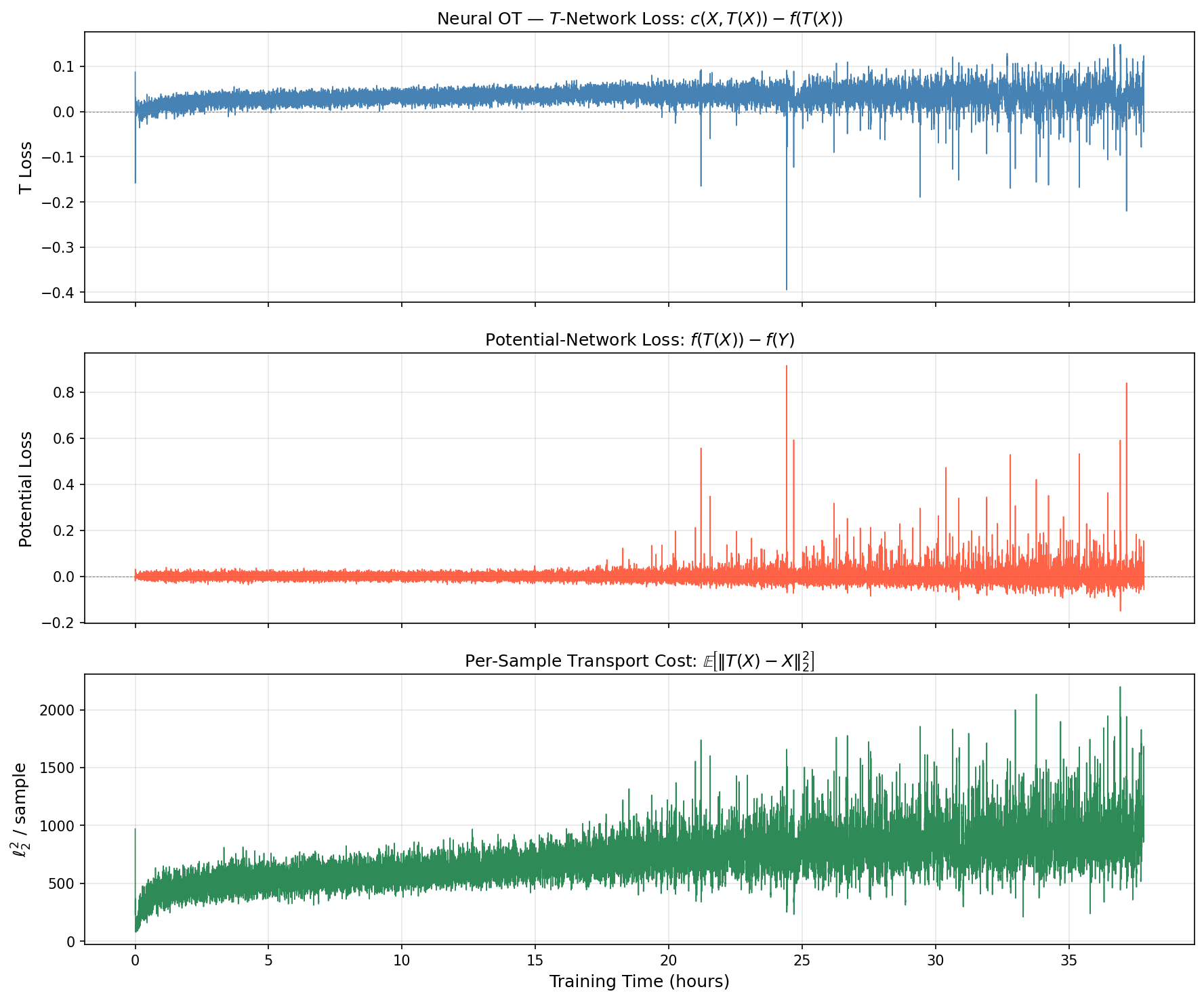}
        \label{fig:training_curve_NOT}
    }
    \hfill
    \subfigure[Training dynamics of CyclOT on the CelebA female-to-male translation task.]{
        \includegraphics[width=0.47\textwidth]{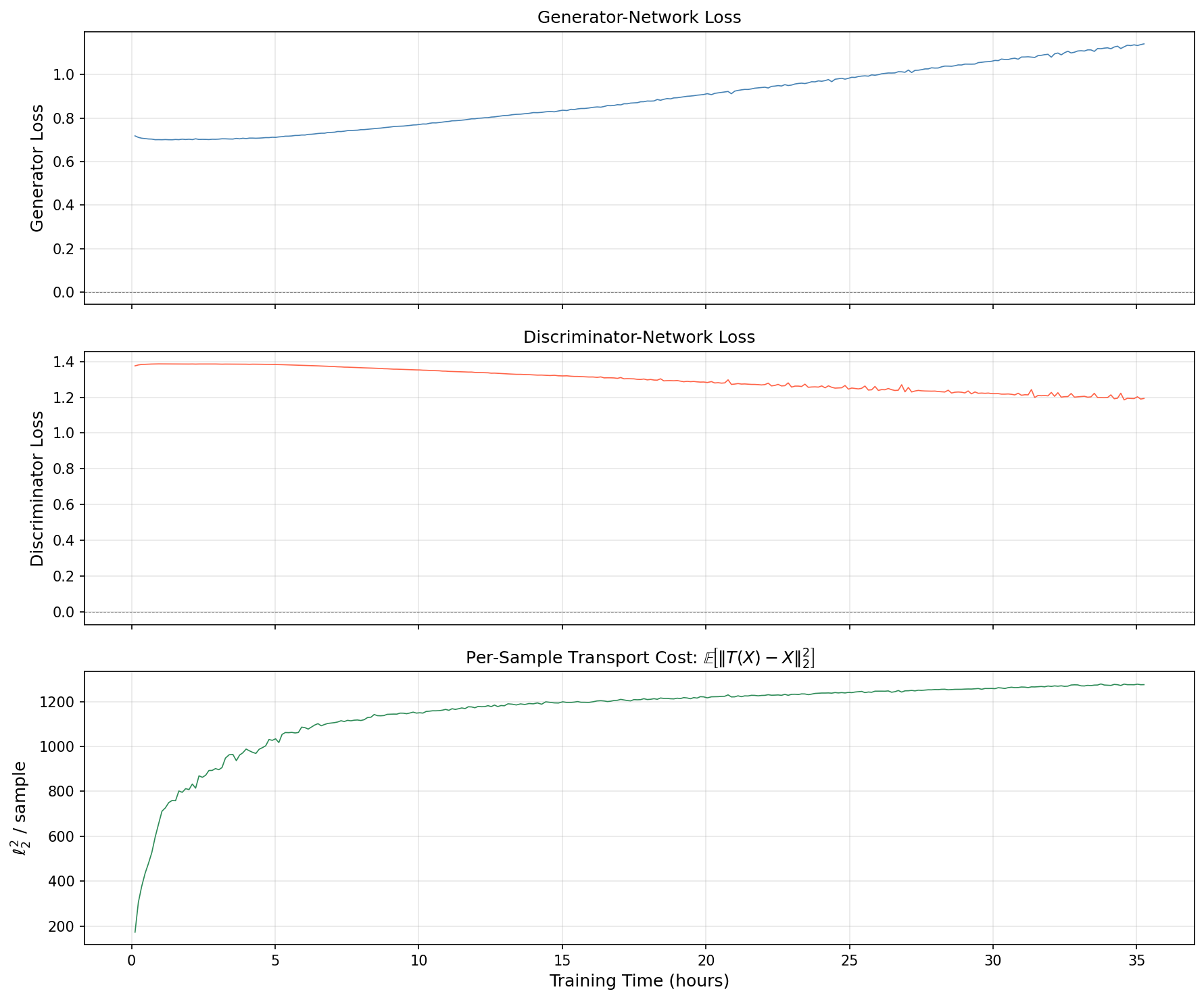}
        \label{fig:training_curve_Ours}
    }

    \caption{Training dynamics on the CelebA female-to-male translation task for Neural OT and CyclOT.}
    \label{fig:training_curves_celebA}
\end{figure}

\begin{table}[t]
    \centering
    \small
    \setlength{\tabcolsep}{4pt}
    \begin{tabular}{lccccccc}
    \toprule
        Method
        & Cost $\downarrow$
        & FID $\downarrow$
        & TDR $\uparrow$
        & TPS $\uparrow$
        & $|\mathrm{Enrich\mbox{-}k100}-0.5| \downarrow$
        & FIE $\downarrow$
        & MIE $\downarrow$\\
        \midrule
        ICNN OT
        & 2266.7034
        & 305.2227
        & 46.0\%
        & $0.4851 \pm 0.25$
        & 0.3759
        & 3.2M
        & 35.2M\\
        Neural OT
        & \textbf{687.9927}
        & \textbf{17.8291}
        & 79.9\%
        & $0.7447 \pm 0.27$
        & \textbf{0.0228}
        & 17.6M
        & 1.6M\\
        CyclOT
        & 1069.7815
        & 25.4478
        & \textbf{92.0\%}
        & \textbf{0.8602 $\pm$ 0.20}
        & 0.0357
        & 2.73M
        & 2.73M\\
        \bottomrule
    \end{tabular}
    \caption{Quantitative comparison between Neural OT and CyclOT on the CelebA female-to-male translation task.}
    \label{tab:celeba_quantitative_comparison}
\end{table}

\begin{table}[t]
    \centering
    \small
    \setlength{\tabcolsep}{6pt}
    \begin{tabular}{lccc}
        \toprule
        Input set
        & TPS
        & Predicted male
        & Predicted female \\
        \midrule
        Real source female
        & 0.1278 $\pm$ 0.23
        & 479 (9.6\%)
        & 4521 (90.4\%) \\
        Real target male
        & 0.8893 $\pm$ 0.21
        & 4629 (92.6\%)
        & 371 (7.4\%) \\
        \bottomrule
    \end{tabular}
    \caption{
        FairFace classifier baseline results on real CelebA source-female
        and target-male images. Each evaluation set contains 5,000 images.
        The classifier correctly predicts 90.4\% of the source-female images
        as female and 92.6\% of the real target-male images as male.
    }
    \label{tab:fairface_baseline}
\end{table}

\begin{figure}[t]
    \centering

    \subfigure[Female-to-male-to-female cycle.]{
        \includegraphics[width=\textwidth]{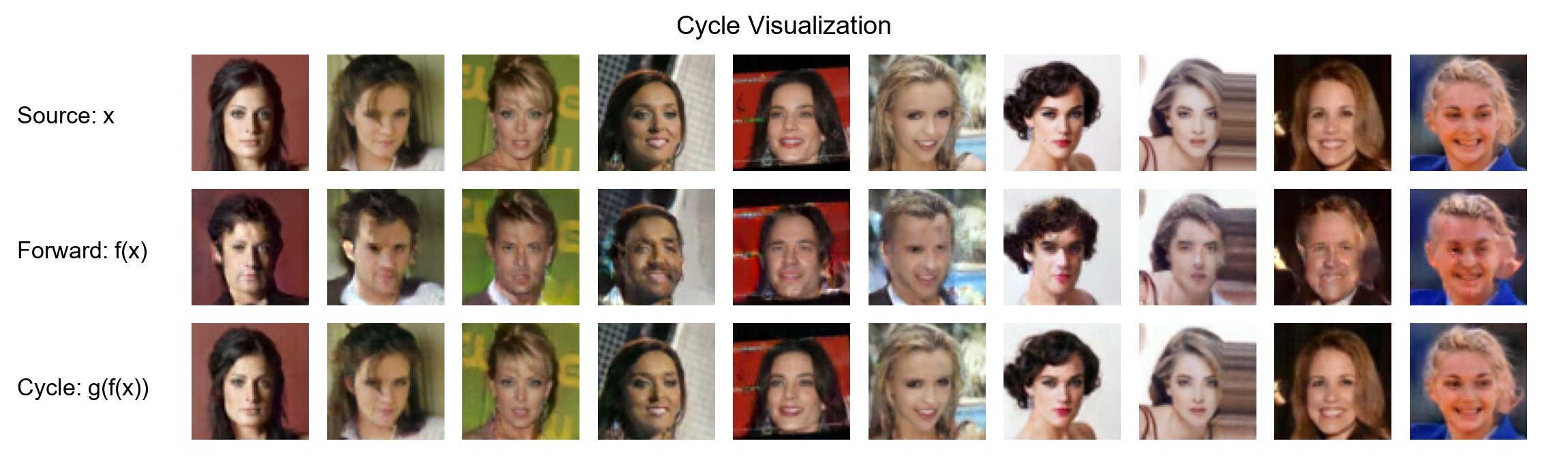}
        \label{fig:face_cycle_f2m}
    }

    \vspace{0.5em}

    \subfigure[Male-to-female-to-male cycle.]{
        \includegraphics[width=\textwidth]{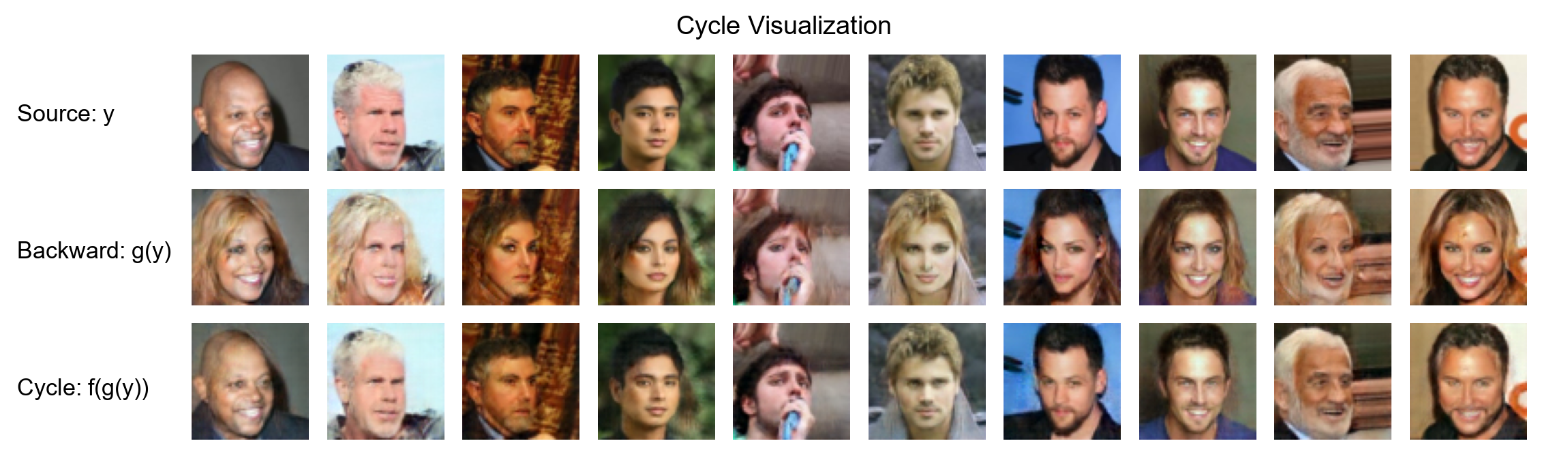}
        \label{fig:face_cycle_m2f}
    }

    \caption{Cycle visualizations on CelebA. The first block shows source female samples $x$, forward mapped samples $f_\theta(x)$, and cycle reconstructions $g_\phi(f_\theta(x))$. The second block shows source male samples $y$, backward mapped samples $g_\phi(y)$, and cycle reconstructions $f_\theta(g_\phi(y))$.}
    \label{fig:face_cycle}
\end{figure}

\begin{figure}[t]
    \centering
    \includegraphics[width=\textwidth]{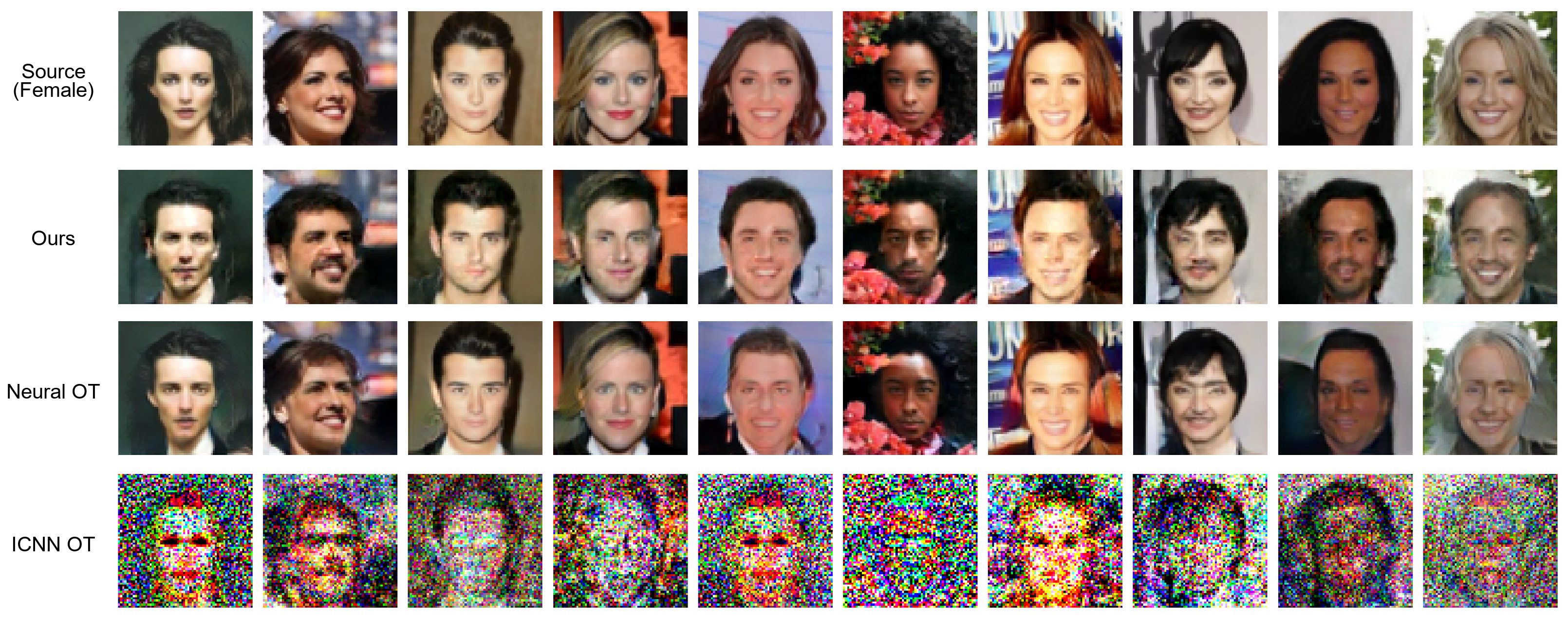}
    \caption{Comparison results for female-to-male face translation on CelebA. All three methods are trained for 4 hours under the same experimental setting.}
    \label{fig:face_comparison}
\end{figure}

\begin{figure}[t]
    \centering
    \includegraphics[width=\textwidth]{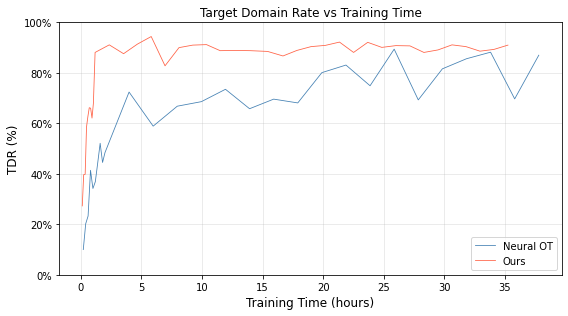}
    \caption{Target Domain Rate (TDR) versus training time for Neural OT and CyclOT.}
    \label{fig:TDR_time}
\end{figure}

\subsection{Single-cell data: transport on scientific manifolds}
\label{subsec:cell}

\paragraph{Setup.}
We evaluate the methods on a single-cell 4i dataset from Bunne et al.~\citep{bunne2023learning} using the trametinib perturbation experiment. We take cells from the control condition as the source domain and trametinib-treated cells as the target domain. Each cell is represented by the $48$ molecular features, and no cell-level correspondence between control and treated cells is assumed. The goal is therefore to learn an unpaired transport map from the control population to the trametinib-treated population. The forward and backward maps are implemented as residual MLPs in the single-cell feature space.

\paragraph{Evaluation.}
We evaluate the single-cell transport task using metrics that capture both
transport regularity and target-population matching. Specifically, we report the \emph{transport cost}, \emph{MMD loss}, and \emph{Enrich-k100}. Detailed definitions and implementation details for these metrics are provided in the appendix \ref{app:evaluation_metrics}.

\paragraph{Results.}
Table~\ref{tab:cell_4i_trametinib} summarizes the control-to-trametinib transport results on the 4i dataset. Among the learned methods, CyclOT achieves the lowest transport cost, the lowest MMD, and the highest signature correlation, while also attaining competitive local neighborhood mixing. ICNN OT achieves the best Enrich-k100 score. Rectified flow and OT-guided flow matching exhibit substantially larger transport costs together with weaker target-population matching. Sinkhorn hard matching is included as a finite-sample transductive reference and does not define a learned global map.

\begin{table}[t]
\centering
\small
\setlength{\tabcolsep}{4pt}
\begin{tabular}{lcccc}
\toprule
Method 
& Transport Cost $\downarrow$ 
& MMD $\downarrow$
& $|\mathrm{Enrich\mbox{-}k100}-0.5| \downarrow$
& Signature Corr $\uparrow$ \\
\midrule
Sinkhorn
& 5.2320 
& -- 
& -- 
& --  \\

CyclOT 
& \textbf{2.8352} 
& \textbf{0.002472}
& 0.00729
& \textbf{0.99093}\\

Neural OT 
& 3.1925
& 0.002879
& 0.01835
& 0.98936 \\

ICNN OT 
& 3.2515 
& 0.002623
& \textbf{0.00456}
& 0.98698 \\

Rectified flow
& 9.4423
& 0.004876
& 0.02044
& 0.94781 \\

OT-CFM
& 7.0150
& 0.012698
& 0.11412
& 0.86522 \\

\bottomrule
\end{tabular}
\caption{
4i trametinib single-cell perturbation transport. 
}
\label{tab:cell_4i_trametinib}
\end{table}

\subsection{Chest X-ray translation: transport on medical images}
\label{subsec:cxr}

\begin{table}[ht]
    \centering
    \small
    \setlength{\tabcolsep}{4pt}
    \begin{tabular}{lcccc}
    \toprule
        Method
        & Cost $\downarrow$
        & $|\mathrm{Enrich\mbox{-}k100}-0.5| \downarrow$
        & TPS $\uparrow$ 
        & \makecell{Training Time\\ (hours:minutes) $\downarrow$} \\
        \midrule
        Sinkhorn
        & 5643.2637 & -- & -- & \textbf{$<$0:01} \\
        ICNN OT
        & 65095.5671 & 0.4984 & 0.1513 $\pm$ 0.0409 & 1:01 \\
        Neural OT
        & 1984.1612 & 0.3246 & \textbf{0.4762 $\pm$ 0.1860} & 19:15 \\
        CyclOT
        & \textbf{1717.2522} & \textbf{0.3169} & 0.4098 $\pm$ 0.1873 & 8:41 \\
        \bottomrule
    \end{tabular}
    \caption{Quantitative comparison between ICCN OT, Neural OT, and CyclOT on the chest x-ray healthy-to-pneumonia transition task.}
    \label{tab:chest_quantitative_comparison}
\end{table}

\begin{table}[t]
    \centering
    \small
    \setlength{\tabcolsep}{6pt}
    \begin{tabular}{lc}
        \toprule
        Input set
        & TPS \\
        \midrule
        Real source healthy
        & 0.1876 $\pm$ 0.1382 \\
        Real target pneumonia
        & 0.6254 $\pm$ 0.1902 \\
        \bottomrule
    \end{tabular}
    \caption{
        Pneumonia classifier baseline results on real source (healthy) and target (pneumonia) images.
    }
    \label{tab:pneumonia_baseline}
\end{table}

\paragraph{Setup.}
We further evaluate CyclOT on chest x-ray images to test whether the proposed forward--backward transport framework can operate on high-dimensional medical data. Compared with natural image attribute translation, chest radiographs impose a stronger
structure-preservation requirement: a meaningful transport map should modify the domain-specific radiographic appearance while preserving the underlying anatomical layout, including the lung fields, ribs, clavicles, and global patient pose.

We use data from the RSNA Pneumonia Detection Challenge\footnote{Data is available to download here: https://nihcc.app.box.com/v/ChestXray-NIHCC} for training \cite{wang2017chestx}. The dataset includes a total of 7,106 chest x-rays from patients with pneumonia and 9,790 x-rays from patients without pneumonia. We resize all images to $256\times256$ and use a random subset of 3,000 images from each class for training, and repeat the training three times for each method. The full dataset is used for evaluation.

\paragraph{Evaluation.}
To evaluate the utility of our and other methods at converting x-ray images between the two classes we use a classification model developed by Weng et al.~\cite{chexnetmodel} on the dataset provided by Rajpurkar et al.~\cite{chexnetpaper}. We report the transport cost, Enrich-k100, TPS, and total training time of each method. The results in Table \ref{tab:chest_quantitative_comparison} show the mapping from healthy images to ones with pneumonia. The results for the mapping in the other direction can be found in \ref{app:pneumonia}. 

\paragraph{Results.}
Figure~\ref{fig:cxr_transport} shows representative qualitative results. The first row shows source chest X-ray images $x$, the second row shows transported images $f_\theta(x)$, and the third row shows cycle reconstructions $g_\phi(f_\theta(x))$. The transported samples exhibit visible changes in global radiographic appearance, including contrast, opacity, and local texture patterns, while largely preserving the coarse anatomical geometry of the input images. In particular, the position of the thoracic cavity, lung fields, shoulders, and mediastinal structure remains aligned with the corresponding source image. The cycle reconstructions recover the main image structure and much of the source-domain appearance, suggesting that the learned map is not simply collapsing different inputs to a small set of target-like images.

\begin{figure}[t]
    \centering
    \includegraphics[width=\linewidth]{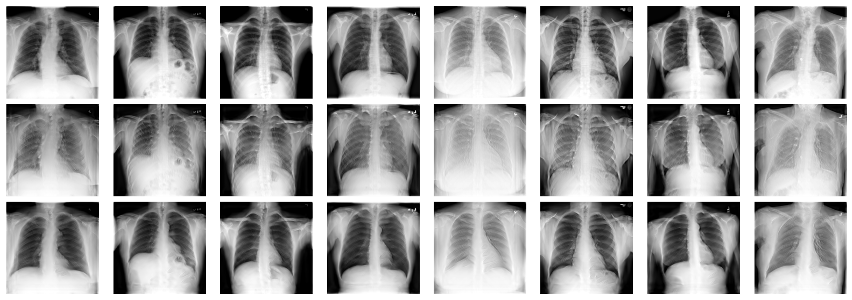}
    \caption{Chest x-ray transport results for CyclOT. The first row shows source (healthy) radiographs $x$, the second row shows the transported samples $f_\theta(x)$, and the third row shows the cycle reconstructions $g_\phi(f_\theta(x))$. The learned transport changes the target-domain radiographic appearance while preserving the coarse anatomical layout.}
    \label{fig:cxr_transport}
\end{figure}

\begin{figure}[t]
    \centering
    \subfigure[Images created by a model trained with ICNN OT.]{
        \includegraphics[width=0.47\textwidth]{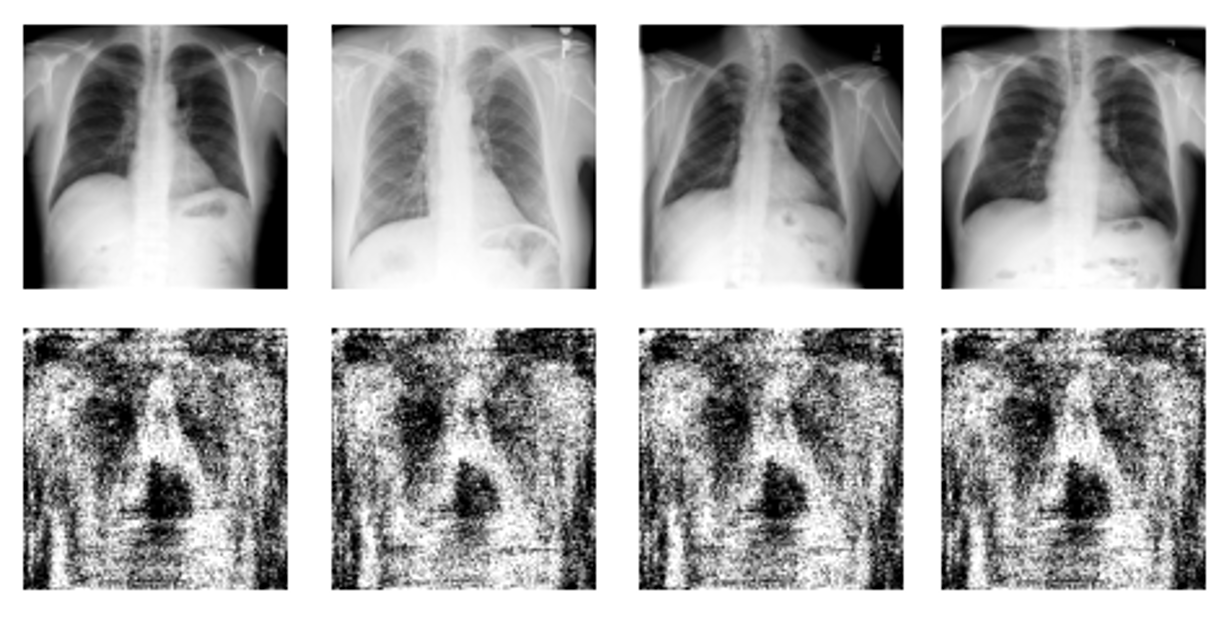}
        \label{fig:icnn_example}
    }
    \hfill
    \subfigure[Images created by a model trained with Neural OT.]{
        \includegraphics[width=0.47\textwidth]{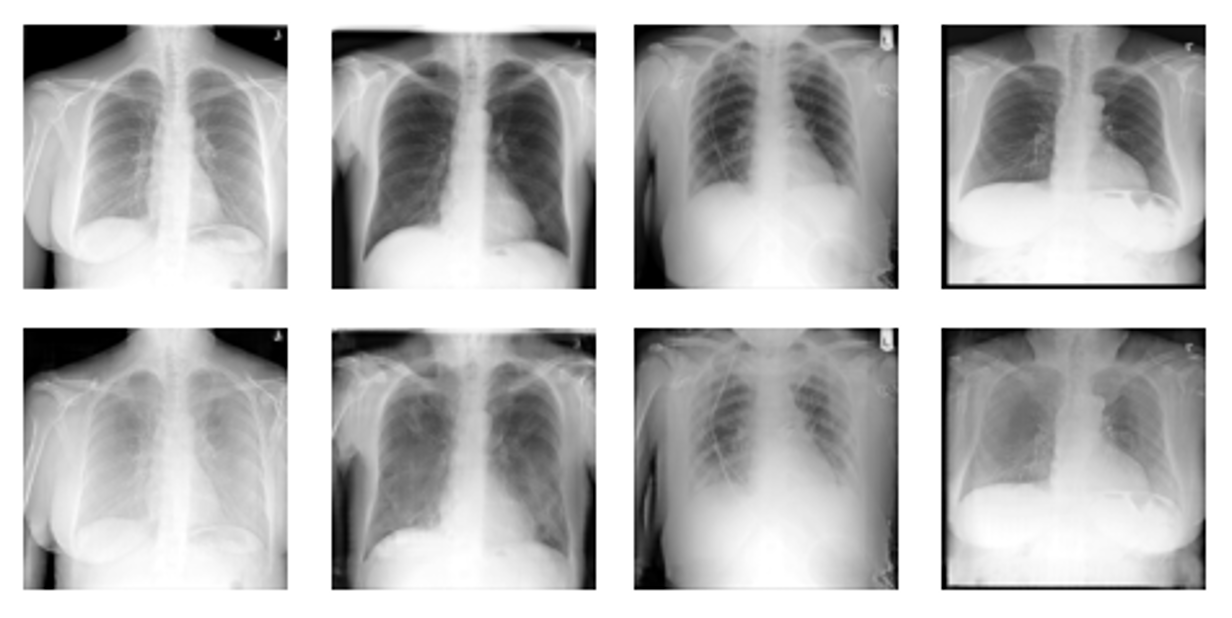}
        \label{fig:neural_example}
    }

    \caption{Chest x-ray transport results for other tested methods. The first row shows source (healthy) radiographs $x$ and the second row shows the transported samples $f_\theta(x)$.}
    \label{fig:cxr_other_methods}
\end{figure}

\subsection{Takeaway}

Taken together, the experiments illustrate a consistent pattern. Entropic OT remains a strong classical baseline for finite-sample marginal alignment, especially in low dimension, but it does not directly provide a learned global map and may require additional projection or regression steps. Neural OT provides expressive learned maps, but without explicit reverse consistency it may admit transport solutions that match target-domain labels while failing to match the full target distribution. ICNN OT encodes the Brenier structure but is difficult to optimize directly in high-dimensional pixel space. Flow-matching baselines provide flexible ODE transports, but without sufficiently informative couplings they may favor low-displacement maps that do not fully reach the target domain. CyclOT combines the flexibility of learned transport with bidirectional structural regularization, leading to maps that are not only distributionally accurate but also more stable and coherent along intermediate paths.

\clearpage
\bibliographystyle{unsrt}
\bibliography{example_paper}

\clearpage
\appendix
\section{Background \& Preliminaries}
\label{app:background}
\subsection{Flow matching \& rectified flow}
\label{subsec:fm_rf}
Flow matching learns a time-dependent vector field whose induced probability flow transports a source distribution to a target distribution~\citep{lipman2022flow}. Specifically, one considers the ODE
\[
    \frac{dZ_t}{dt}=v_\theta(Z_t,t),
    \qquad Z_0\sim\mu_0,
\]
and trains $v_\theta$ using a prescribed reference interpolation between samples from $\mu_0$ and $\mu_1$. A common choice is the linear interpolation
\[
    X_t=(1-t)X_0+tX_1,
    \qquad X_0\sim\mu_0,\quad X_1\sim\mu_1,
\]
where $(X_0,X_1)$ is sampled from a chosen coupling, often the independent coupling. The corresponding reference velocity is
\[
    u(X_0,X_1)=X_1-X_0.
\]
Flow matching then minimizes the regression objective
\[
    \mathcal{L}_{\mathrm{FM}}(\theta)
    =
    \int_0^1
    \mathbb{E}
    \left[
        \left\|
            X_1-X_0-v_\theta(X_t,t)
        \right\|^2
    \right]dt .
\]
At the population level, the minimizer over all measurable vector fields is
\[
    v^\star(z,t)
    =
    \mathbb{E}
    \left[
        X_1-X_0
        \mid
        X_t=z
    \right].
\]
The reference marginals $\mu_t=\mathcal L(X_t)$ then satisfy the continuity equation
\[
    \partial_t \mu_t+\nabla\cdot(\mu_t v_t^\star)=0
\]
in the weak sense. If this continuity equation is well posed for $v^\star$, then the ODE driven by $v^\star$ has the same time marginals as the reference interpolation.

Rectified flow applies this projection iteratively. After learning a velocity field, it constructs new endpoint pairings from the learned flow and retrains on the resulting trajectories. This reflow procedure preserves the endpoint marginals while progressively straightening the learned paths.

\subsection{Independence constraints \& Wasserstein barycenters.}
\label{subsec:bary_indep}
Following Xu and Strohmer~\cite{xu2023fair}, we recall the connection between independence-constrained $L^2$ projection and Wasserstein barycenters. Let $Z\in\{0,1\}$ be a group label with
\[
    \mathbb{P}(Z=k)=\lambda_k,
    \qquad
    \mu_k=\mathcal{L}(X\mid Z=k),
    \qquad k\in\{0,1\}.
\]
Consider the independence-constrained projection problem
\[
    \inf_{\bar X:\,\bar X\perp Z}
    \mathbb{E}\left[\|X-\bar X\|^2\right].
\]
This constraint means that the conditional law of $\bar X$ is the same for both groups. Equivalently, there exists a common distribution $\nu$ such that
\[
    \mathcal{L}(\bar X\mid Z=0)
    =
    \mathcal{L}(\bar X\mid Z=1)
    =
    \nu .
\]

For a fixed common law $\nu$, the optimal way to map the $k$-th group distribution $\mu_k$ to $\nu$ under quadratic distortion has cost
\[
    W_2^2(\mu_k,\nu).
\]
Therefore, the independence-constrained projection problem reduces to choosing the best common law:
\[
    \inf_{\bar X:\bar X\perp Z}
    \mathbb{E}\|X-\bar X\|^2
    =
    \inf_{\nu\in\mathcal{P}_2(\mathbb{R}^d)}
    \sum_{k=0}^1
    \lambda_k W_2^2(\mu_k,\nu).
\]
Any minimizer $\nu^\star$ of the right-hand side is the quadratic Wasserstein barycenter of the group-conditional distributions~\cite{xu2023fair}.

This result shows that an independence constraint can convert an $L^2$ projection problem into a Wasserstein barycenter problem: independence enforces a shared output law, while the quadratic projection cost measures the optimal transport cost from each group distribution to that shared law.

\section{Proofs}

\subsection{Proof of Lemma~\ref{lem:joint_cycle_approximation}}\label{a:proof_joint_cycle_approximation}

\begin{proof}
Since \(\mathcal X\) is nonempty, closed, and convex, its Euclidean metric
projection
\[
\Pi_{\mathcal X}:\mathbb R^d\to\mathcal X
\]
is well defined and \(1\)-Lipschitz. Let
\[
E_0:=\{x\in\mathcal X:S(T(x))=x\}.
\]
Then \(\mu_0(E_0)=1\). Fix \(0<\eta<1/3\). By Lusin's theorem, there
exist compact sets \(C_0,C_1\subset\mathcal X\) such that $\mu_0(C_0)>1-\eta$ and $\mu_1(C_1)>1-\eta$ with \(T|_{C_0}\) and \(S|_{C_1}\) are continuous. Since
\(T_\#\mu_0=\mu_1\),
\[
\mu_0(T^{-1}(C_1))=\mu_1(C_1)>1-\eta.
\]
Consequently,
\[
\mu_0\bigl(C_0\cap E_0\cap T^{-1}(C_1)\bigr)>1-2\eta.
\]
By inner regularity, we can choose a compact set $K_0\subset C_0\cap E_0\cap T^{-1}(C_1)$ such that
\[
\mu_0(K_0)>1-3\eta,
\]
and define $K_1:=T(K_0)$. Because \(T|_{K_0}\) is continuous, \(K_1\) is compact. Also, since $K_0 \subset T^{-1}(C_1)$, we have \(K_1\subset C_1\). Moreover,
\[
\mu_1(K_1)
=
\mu_0(T^{-1}(K_1))
\geq
\mu_0(K_0)
>
1-3\eta.
\]
Since \(K_0\subset E_0\), $S\circ T=\operatorname{Id}$ on $K_0$. If \(y\in K_1\), then \(y=T(x)\) for some \(x\in K_0\), and hence
\[
T(S(y))=T(S(T(x)))=T(x)=y.
\]
Thus, $T\circ S=\operatorname{Id}$ on $K_1$. Applying the Tietze extension theorem coordinatewise to obtain continuous maps
\[
A_\eta,B_\eta:\mathcal X\to\mathbb R^d
\]
such that $A_\eta=T$ on $K_0$ and $B_\eta=S$ on $K_1$. Define
\[
\widetilde T_\eta:=\Pi_{\mathcal X}\circ A_\eta,
\qquad
\widetilde S_\eta:=\Pi_{\mathcal X}\circ B_\eta.
\]
Then $\widetilde T_\eta,\widetilde S_\eta
\in C(\mathcal X;\mathcal X)$ and the projection fixes the prescribed values, so $\widetilde T_\eta=T$ on $K_0$ and $\widetilde S_\eta=S$ on $K_1$. By uniform density, for every \(e>0\), there exist $T_{\eta,e}\in\mathcal F$ and $S_{\eta,e}\in\mathcal G$ such that
\[
\|T_{\eta,e}-\widetilde T_\eta\|_\infty<e,
\qquad
\|S_{\eta,e}-\widetilde S_\eta\|_\infty<e.
\]
All these maps are \(\mathcal X\)-valued, so their compositions remain in \(\mathcal X\). For \(H\in C(\mathcal X;\mathbb R^d)\), let
\[
\omega_H(t)
:=
\sup\{\|H(u)-H(v)\|:
u,v\in\mathcal X,\ \|u-v\|\leq t\}.
\]
Compactness of \(\mathcal X\) implies
\(\omega_H(t)\to0\) as \(t\downarrow0\). For \(x\in K_0\),
\[
\begin{aligned}
\|S_{\eta,e}(T_{\eta,e}(x))-x\|
&\leq
\|S_{\eta,e}(T_{\eta,e}(x))
-\widetilde S_\eta(T_{\eta,e}(x))\|\\
&\quad+
\|\widetilde S_\eta(T_{\eta,e}(x))
-\widetilde S_\eta(T(x))\|\\
&\quad+
\|\widetilde S_\eta(T(x))-x\|\\
&\leq
e+\omega_{\widetilde S_\eta}(e) + 0.
\end{aligned}
\]
The final term vanishes because \(T(x)\in K_1\). Similarly, for
\(y\in K_1\), we obtain
\[
\|T_{\eta,e}(S_{\eta,e}(y))-y\|
\leq
e+\omega_{\widetilde T_\eta}(e).
\]

Set $D_{\mathcal X}:=\operatorname{diam}(\mathcal X)$. On the complementary sets with measure at most $3\eta$, all points and compositions lie in \(\mathcal X\), so the corresponding errors are bounded by \(D_{\mathcal X}\). It follows that
\[
\begin{aligned}
\mathcal C(T_{\eta,e},S_{\eta,e})
&\leq
\bigl(e+\omega_{\widetilde S_\eta}(e)\bigr)^2
+
\bigl(e+\omega_{\widetilde T_\eta}(e)\bigr)^2
+6D_{\mathcal X}^2\eta,
\end{aligned}
\]
and
\[
\|T_{\eta,e}-T\|_{L^2(\mu_0)}^2
\leq
e^2+3D_{\mathcal X}^2\eta,
\qquad
\|S_{\eta,e}-S\|_{L^2(\mu_1)}^2
\leq
e^2+3D_{\mathcal X}^2\eta.
\]

Choose \(\eta_j\downarrow0\) with \(0<\eta_j<1/3\). For each \(j\), perform the preceding
construction with \(\eta=\eta_j\), and then choose \(e_j>0\) so that
\[
e_j<\frac1j,
\qquad
e_j+\omega_{\widetilde T_{\eta_j}}(e_j)<\frac1j,
\qquad
e_j+\omega_{\widetilde S_{\eta_j}}(e_j)<\frac1j.
\]
Setting $(T_j,S_j):=(T_{\eta_j,e_j},S_{\eta_j,e_j})$ and using the preceding estimates proves both conclusions. We are done.
\end{proof}

\subsection{Proof of Appendix~\ref{prop:cycle_resolution_invertibility}}\label{a:proof_cycle_resolution_invertibility}

\begin{proof}
Markov's inequality gives
\[
\mathbb P_{X\sim\mu_0}
\bigl(\|g(f(X))-X\|>\varepsilon\bigr)
\leq
\frac{1}{\varepsilon^2}
\int_{\mathcal X}
\|g(f(x))-x\|^2\,d\mu_0(x)
\]
and
\[
\mathbb P_{Y\sim\mu_1}
\bigl(\|f(g(Y))-Y\|>\varepsilon\bigr)
\leq
\frac{1}{\varepsilon^2}
\int_{\mathcal X}
\|f(g(y))-y\|^2\,d\mu_1(y).
\]
Adding these inequalities yields
\[
p_\varepsilon(f,g)
\leq
\frac{\mathcal C(f,g)}{\varepsilon^2}.
\]
The bound by \(2\) follows because \(p_\varepsilon(f,g)\) is the sum of
two probabilities.
\end{proof}

\section{Evaluation Metrics \& Experimental Details}
\label{app:experimental_details}
\subsection{Evaluation metrics}
\label{app:evaluation_metrics}
\subsubsection{Transport cost evaluation}
\paragraph{Transport Cost}
Transport cost in task MNIST, CelebA and Chest X-ray uses average per-sample squared $L_2$ distance in pixel space, computed as the sum of the squared pixel-wise differences between each source image and its transported image, averaged over all samples. 

\[
\mathrm{Transport\ Cost}
=
\frac{1}{N}
\sum_{i=1}^{N}
\left\|T(x_i)-x_i\right\|_2^2
=
\frac{1}{N}
\sum_{i=1}^{N}
\sum_{c=1}^{C}
\sum_{h=1}^{H}
\sum_{w=1}^{W}
\left(T(x_i)_{c,h,w}-x_{i,c,h,w}\right)^2.
\]
\begin{itemize}
    \item $N$ denotes the total number of images used for evaluation;
    \item $C$ denotes the number of image channels;
    \item $H$ denotes the image height;
    \item $W$ denotes the image width.
\end{itemize}

\subsubsection{Marginal matching quality}
\paragraph{Target-Domain Rate (TDR).}
 TDR measures the percentage of transported samples that are classified as belonging to the target domain, with a higher TDR indicating more successful domain transfer. For MNIST, we train a classifier on the full MNIST training set and evaluate it on the corresponding test set, obtaining a baseline classification accuracy of $0.9905$. For CelebA, we use the FairFace classifier \citep{DBLP:journals/corr/abs-1908-04913}, which is trained on a separate face dataset containing 108,501 images and does not use CelebA as its training data. 

\paragraph{Target Probability Score (TPS).}
TPS measures the average probability(confidence) that the classifier assigns to the target-domain classes. 

For the MNIST experiment, the set of target-domain labels \[ \mathcal{Y}_{\mathrm{target}} = \{5,6,7,8,9\}. \] For each transported sample $T(x_i)$, the classifier outputs the class probabilities \[ p_{\theta}\bigl(y=j \mid T(x_i)\bigr), \qquad j \in \{0,\ldots,9\}. \] The probability mass assigned to the target domain for the $i$-th transported sample is \[ q_i = \sum_{j \in \mathcal{Y}_{\mathrm{target}}} p_{\theta}\bigl(y=j \mid T(x_i)\bigr). \] TPS is then defined as \[ \mathrm{TPS} = \frac{1}{N} \sum_{i=1}^{N} \sum_{j \in \mathcal{Y}_{\mathrm{target}}} p_{\theta}\bigl(y=j \mid T(x_i)\bigr), \] where $N$ denotes the total number of transported samples. A higher TPS indicates that the classifier is more confident in assigning the transported samples to the target domain.

\paragraph{Fréchet Inception Distance (FID).}
FID~\cite{heusel2017gans} measures the discrepancy between the distribution of transported samples and that of real target-domain samples in a pretrained feature space. For MNIST, transported samples and real target samples are processed by the feature extractor of a pretrained MNIST classifier, and the distance is computed by: 
\[
\mathrm{FID}_{\mathrm{MNIST}}
=
\left\|
\mu_{\mathrm{map}}-\mu_{\mathrm{target}}
\right\|_2^2
+
\operatorname{Tr}\!\left(
\Sigma_{\mathrm{map}}
+
\Sigma_{\mathrm{target}}
-
2\left(
\Sigma_{\mathrm{map}}
\Sigma_{\mathrm{target}}
\right)^{1/2}
\right).
\]
where $\mu$ and $\Sigma$ denote the empirical mean and covariance matrix, respectively, of the features extracted from the samples.
For CelebA, we uses \texttt{pytorch-fid} implementation \citep{Seitzer2020FID}.

\paragraph{Class-distribution Total Vriation Distance (Class-TV).}
Class-TV measures how closely the class distribution of the generated target samples matches the class distribution of the real target dataset over digits 5-9. Let $p_k$ denote the proportion of generated samples predicted as class $k$, conditioned on the prediction belonging to the target classes $\{5,\ldots,9\}$, and let $q_k$ denote the proportion of class $k$ in the real target test set. The metric is defined as
\[
    \operatorname{TV}(p,q)
    =
    \frac{1}{2}\sum_{k=5}^{9}|p_k-q_k|.
\]
A smaller value indicates better agreement between the generated and
real target-class distributions.

\paragraph{maximum mean discrepancy (MMD).} MMD~\cite{JMLR:v13:gretton12a} measures the discrepancy between the transported distribution and the target distribution. It is computed by Gaussian RBF kernels. $\{x_i\}_{i=1}^{n}$ denote the target samples and
$\{y_j\}_{j=1}^{m}$ denote the transported samples. For a kernel
\[
k_{\gamma}(u,v)
=
\exp\!\left(-\gamma\|u-v\|_2^2\right),
\]
the empirical squared MMD is
\[
\widehat{\operatorname{MMD}}_{\gamma}^{\,2}
=
\frac{1}{n^2}\sum_{i,i'} k_{\gamma}(x_i,x_{i'})
+
\frac{1}{m^2}\sum_{j,j'} k_{\gamma}(y_j,y_{j'})
-
\frac{2}{nm}\sum_{i,j} k_{\gamma}(x_i,y_j).
\]
We evaluate this quantity over $50$ logarithmically spaced kernel
parameters $\gamma\in[10^{-3},10]$ and report their average. A smaller
MMD value indicates closer agreement between the transported and target
distributions.

\paragraph{Enrichment-k100.}
Enrichment-k100 assesses the mixing between the transported and real target samples. For each transported sample $x_i$, we identify its $k$ nearest neighbors $\mathcal{N}_k(x_i)$, in the combined transported and target sample set using Euclidean distance, excluding the query sample itself. The enrichment score is defined as
\[
\operatorname{Enrichment}_k
=
\frac{1}{nk}
\sum_{i=1}^{n}
\sum_{z\in\mathcal{N}_k(x_i)}
\mathbf{1}\{z \text{ is transported}\}.
\]
We use $k=100$. When the transported and target sets have equal sizes,
the random-mixing baseline is approximately $0.5$. Therefore, values
closer to $0.5$ indicate better local mixing, whereas larger values
indicate that transported samples cluster separately from the real
target samples.

\subsubsection{Efficiency}
\paragraph{Training Time}
Training time is measured as the elapsed training time required to reach the checkpoint used for quantitative evaluation. It includes the training loop, data loading, optimizer updates, logging, and checkpointing, but excludes post-training evaluation such as FID, classifier-based metrics, and kNN enrichment. For fairness, all methods were trained using the same batch size and learning rate, with Adam optimizer and learning rate ($2\times 10^{-4}$). MNIST, CelebA, and Cell experiments were run on a single NVIDIA Gerce RTX 5060 Laptop GPU with 8 GB memory, using PyTorch 2.8.0+cu128, CUDA 12.8, and Python 3.9.25. 

\paragraph{Female/Male Image Exposures.}
Female/Male Image Exposures count how many times real female/male images from the CelebA training set are used during training.These metrics provide a standardized measure of sample usage across methods, with fewer exposures indicating greater sample efficiency. In Neural OT, FIE and MIE differ because its multiple inner transport-map updates repeatedly access source female images without requiring additional target male images.

\newpage
{
\section{Additional Experimental Results}

\setlength{\intextsep}{10pt}

\subsection{MNIST}\label{app:mnist_extra}
\begin{figure}[H]
    \centering
    \includegraphics[width=\textwidth]{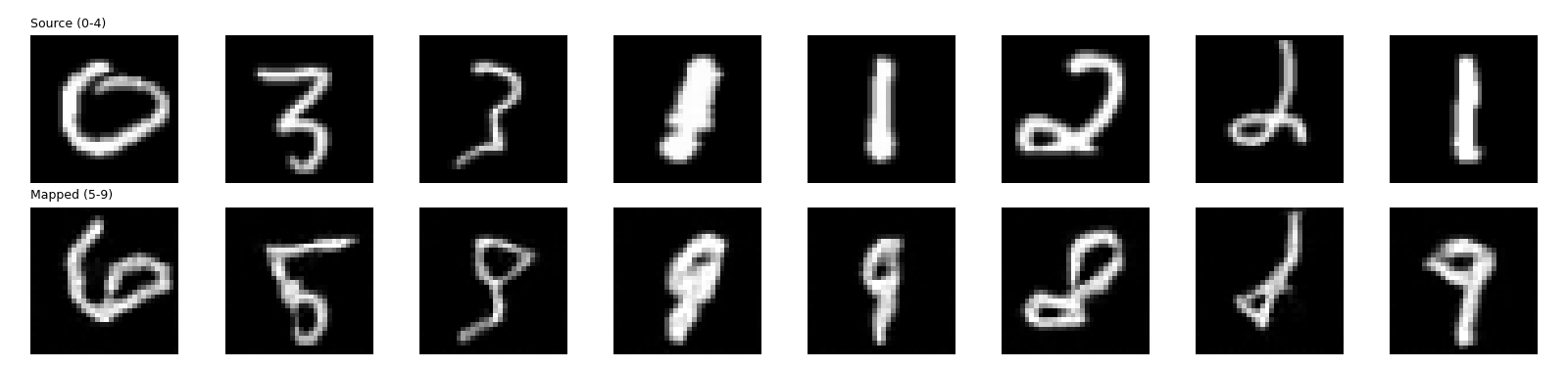}

    \vspace{2mm}

    \includegraphics[width=\textwidth]{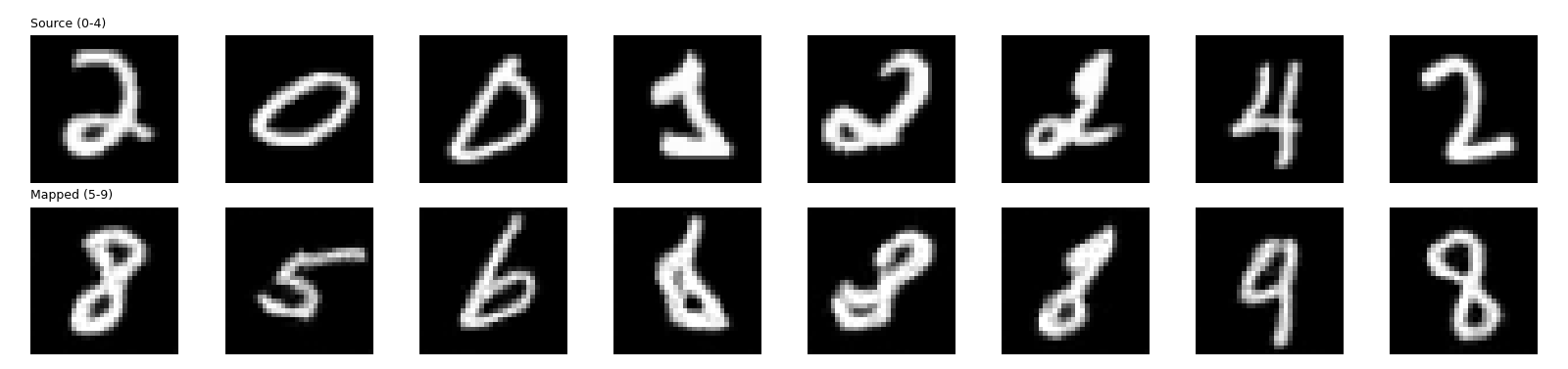}

    \vspace{2mm}

    \includegraphics[width=\textwidth]{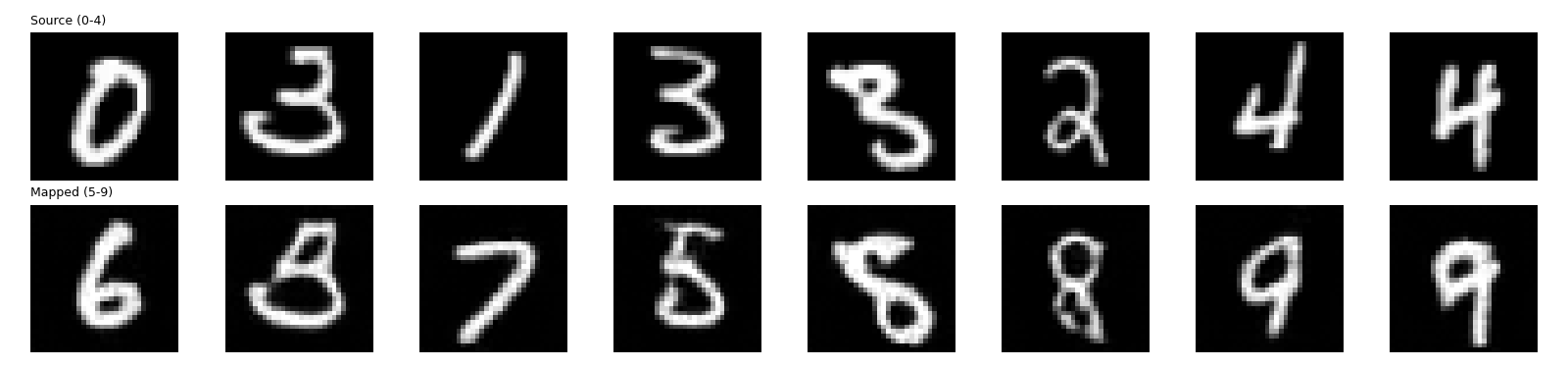}

    \vspace{2mm}

    \includegraphics[width=\textwidth]{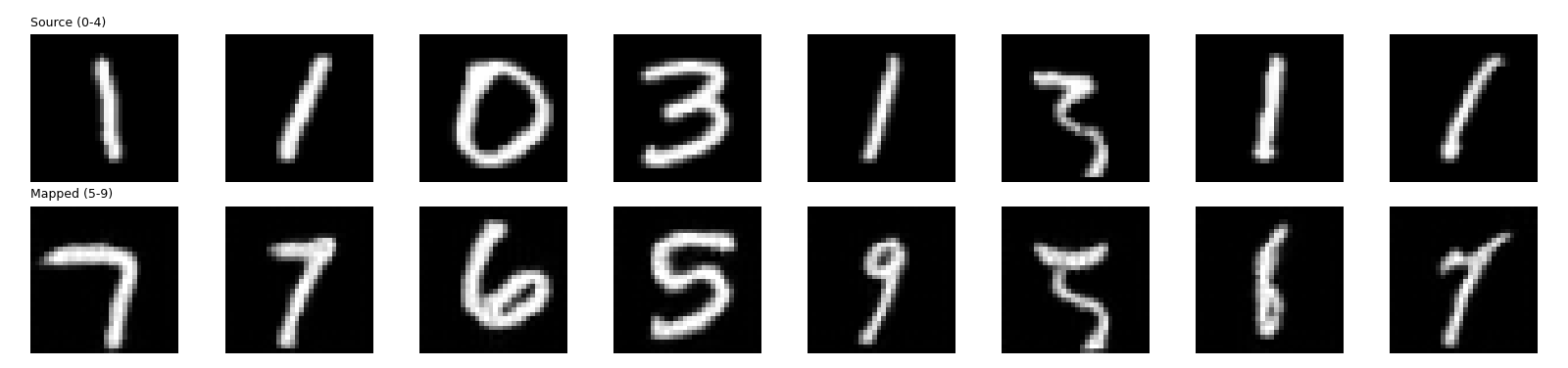}

    \caption{Additional qualitative results on the MNIST Digit Transfer (from digit 0-4 to digit 5-9) translation task.}
    \label{fig:additional_celeba_results}
\end{figure}

\subsection{CelebA}
\label{app:celeba_extra}
\begin{figure}[H]
    \centering
    \includegraphics[width=\textwidth]{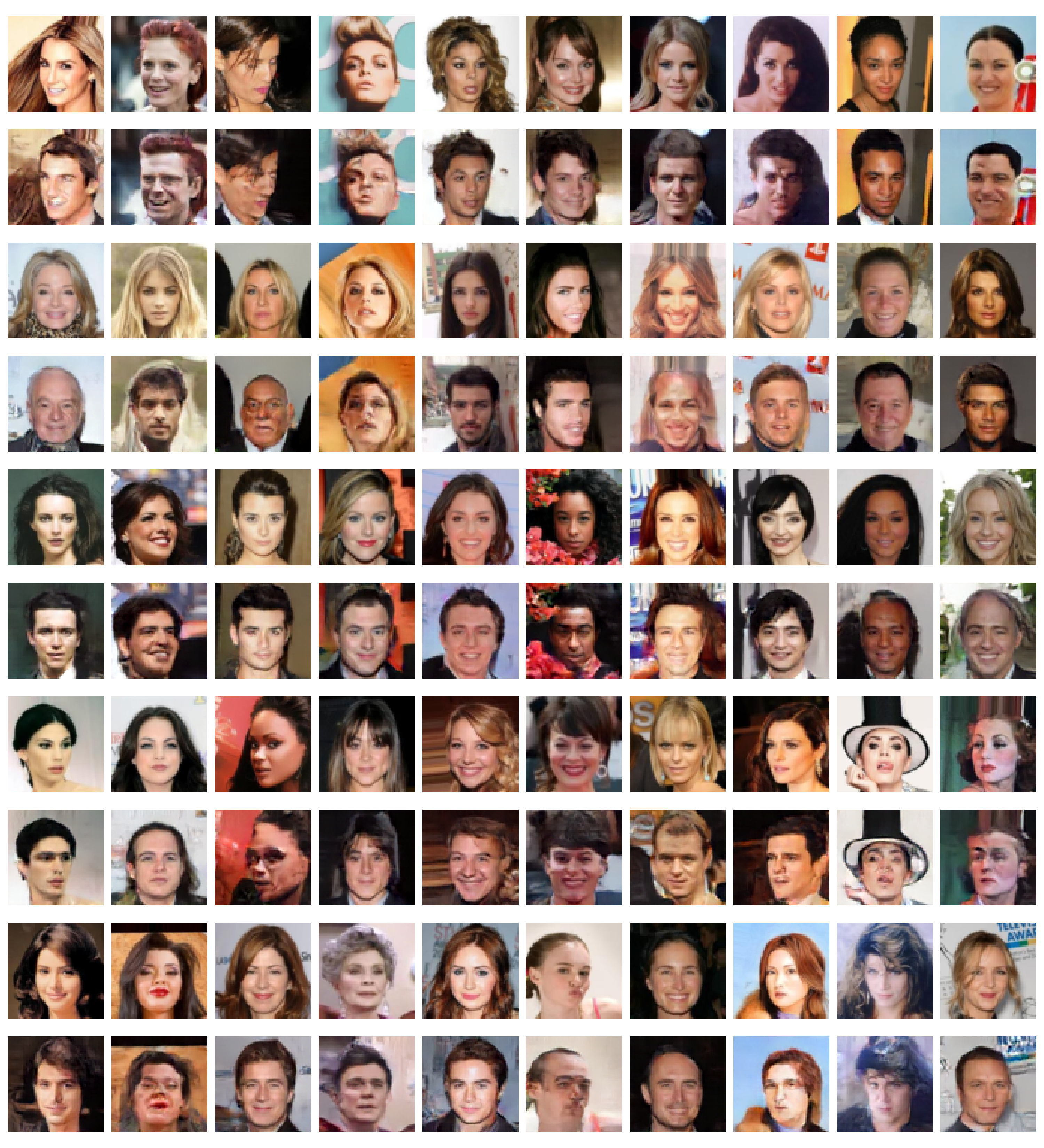}
    \caption{Qualitative results using the first 50 female images from the CelebA dataset as fixed source samples. Odd-numbered rows show the source female images, and the corresponding even-numbered rows show the mapped male images generated by CyclOT. The full image grid is displayed without post-hoc cropping.}
    \label{fig:celeba_gender_more}
\end{figure}

\begin{figure}[H]
    \centering
    \includegraphics[width=\textwidth]{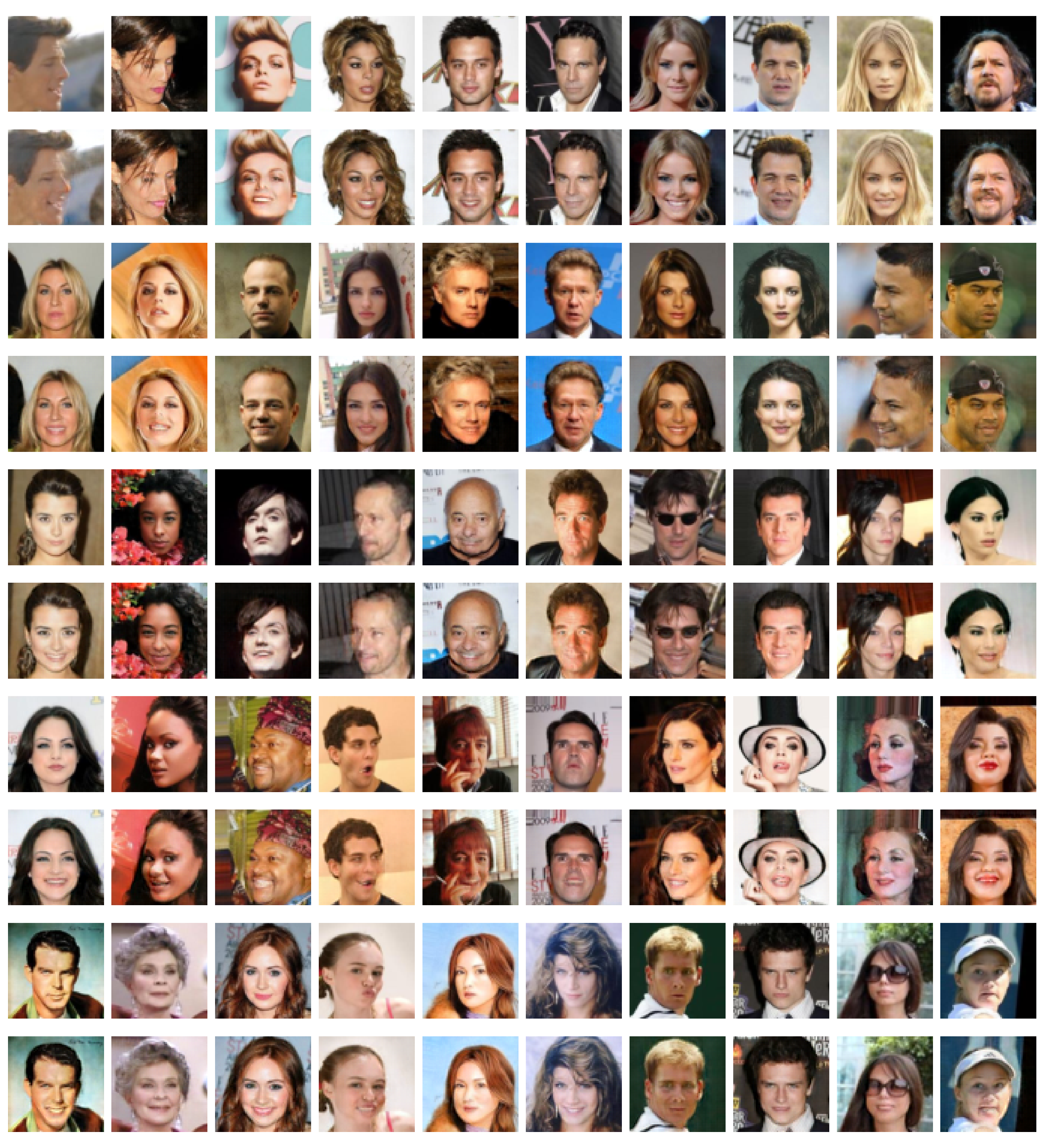}
    \caption{Qualitative results using the first 50 non-smiling images from the CelebA dataset as fixed source samples. Odd-numbered rows show the source non-smiling images, and the corresponding even-numbered rows show the mapped smiling images generated by CyclOT. The full image grid is displayed without post-hoc cropping.}
    \label{fig:celeba_smiling}
\end{figure}
}

\subsection{Pneumonia} \label{app:pneumonia}

Quantitative results in \ref{subsec:cxr} only included the transport from healthy to unhealthy images. Table \ref{tab:chest_quantitative_full_comparison} includes results in both directions, where \textit{Pos $\to$ Neg} indicates transporting from healthy to unhealthy, and \textit{Neg $\to$ Pos} indicates transporting from unhealthy to healthy. Baseline predictions for true images in each class can be found in Table \ref{tab:pneumonia_baseline}.

\begin{table}[H]
    \centering
    \small
    \setlength{\tabcolsep}{4pt}
    \begin{tabular}{llccc}
    \toprule
        Method
        & Direction
        & Cost $\downarrow$
        & $|\mathrm{Enrich\mbox{-}k100}-0.5| \downarrow$
        & TPS $\uparrow$ \\
        \midrule
        \multirow{2}{*}{ICNN OT}
        & Pos $\to$ Neg & 73142.1836 & 0.4986 & 0.1517 $\pm$ 0.0260 \\
        & Neg $\to$ Pos & 65095.5671 & 0.4984 & 0.1513 $\pm$ 0.0409 \\
        \midrule
        \multirow{2}{*}{Neural OT}
        & Pos $\to$ Neg & 1623.5243  & 0.0507 & 0.4180 $\pm$ 0.1811 \\
        & Neg $\to$ Pos & 1984.1612  & 0.3246 & 0.4762 $\pm$ 0.1860 \\
        \midrule
        \multirow{2}{*}{Ours}
        & Pos $\to$ Neg & 1472.9777 & 0.0133 & 0.4484 $\pm$ 0.1856\\
        & Neg $\to$ Pos & 1717.2522 & 0.3169 & 0.4098 $\pm$ 0.1873 \\
        \bottomrule
    \end{tabular}
    \caption{Quantitative comparison between ICNN OT, Neural OT, and our method on the chest x-ray healthy to pneumonia transition task, broken down by transport direction.}
    \label{tab:chest_quantitative_full_comparison}
\end{table}

\subsection{Ablation on cycle-consistency loss} \label{app:cycle_ablation}

\begin{figure}[H]
    \centering
    \includegraphics[width=\textwidth]{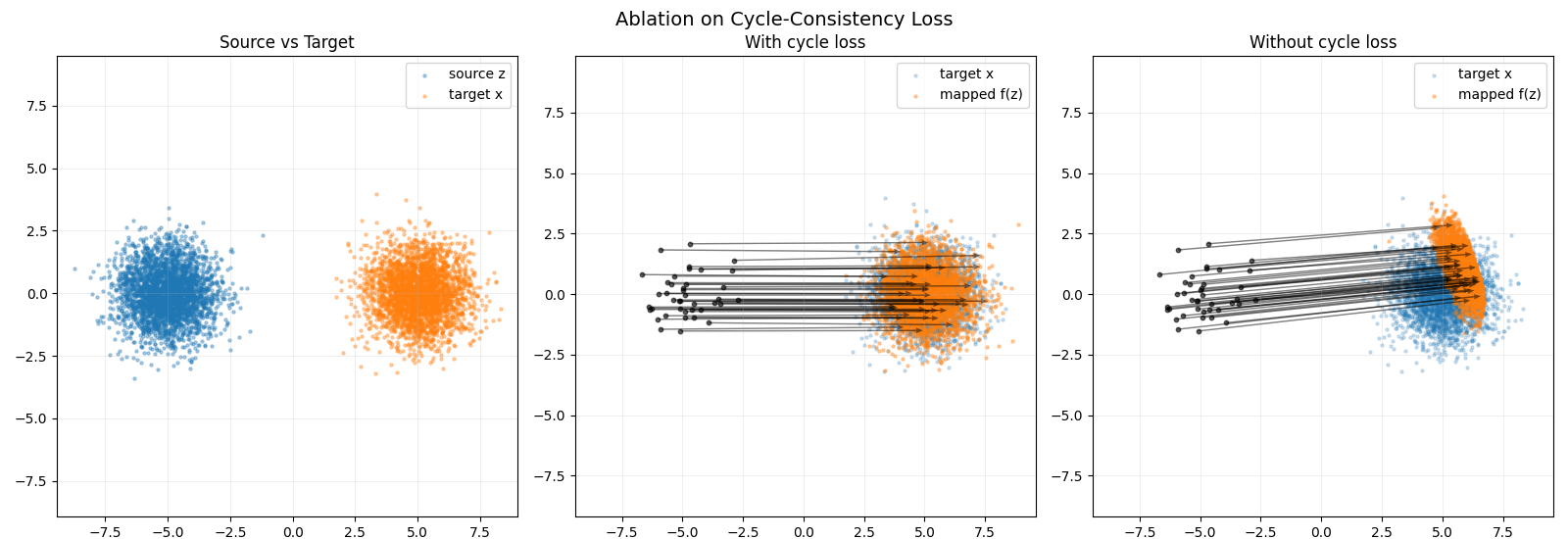}
    \caption{Ablation study on the cycle-consistency loss. Removing the cycle loss leads to mode collapse transport (right) compared to the full model (middle).}
    \label{fig:cycle_ablation}
\end{figure}

\end{document}